\documentclass[conference]{IEEEtran}
\usepackage[utf8]{inputenc}
\usepackage{multirow}
\usepackage{graphicx}
\usepackage{color,soul}
\usepackage{amsmath,amssymb,amsthm}
\usepackage{subfigure}
\usepackage[maxfloats=100]{morefloats}
\usepackage{algorithm}
\usepackage{algorithmic}

\newtheorem{theorem}{Theorem}
\newtheorem{corollary}[theorem]{Corollary}
\newtheorem{lemma}[theorem]{Lemma}

\newtheorem{assumption}{Assumption}
\newtheorem{definition}{Definition}

\newtheorem{remark}{Remark}

\usepackage{ifthen}
\newboolean{showcomments}
\setboolean{showcomments}{true}
\newcommand{\josh}[1]{\ifthenelse{\boolean{showcomments}}
	{ \textcolor{red}{(Josh says:  #1)}}{}}
\newcommand{\zhenhua}[1]{\ifthenelse{\boolean{showcomments}}
	{ \textcolor{red}{(Zhenhua says:  #1)}}{}}
\newcommand{\andrey}[1]{\ifthenelse{\boolean{showcomments}}
	{ \textcolor{red}{(Andrey says:  #1)}}{}}
\newcommand{\outline}[1]{\ifthenelse{\boolean{showcomments}}
	{ \textcolor{blue}{Outline:  #1}}{}}
\newcommand{\addcite}[0]{\ifthenelse{\boolean{showcomments}}
	{ \textcolor{green}{(add citation(s))}}{}}
\newcommand{\addref}[0]{\ifthenelse{\boolean{showcomments}}
	{ \textcolor{green}{(add ref)}}{}}
\newcommand{\todo}[1]{\ifthenelse{\boolean{showcomments}}
	{ \textcolor{red}{(To do:  #1)}}{}}

\newboolean{showfixes}
\setboolean{showfixes}{true}
\newcommand{\fixes}[1]{\ifthenelse{\boolean{showfixes}}
	{\textcolor{black}{#1}}{}}

\ifCLASSINFOpdf
\else
\fi
\begin{document}
%
\title{Sample Complexity of Power System State Estimation using Matrix Completion}

\author{\IEEEauthorblockN{Joshua Comden\IEEEauthorrefmark{1},
Marcello Colombino\IEEEauthorrefmark{2},
Andrey Bernstein\IEEEauthorrefmark{3}, and
Zhenhua Liu\IEEEauthorrefmark{1}}
\IEEEauthorblockA{\IEEEauthorrefmark{1}Stony Brook University, \{joshua.comden, zhenhua.liu\}@stonybrook.edu}
\IEEEauthorblockA{\IEEEauthorrefmark{2}McGill University, marcello.colombino@mcgill.ca}
\IEEEauthorblockA{\IEEEauthorrefmark{3}National Renewable Energy Laboratory, andrey.bernstein@nrel.gov}}


\maketitle
\thispagestyle{plain}
\pagestyle{plain}
\begin{abstract}
	In this paper, we propose an analytical framework to quantify the amount of data samples needed to obtain accurate state estimation in a power system — a problem known as \emph{sample complexity} analysis in computer science. Motivated by the increasing adoption of distributed energy resources into the distribution-level grids, it becomes imperative to estimate the state of distribution grids in order to ensure stable operation. Traditional power system state estimation techniques mainly focus on the transmission network which involve solving an overdetermined system and eliminating bad data. However, distribution networks are typically \emph{underdetermined} due to the large number of connection points and high cost of pervasive installation of measurement devices. In this paper, we consider the recently proposed state-estimation method for underdetermined systems that is based on \emph{matrix completion}. In particular, a \emph{constrained} matrix completion algorithm was proposed, wherein the standard matrix completion problem is augmented with additional equality constraints representing the physics (namely power-flow constraints). We analyze the sample complexity of this general method by proving an upper bound on the sample complexity that depends directly on the properties of these constraints that can lower number of needed samples as compared to the unconstrained problem. To demonstrate the improvement that the constraints add to distribution state estimation, we test the method on a 141-bus distribution network case study and compare it to the traditional least squares minimization state estimation method.

\end{abstract}


%

\section{Introduction}
\begin{figure*}
    \centering
    \subfigure[18 Bus \cite{grady1992application}]{\includegraphics[width=0.65\columnwidth]{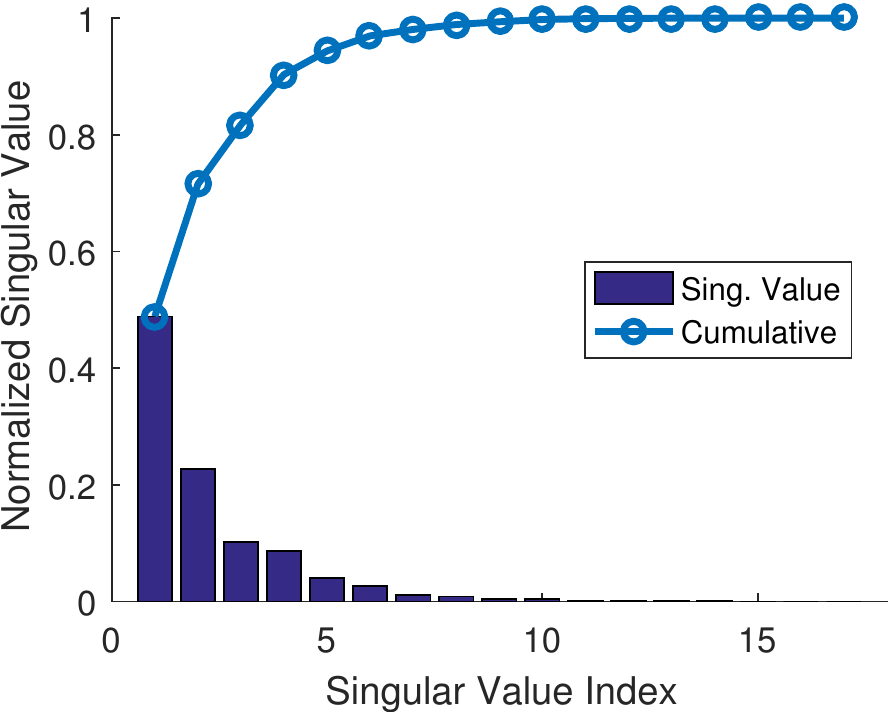}}
    \subfigure[22 Bus \cite{raju2012direct}]{\includegraphics[width=0.65\columnwidth]{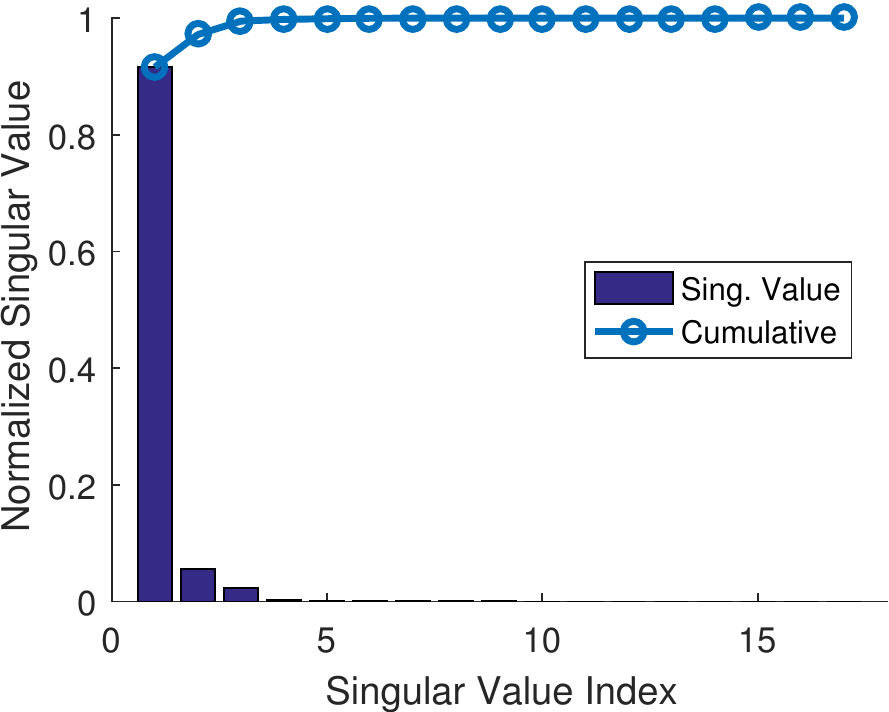}}
    \subfigure[33 Bus \cite{baran1989network}]{\includegraphics[width=0.65\columnwidth]{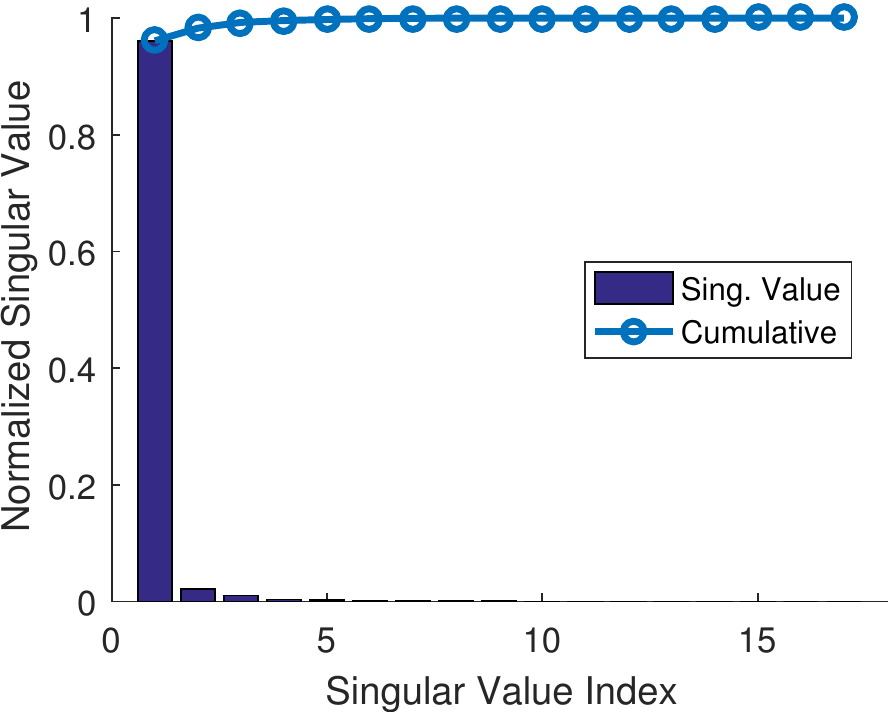}}
    \subfigure[69 Bus \cite{das2008optimal}]{\includegraphics[width=0.65\columnwidth]{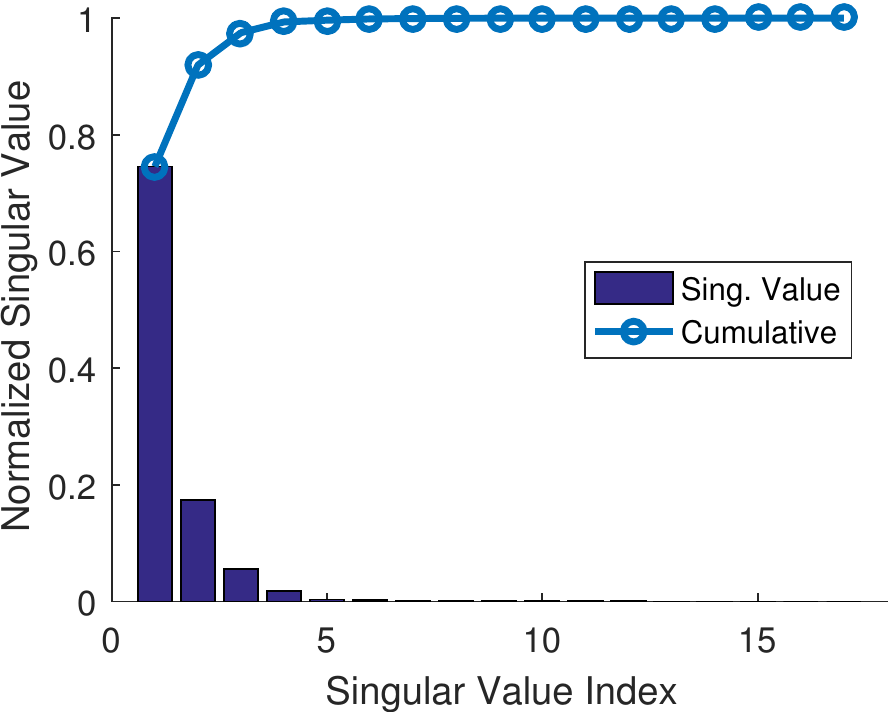}}
    \subfigure[85 Bus \cite{das1995simple}]{\includegraphics[width=0.65\columnwidth]{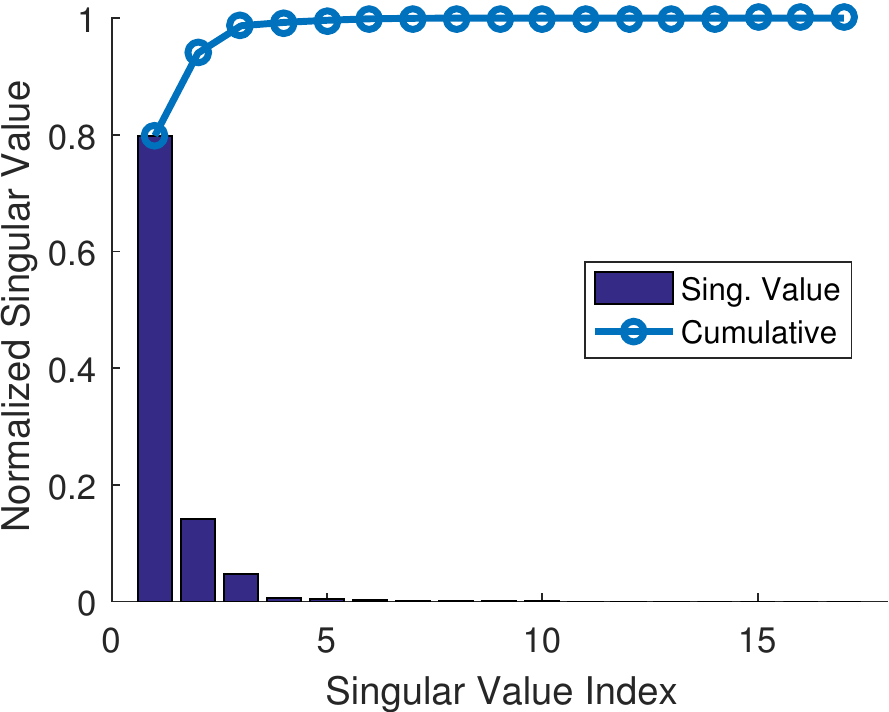}}
    \subfigure[141 Bus \cite{khodr2008maximum}]{\includegraphics[width=0.65\columnwidth]{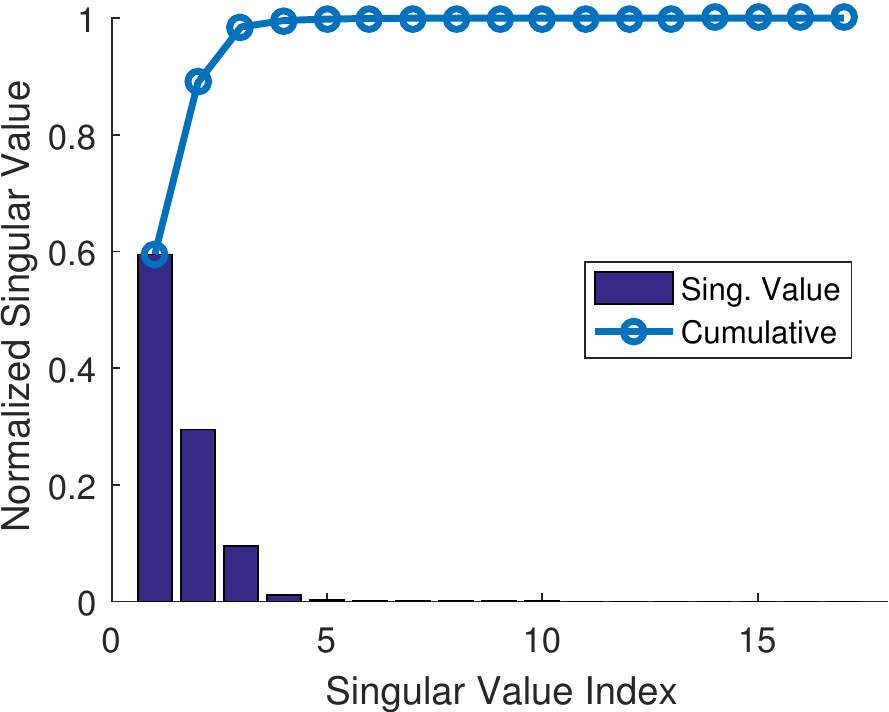}}
    \caption{Singular values of matrices that represent the states of six different radial distribution test cases.
    The bars are the individual values while the circles are the cumulative values.
    All singular values are normalized by their sum.}
    \label{fig:cases_scree}
\end{figure*}


State estimation is one of the fundamental data analysis tasks in power systems. In its classical form, it amounts to estimating voltage phasors at all the buses of the network given some data gathered from the network~\cite{liacco1982role}. It has a long and established history in transmission networks, where classical approaches based on weighted least-squares methods are applicable due to full observability of the network \cite{Abur2004}. The latter conditions roughly speaking mean that the underlying system of equations for the estimation problem is \emph{overdetermined}, i.e., it has more observables (and, hence, equations) than unknown variables~\cite{meliopoulos2001power}. In traditional distribution networks, however, state estimation is typically not used, or used very rarely~\cite{dehghanpour2018survey}. There are two main reasons for that. First, unlike in transmission networks, there is a lack of pervasive installation of measurement devices such as phasor measurement units (PMUs)~\cite{sexauer2013phasor,bolognani2014state}. Hence, the estimation problem is \emph{underdetermined} and so classical, simple approaches (e.g. weighted least-squares) cannot be applied since they require full observability~\cite{baran2001}. Second, since (traditional) distribution networks are mostly passive and overprovisioned~\cite{della2014electrical,ochoa2010distribution}, there has not been much motivation to develop state estimation algorithms apart from simple heuristics (e.g., based on simple load-allocation rules~\cite{deng2002branch,pereira2004integrated}).

However, recently, distribution networks have undergone a radical change due to massive penetration of distributed energy resources (DERs) at the edge of the network~\cite{driesen2008design,kulmala2014coordinated,syahputra2014optimal,song2013operation}. This creates both challenges and opportunities. On the challenges side, DERs (and especially renewable energy resources such as photovoltaic panels and wind farms) introduce a lot of uncertainty into the system~\cite{katiraei2011solar,inman2013solar,hill2012battery,energy2010western}. Hence, traditional approaches for operating the network are not applicable anymore, and there is a need in active and accurate inference for decision making. In particular, accurate real-time state estimation is needed to ensure stable and safe operation of the network~\cite{della2014electrical,huang2012state}. On the opportunities side, the vast deployment of DERs introduces both control and measurement points that now allow the application of modern machine learning and data analytics methods to deal with problems such as state estimation~\cite{alarcon2010multi,basak2012literature}. However, observability is still an issue: the corresponding estimation problem is typically underdetermined, and hence use of psuedo-measurements, structured estimation methods, and methods that use historical data to complete the missing information, have been proposed \cite{baran95,singh2009,manitsas2012,wu2013,bhela2018,donti2019matrix}.

In this paper, we consider the recently proposed method for state estimation in underdetermined systems using \emph{low-rank matrix completion} \cite{donti2019matrix}. The method is based on augmenting the standard matrix completion approach \cite{candes2009exact} with power-flow constraints which provide an additional link between parameter values. As shown in \cite{donti2019matrix} numerically with extensive simulations, this structured (or \emph{physics-based}) approach performs very well in distribution networks under realistic low-observability scenarios. In the present paper, we set our goal to study the sample complexity of this approach. 

Sample complexity in power-system state estimation is largely unexplored. 
Roughly speaking, sample complexity is the amount of data samples needed to obtain accurate estimation of the true state. 
Even in the case of the classical weighted least-squares methods, the literature is scarce whereas there is active research in computer science and machine learning community on the topic \cite{simchowitz18a,rantzer2018}.
To that end, we extend the sample complexity analysis of the standard (unconstrained) matrix completion \cite{candes2009exact,candes2010matrix} to the general constrained case to obtain better theoretical bounds.
The main theoretical challenge is on how to measure the information from the added constraints in terms of the amount of sampled state variables, which can be used to partially replace the need for a specific number of measurements.
The theoretical results are further verified by the significant reduction in sample sizes illustrated by numerical evaluations using a distribution network test case. 

The main contribution of this paper is two-fold. From the theoretical perspective, we derive new results on the sample complexity of the general matrix completion problem for the constrained case with equality constraints that can lower number of needed samples as compared to the unconstrained problem. On the practical side, we demonstrate that by incorporating additional physical constraints, the sample sizes are greatly reduced. This is crucial for state estimation in the distribution networks which are often underdetermined, and therefore pave the road for more sophisticated control and optimization based on the state estimation. 
To the best of our knowledge, these are the first results in the literature on sample complexity for constrained matrix completion in general, and for state estimation on power systems in particular.

\section{Motivational Example}
\label{sec:motiv_examp}
The objective of this motivational example is to give evidence of why using low-rank matrix completion techniques make appropriate approximations for the state of a distribution system.
Singular Value Decomposition (SVD) is used in Principal Component Analysis (PCA) to extract lower dimensional feature vectors from high dimensional data sets.
This works under the assumption that most of the useful information within a data set can be captured by a highly compressible matrix, i.e. a lower rank matrix, so that more intensive computational analysis can be performed with minimal loss in accuracy~\cite{abdi2010principal}.
In the case of low-rank matrix completion for state-estimation, we will use the structure from this lower rank approximation to fill in unmeasured state variables.

SVD decomposes any matrix into a linear combination of rank-1 matrices.
Specifically if a matrix has rank $r$, then SVD will decompose it into $r$ rank-1 matrices.
The scalars which multiply the $r$ rank-1 matrices in the linear combination are called singular values.
In either PCA or low-rank matrix completion, it is important to first show that most of information can in fact be obtained by $r$ rank-1 matrices as measured by the sum of the largest $r$ singular values.
In the rest of this section, we will show that this to be true for distribution network states.

Six different radial distribution network test cases were simulated on MATPOWER~\cite{zimmerman2011matpower} using their default settings which are based on real-world systems\cite{grady1992application,raju2012direct,baran1989network,das2008optimal,das1995simple,khodr2008maximum}.
In addition to the six test cases, we used a different IEEE 123 bus distribution network~\cite{kersting1991radial} which had its 3-phase voltages and power injections simulated over an entire week with one-minute time granularity using OpenDSS.
For each network, its state was organized into a matrix with the rows corresponding to buses and lines, and the columns corresponding to the state variable type.
See Section \ref{sec:state_est_prob} for more details on the matrix structure.

For each distribution network state matrix, SVD was performed, and the corresponding singular values are  plotted in Figure \ref{fig:cases_scree} in descending order.
A normalized singular value represents the fraction of the state matrix represented by its associated rank-1 matrix.
Therefore, the $k$th cumulative normalized value gives the fraction of the state matrix represented by the first $k$ rank-1 matrices and their associated singular values.
In Figure \ref{fig:cases_scree}, the high cumulative values at a low singular value index means that most of the information in the state matrix can be represented by the first few rank-1 matrices.
Specifically in all cases, more than 95\% of the state information can be recovered by just 5 rank-1 matrices.
This was also found to be true for the 123 bus network which is shown in Figure \ref{fig:ieee123_week_scree} as the average for all the states over the measured week.
Error bars were added to the figure to show the minimum and maximum normalized singular values.
Because the distance between the minimum and maximum values were was only 0.0045, they are almost indistinguishable from the average value and indicates that the singular values change very little even when the state changes.
For this reason, low-rank matrix completion techniques can be useful in estimating the state of a distribution network from only measuring a small of the connection points.

\begin{figure}
    \centering
    \includegraphics[width=0.85\columnwidth]{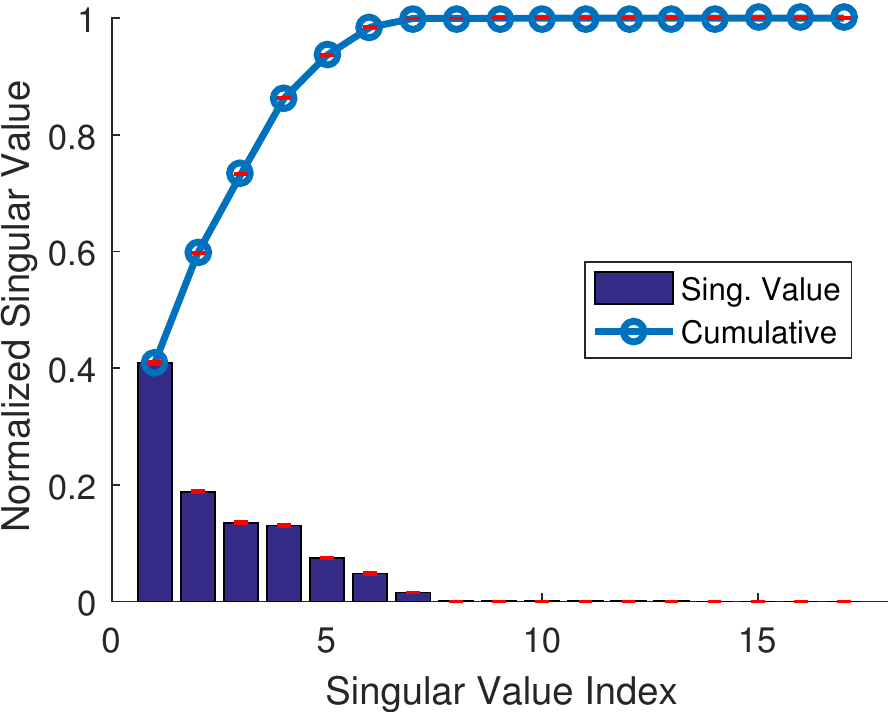}
    \caption{Singular values of matrices that represent the states of the IEEE 123 bus distribution network~\cite{kersting1991radial} over a week.
    The bars are the individual values while the circles are the cumulative values.
    All singular values are normalized by their sum.  The red error bars show the minimum and maximum values.}
    \label{fig:ieee123_week_scree}
\end{figure}

\section{Problem Formulation}
\label{sec:prob}

\paragraph{Notation}
A column vector $\mathbf{x}$ is represented by a bold lowercase letter and a matrix $\mathbf{X}$ is represented by bold uppercase letter, while a scalar $x$ or an entry $X_{ij}$ are not bold and can be either upper or lower case.
For complex number $x$, let $\text{Re}(x)$, $\text{Im}(x)$, and $|x|$ be its real component, its imaginary component, and its magnitude respectively. 
The $k$th matrix $\mathbf{X}^{(k)}$ in a sequence may be labeled by a superscript in parenthesis.
$\mathbf{I}_n$ is the $n\times n$ identity matrix.
A calligraphic letter $\mathcal{X}$ can be a set, vector space, or operator which will be distinctly made clear in context.
Specifically, $\mathcal{P}_\mathcal{X}$ is the orthogonal projection onto vector space $\mathcal{X}$.
The perpendicular vector space to $\mathcal{X}$ is $\mathcal{X}^\perp$.
The transpose of matrix $\mathbf{X}$ is $\mathbf{X}^\intercal$.
The $\ell_2$-norm of vector $\mathbf{x}$ is $\|\mathbf{x}\|$.
The Euclidean inner product of matrices $\mathbf{A}\in\mathbb{R}^{n_1\times n_2}$ and $\mathbf{X}\in\mathbb{R}^{n_1\times n_2}$ is $\langle \mathbf{A},\mathbf{X} \rangle:=\text{trace}(\mathbf{A}^\intercal\mathbf{X})$.
The Frobenius norm of matrix $\mathbf{X}$ is $\|\mathbf{X}\|_F:=\sqrt{\langle\mathbf{X},\mathbf{X}\rangle}=\sqrt{\sum_i\sum_j|X_{ij}|^2}$.
The nuclear norm of matrix $\mathbf{X}$ is denoted by $\|\mathbf{X}\|_*$ and is the sum of the its singular values, while the spectral norm is denoted $\|\mathbf{X}\|$ and is the value of its largest singular value.
\fixes{The Schatten b-norm of $\mathbf{X}$ is $\|\mathbf{X}\|_{S_b}:=\left(\sum_{k=1}^{n_2} (\sigma_k(\mathbf{X}))^b\right)^{\frac{1}{b}}$ where $\sigma_k(\mathbf{X})$ is its $k$-th largest singular value.}
The norm of the operator $\mathcal{R}$ is its spectral norm denoted as $\|\mathcal{R}\|:=\sup_{\mathbf{X}:\|\mathbf{X}\|_F\leq 1}\|\mathcal{R}(\mathbf{X})\|_F$.

\subsection{Power System Model}

Consider a power network with $n_\text{b}$ PQ buses in the set $\mathcal{N}$ and $n_\text{l}$ lines in the set $\mathcal{L}\subset\mathcal{N}\times\mathcal{N}$.
For each line $(s,t)\in\mathcal{L}$, bus $s$ is denoted as the ``From" bus and bus $t$ is denoted as the ``To" bus.
Typically in a radial distribution network, the slack or feeder bus is labeled as bus 1 and all other buses are labeled sequentially outward so that when the lines are directed away from the feeder, the From bus has a smaller index than the To bus.

Complex power is split into its real and reactive components represented by $P$ and $Q$ respectively.
Power flows across lines are treated as injections into the line from both the From and To sides so that their sum equals the power Loss:
\begin{subequations}
    \label{eq:syseqn_line_pwrcons}
    \begin{align}
        P^\text{From}_{s,t} + P^\text{To}_{s,t} & = P^\text{Loss}_{s,t} & \quad \forall (s,t)\in\mathcal{L} \\
        Q^\text{From}_{s,t} + Q^\text{To}_{s,t} & = Q^\text{Loss}_{s,t} & \quad \forall (s,t)\in\mathcal{L}.
    \end{align}
\end{subequations}
Therefore from the conservation of power at each bus, its power injection into the bus must equal the power injections into the lines it is connected to:
\begin{subequations}
    \label{eq:syseqn_bus_pwrcons}
    \begin{align}
        P_s & = \sum_{t:(s,t)\in\mathcal{L}}P^\text{From}_{s,t} + \sum_{t:(t,s)\in\mathcal{L}}P^\text{To}_{t,s} & \quad \forall s\in\mathcal{N} \\
        Q_s & = \sum_{t:(s,t)\in\mathcal{L}}Q^\text{From}_{s,t} + \sum_{t:(t,s)\in\mathcal{L}}Q^\text{To}_{t,s} & \quad \forall s\in\mathcal{N}.
    \end{align}
\end{subequations}

The complex current injection $I_s$ at each bus $s$ and the complex current flow $I_{s,t}$ across each line $(s,t)$ follow Kirchhoff's Current Law:
\begin{subequations}
    \label{eq:syseqn_bus_kirchoffs_law}
    \begin{align}
        \sum_{t:(s,t)\in\mathcal{L}}\text{Re}\left(I_{s,t}\right) & = \text{Re}\left(I_s\right) + \sum_{t:(t,s)\in\mathcal{L}}\text{Re}\left(I_{t,s}\right) & \forall s\in\mathcal{N} \\
        \sum_{t:(s,t)\in\mathcal{L}}\text{Im}\left(I_{s,t}\right) & = \text{Im}\left(I_s\right) + \sum_{t:(t,s)\in\mathcal{L}}\text{Im}\left(I_{t,s}\right) & \forall s\in\mathcal{N}.
    \end{align}
\end{subequations}
Additionally, using the complex voltage $V_s$ at each bus, Ohm's Law relates the voltage difference between the two sides of a line to its current flow:
\begin{subequations}
    \label{eq:syseqn_ohms_law}
    \begin{align}
        \text{Re}\left(I_{s,t}\right) & = G_{s,t}(\text{Re}(V_s)-\text{Re}(V_t)) - B_{s,t}(\text{Im}(V_s)-\text{Im}(V_t)) \nonumber \\
        & \quad\quad\quad\quad\quad\quad\quad\quad\quad\quad\quad\quad\quad\quad\quad\quad \forall (s,t)\in\mathcal{L} \\
        \text{Im}\left(I_{s,t}\right) & = B_{s,t}(\text{Re}(V_s)-\text{Re}(V_t)) + G_{s,t}(\text{Im}(V_s)-\text{Im}(V_t)) \nonumber \\
        & \quad\quad\quad\quad\quad\quad\quad\quad\quad\quad\quad\quad\quad\quad\quad\quad \forall (s,t)\in\mathcal{L}
    \end{align}
\end{subequations}
where $G_{s,t}$ and $B_{s,t}$ are the conductance and susceptance of line $(s,t)$.
The power injections into each line from either side are determined from its current flow and voltage on that side:
\begin{subequations}
    \label{eq:syseqn_PVR}
    \begin{align}
        P^\text{From}_{s,t} & = \text{Re}(V_s)\text{Re}(I_{s,t})+\text{Im}(V_s)\text{Im}(I_{s,t})  & \quad \forall (s,t)\in\mathcal{L} \\
        Q^\text{From}_{s,t} & = \text{Im}(V_s)\text{Re}(I_{s,t})-\text{Re}(V_s)\text{Im}(I_{s,t}) & \quad \forall (s,t)\in\mathcal{L} \\
        P^\text{To}_{s,t} & = -(\text{Re}(V_t)\text{Re}(I_{s,t})+\text{Im}(V_t)\text{Im}(I_{s,t})) & \quad \forall (s,t)\in\mathcal{L} \\
        Q^\text{To}_{s,t} & = -(\text{Im}(V_t)\text{Re}(I_{s,t})-\text{Re}(V_t)\text{Im}(I_{s,t})) & \quad \forall (s,t)\in\mathcal{L}
    \end{align}
\end{subequations}
The power loss across each line are determined from its magnitude of the current flow $|I_{s,t}|$:
\begin{subequations}
    \label{eq:syseqn_loss}
    \begin{align}
        P^\text{Loss}_{s,t} & = R_{s,t}|I_{s,t}|^2 & \quad \forall (s,t)\in\mathcal{L} \\
        Q^\text{Loss}_{s,t} & = X_{s,t}|I_{s,t}|^2 & \quad \forall (s,t)\in\mathcal{L}
    \end{align}
\end{subequations}
where $R_{s,t}$ and $X_{s,t}$ are the resistance and reactance of line $(s,t)$.

Trivially, we also have the magnitudes of the complex voltages and currents derived from their real and imaginary parts:
\begin{subequations}
    \label{eq:syseqn_mag}
    \begin{align}
        |V_s| & = \sqrt{\text{Re}(V_s)^2+\text{Im}(V_s)^2} & \quad \forall s\in\mathcal{N} \\
        |I_s| & = \sqrt{\text{Re}(I_s)^2+\text{Im}(I_s)^2} & \quad \forall s\in\mathcal{N} \\
        |I_{s,t}| & = \sqrt{\text{Re}(I_{s,t})^2+\text{Im}(I_{s,t})^2} & \quad  \forall (s,t)\in\mathcal{L}.
    \end{align}
\end{subequations}

\subsection{State Estimation Problem}
\label{sec:state_est_prob}

We represent the state of the power system in a block matrix $\mathbf{M}$ where one matrix $\mathbf{M}_\text{b}$ holds the state of the buses and the other matrix $\mathbf{M}_\text{l}$ holds the state of the lines
\begin{align}
    \mathbf{M} := \left[
        \begin{array}{c|c}
            \mathbf{M}_\text{b} & \mathbf{0} \\
            \hline
            \mathbf{0} & \mathbf{M}_\text{l}
        \end{array}\right]. \nonumber
\end{align}
The state of the buses $\mathbf{M}_\text{b}$ is in an $n_\text{b}\times 8$ matrix which holds the following values in the row associated with bus $s\in\mathcal{N}$:
\begin{align}
    \left(P_s,Q_s,\text{Re}(V_s),\text{Im}(V_s),|V_s|,\text{Re}(I_s),\text{Im}(I_s),|I_s|\right) \nonumber
\end{align}
while the state of the lines $\mathbf{M}_\text{l}$ is in an $n_\text{l}\times 9$ matrix which holds the following values in the row associated with line $(s,t)\in\mathcal{L}$:
\begin{align}
    \big(P^\text{From}_{s,t},Q^\text{From}_{s,t},P^\text{To}_{s,t},Q^\text{To}_{s,t},P^\text{Loss}_{s,t},Q^\text{Loss}_{s,t}, \quad\quad\quad\quad\quad\quad \nonumber \\
    \text{Re}\left(I_{s,t}\right),\text{Im}\left(I_{s,t}\right),\left|I_{s,t}\right|\big). \nonumber
\end{align}

Classically, the state is represented in a more compact form of only the complex bus voltages since, given voltages, all other variables can be computed using   \eqref{eq:syseqn_line_pwrcons}-\eqref{eq:syseqn_mag}. However, complex voltages can only be measured by PMUs which are expensive and are not available at almost all of the buses or lines in a distribution network.
On the other hand, measurements of some of the other  variables are more widely available such as $\left(P_s,Q_s,|V_s|,|I_s|\right)$ for a bus and $\left(P^\text{From}_{s,t},Q^\text{From}_{s,t},P^\text{To}_{s,t},Q^\text{To}_{s,t},\left|I_{s,t}\right|\right)$ for a line.
Let $\Omega$ be the set of state matrix locations which have available measurements.
While it is possible that there may be just enough well placed measurements in $\Omega$ that can uniquely determine the other missing values by using \eqref{eq:syseqn_line_pwrcons}-\eqref{eq:syseqn_mag}, it not likely to be the case in a distribution network.

Therefore, the goal of this \emph{under-determined} state estimation problem is to accurately fill in any unmeasured values in the state matrix $\mathbf{M}$, especially the complex bus voltages, using the available measured values from locations $\Omega$ and the power system equations \eqref{eq:syseqn_line_pwrcons}-\eqref{eq:syseqn_mag}. To that end, it was recently proposed in \cite{donti2019matrix} to leverage the approximately low-rank structure of the state matrix (as was demonstrated in Section \ref{sec:motiv_examp}) to find a \emph{minimum rank matrix} that satisfies  \eqref{eq:syseqn_line_pwrcons}-\eqref{eq:syseqn_mag} and matches the measured state values:
\begin{subequations}
    \label{eq:prob_state_est}
    \begin{align}
        \min_{\mathbf{X}} \quad & \text{rank}(\mathbf{X}) \nonumber \\
        \text{s.t.} \quad & X_{ij}=M_{ij} \quad \forall (i,j)\in\Omega \nonumber \\
        & \eqref{eq:syseqn_line_pwrcons}-\eqref{eq:syseqn_mag}. \nonumber
    \end{align}
\end{subequations}

However, there are two issues with the above problem formulation that make it non-convex, thus computationally hard to solve: (i) the objective function is non-convex; and (ii) the equality constraints \eqref{eq:syseqn_PVR}-\eqref{eq:syseqn_mag} are not linear, therefore the feasible solution space for $\mathbf{X}$ is non-convex. To tackle the first challenge, a standard relaxation using \emph{nuclear norm} \cite{candes2009exact} is used; see Section \ref{sec:nuclear} for details. To tackle the second challenge, constraints \eqref{eq:syseqn_PVR}-\eqref{eq:syseqn_mag} are replaced with their \emph{linear approximation} as proposed in \cite{donti2019matrix}. One must be careful to only add the linear approximations which contain at least one state variable that is not measured; otherwise, the problem may be infeasible.
This is because while $\mathbf{M}$ is assumed to satisfy Equations \eqref{eq:syseqn_line_pwrcons}-\eqref{eq:syseqn_mag}, it is possible that it does not satisfy the linear approximation equations.
For example, the relationship of the voltage magnitude difference across a line and complex power flow $\left(P^\text{Flow}_{s,t},Q^\text{Flow}_{s,t}\right)$ on that line can be linearly approximated for a radial distribution network~\cite{baran1989optimal,donti2019matrix}
\begin{align}
    \label{eq:lin_approx_const}
    |V_t|-|V_s| = \frac{1}{|V_1|}\left(R_{s,t}P^\text{Flow}_{s,t} + X_{s,t}Q^\text{Flow}_{s,t}\right) \quad \forall (s,t)\in\mathcal{L}
\end{align}
where $R_{s,t}$ and $X_{s,t}$ are the resistance and reactance of line $(s,t)$ and $V_1$ is the voltage of the slack bus.
It assumes either that the lines have no losses or that the power flow is so low that losses are negligible.
Since the state $\mathbf{M}$ does not make this assumption and does not encode the power flows directly, we can approximate the power flow by taking the average power injection into the line, i.e. $P^\text{Flow}_{s,t}:=\left(P^\text{From}_{s,t}-P^\text{To}_{s,t}\right)/2$ and $Q^\text{Flow}_{s,t}:=\left(Q^\text{From}_{s,t}-Q^\text{To}_{s,t}\right)/2$.

\subsection{Constrained Matrix Completion and Sample Complexity}
\label{sec:const_matrix_comp}

The under-determined state estimation problem with linear system equations can be generalized to the following low-rank matrix completion problem with $\fixes{h}$ linear equality constraints: 
\begin{subequations}
    \label{eq:min_rank_prob}
	\begin{align}
		\min_{\mathbf{X}} \quad & \text{rank}(\mathbf{X}) &  \\
    	\text{s.t.} \quad &  X_{ij}=M_{ij} & \forall (i,j)\in\Omega \\
        & \langle \mathbf{A}^{(l)}, \mathbf{X} \rangle = b^{(l)} & \forall l\in\{1,\dots,\fixes{h}\} \label{eq:min_rank_prob_const}
	\end{align}
\end{subequations}
where the matrix inner product is defined as $\langle \mathbf{A}, \mathbf{X} \rangle:=\text{trace}(\mathbf{A}^\intercal\mathbf{X})$.
An equivalent way to write the linear equality constraints is $$\sum_i\sum_jA^{(l)}_{ij}X_{ij}=b^{(l)}:\forall l\in\{1,\dots,\fixes{h}\}.$$

Let $m$ be cardinality of $\Omega$ and assume that the locations of $\mathbf{M}$ that make up $\Omega$ are sampled uniformly at random.
The question for this general constrained matrix completion problem becomes \emph{how large does $m$ need to be so that the solution to Problem \eqref{eq:min_rank_prob} is guaranteed to exactly match $\mathbf{M}$?} This value of $m$ is is referred to as \emph{sample complexity}.

\subsection{Nuclear Norm Minimization} \label{sec:nuclear}

Notice that if a matrix has rank $r$, then it also means that it has $r$ nonzero singular values.
Therefore, a simple heuristic of minimizing the sum of its singular values is used to approximate the minimization its rank~\cite{fazel2002matrix}.
This heuristic is actually the definition of the nuclear norm which is convex:
\begin{align}
    \|\mathbf{X}\|_* := \sum_{k=1}^{r}\sigma_k(\mathbf{X}) \nonumber
\end{align}
where $\sigma_k(\mathbf{X})$ is the $k$th largest singular value.
With the substitution of the nuclear norm in place of the rank operator, we reformulate the matrix completion problem \eqref{eq:min_rank_prob} to be
\begin{subequations}
    \label{eq:nuc_norm_prob}
	\begin{align}
		\min_{\mathbf{X}} \quad & \|\mathbf{X}\|_* & \\
    	\text{s.t.} \quad &  X_{ij}=M_{ij} & \forall (i,j)\in\Omega \\
        & \langle \mathbf{A}^{(l)}, \mathbf{X} \rangle = b^{(l)} & \forall l\in\{1,\dots,\fixes{h}\} \label{eq:nuc_norm_prob_const}
	\end{align}
\end{subequations}
with the same question as before on the sample complexity for $\Omega$ under uniform random sampling.

\section{Main Result on Sample Complexity}
\label{sec:approx}

In this section, we formulate an improved result on sample complexity that takes advantage of the linear equality constraints in the problem formulation.
The main challenge is on how to measure the information from the added constraints in terms of sample size which can be used to partially replace the need for extra measurements.
The intuition behind the usefulness of the added constraints \eqref{eq:min_rank_prob_const} is that each constraint may eliminate a single degree of freedom from the feasible solution set.
Thus, a set of constraints may decrease the search space for an approximation method so that less samples are needed to recover the underlying matrix $\mathbf{M}$.

\subsection{Degrees of Freedom in a Matrix} \label{sec:degrees}
Let $\mathbf{M}$ be an $n_1\times n_2$ matrix of rank $r$ which satisfies:
\begin{subequations}
	\begin{align}
		\mathbf{M} & = \sum_{k=1}^r\sigma_k\mathbf{u}_k\mathbf{v}_k^\intercal \label{eq:svd} \\
        \langle \mathbf{A}^{(l)}, \mathbf{M} \rangle & = b^{(l)} \quad \forall l\in\{1,\dots,\fixes{h}\} \label{eq:svd_const}
	\end{align}
\end{subequations}
where \eqref{eq:svd} is its Singular Value Decomposition (SVD).
Without loss of generality, we assume that $n_1\geq n_2$.
The vectors $\mathbf{u}_1,\dots,\mathbf{u}_r$ are unit vectors of size $n_1$ that are orthogonal to each other and the vectors $\mathbf{v}_1,\dots,\mathbf{v}_r$ are unit vectors of size $n_2$ that are orthogonal to each other.
The scalars $\sigma_1,\dots,\sigma_r$ which are used to linearly combine the matrices $\mathbf{u}_1\mathbf{v}_1^\intercal,\dots,\mathbf{u}_r\mathbf{v}_r^\intercal$ to be equal to $\mathbf{M}$ are called its singular values.
By convention, the singular values are listed in decreasing order so that $\sigma_k$ refers to the $k$th largest singular value in \eqref{eq:svd}.

The number of degrees of freedom of any $n_1\times n_2$ matrix of rank $r$ is $r(n_1+n_2-r)$.
This can be calculated by summing the degrees of freedom from each set individually that together make up the SVD \eqref{eq:svd}: $\{\mathbf{u}_1,\dots,\mathbf{u}_r\}$, $\{\mathbf{v}_1,\dots,\mathbf{v}_r\}$, and $\{\sigma_1,\dots,\sigma_r\}$.
The first and second sets have $r(n_1-1)-\sum_{k=1}^{r-1} k = r\left(n_1-\frac{1}{2}(r-1)\right)$ and $r\left(n_2-\frac{1}{2}(r-1)\right)$ degrees of freedom, respectively, since there are $r$ unit vectors of size $n_1$ (and $n_2$) while subtracting off the fact that the vectors within the set must be orthogonal to each other.
The last set trivially has $r$ degrees of freedom and summed all together it gives $r(n_1+n_2-r)$.

\subsection{High-Probability Exact Completion}

Due to the probabilistic nature of the question on sample complexity, the answer will also be probabilistic.
This is because for any given number of samples taken that is less than $(n_1-1)n_2$, there is some probability that the sampled locations will miss an entire row and thus have no information that can be used to recover it.
Thus, our goal will be to determine how large does $m$, the cardinality of $\Omega$, need to be to ensure a high probability of exact completion using the optimal solution to Problem \eqref{eq:nuc_norm_prob}.
Another way to frame the objective is to find the conditions on $m$ and $\mathbf{M}$ such that $\mathbf{M}$ is the unique solution to \eqref{eq:nuc_norm_prob} with some probability.

A property of the underlying matrix $\mathbf{M}$ that must be understood is how well its information is spread among its columns and rows.
A matrix with its information not well spread will require many samples.
For example, suppose there is a rank-1 matrix where $\mathbf{u}_1=[1 \ 0 \ \dots \ 0]^\intercal$ and $\mathbf{v}_1=[1 \ 0 \ \dots \ 0]^\intercal$ in \eqref{eq:svd}.
Therefore, in order to have a high probability of exact completion by any method, there must be a high probability that the location $(1,1)$ is in $\Omega$.
Otherwise, it would be impossible to have a guaranteed correct guess of the value in $(1,1)$ without having observed it.
For this reason, \cite{candes2009exact} defines a property on the space spanned by either $(\mathbf{u}_1,\dots,\mathbf{u}_r)$ or $(\mathbf{v}_1,\dots,\mathbf{v}_r)$ which measures the spread of the weight of its elements compared to the standard basis, called \emph{coherence}. 
\begin{definition}
	For any subspace $\mathcal{U}$ in $\mathbb{R}^n$ with dimension $r$, let the coherence of $\mathcal{U}$ be defined as
    \begin{align}
    	\mu(\mathcal{U}):=\frac{n}{r}\max_{i\in\{1,\dots,n\}} \left\|\mathbf{P}_\mathcal{U}\mathbf{e}_i\right\|^2 \nonumber
    \end{align}
    where $\mathbf{P}_\mathcal{U}$ is the orthogonal projection matrix onto $\mathcal{U}$ and $\mathbf{e}_i$ is the $i$-th standard basis vector with dimension $n$.
\end{definition}
The maximum possible value of $\mu(\mathcal{U})$ is $\frac{n}{r}$ when the subspace contains a standard basis vector, while its minimum possible value is 1, for example if its basis is spanned by vectors with elements that each have a magnitude of $\frac{1}{\sqrt{n}}$.
With the following assumption, the lack of spread of information within $\mathbf{M}$ can be bounded by bounding the coherence of the spaces defined by the vectors in its SVD \eqref{eq:svd}.
\begin{assumption}
    \label{ass:coherence}
	The coherence of $\mathcal{U}:=\text{span}(\mathbf{u}_1,\dots,\mathbf{u}_r)$ and the coherence of $\mathcal{V}:=\text{span}(\mathbf{v}_1,\dots,\mathbf{v}_r)$ are both upper bounded by some constant $\mu_0>0$, i.e.
    \begin{align}
    	\max\{\mu(\mathcal{U}),\mu(\mathcal{V})\}\leq \mu_0 \nonumber
    \end{align}
\end{assumption}


\fixes{To limit the concentration of information in the subgradient of the nuclear norm at $\mathbf{M}$ for any specific matrix location, an assumption is placed on the maximum value of sum of the rank-1 matrices $\mathbf{u}_1\mathbf{v}_1^\intercal,\dots,\mathbf{u}_r\mathbf{v}_r^\intercal$ through the parameter $\nu_0$.
\begin{assumption}
    \label{ass:maxEval}
	The absolute value of each element in $\sum_{k=1}^r\mathbf{u}_k\mathbf{v}_k^\intercal$ is upper bounded by $\nu_0\sqrt{r/(n_1n_2)}$ for some constant $\nu_0$.
\end{assumption}}

One important item needed in proving that $\mathbf{M}$ is the unique solution to \eqref{eq:nuc_norm_prob} is a vector space of matrices $\mathcal{T}$ that contains all $n_1\times n_2$ matrices which have a column space in $\mathcal{U}:=\text{span}(\mathbf{u}_1,\dots,\mathbf{u}_r)$, i.e. the column space of $\mathbf{M}$, and all $n_1\times n_2$ matrices which have a row space in $\mathcal{V}:=\text{span}(\mathbf{v}_1,\dots,\mathbf{v}_r)$, i.e. the row space of $\mathbf{M}$.
Specifically, a vector space $\mathcal{T}$ of matrices is built from all the combinations of $\mathbf{u}_1,\dots,\mathbf{u}_r$ that can span the column space and all the combinations of $\mathbf{v}_1,\dots,\mathbf{v}_r$ that can span the row space via their outer products with the vectors $\{\mathbf{x}_1,\dots,\mathbf{x}_r\}\in\mathbb{R}^{n_2}$ and $\{\mathbf{y}_1,\dots,\mathbf{y}_r\}\in\mathbb{R}^{n_1}$:
\begin{align}
    \mathcal{T}:=\bigg\{\sum_{k=1}^r\left(\mathbf{u}_k\mathbf{x}^\intercal_k+\mathbf{y}_k\mathbf{v}_k^\intercal\right):\mathbf{x}_k\in\mathbb{R}^{n_2},\mathbf{y}_k\in\mathbb{R}^{n_1}\bigg\} \nonumber
\end{align}
This vector space has a dimension of $r(n_1+n_2-r)$ which is equal to the degrees of freedom in any $n_1\times n_2$ matrix of rank $r$ (see Section \ref{sec:degrees}). 
Additionally, its orthogonal complement $\mathcal{T}^\perp$ will  also be important which is the vector space that contains the matrices $\mathbf{y}\mathbf{x}^\intercal$, where $\mathbf{y}$ is any vector orthogonal to the column space of $\mathbf{M}$ and $\mathbf{x}$ is any vector orthogonal to the row space of $\mathbf{M}$.

Conditions on the orthogonal projection of the nuclear norm's subgradient onto $\mathcal{T}$ and $\mathcal{T}^\perp$ will be shown later to give sufficient conditions for determining if a particular matrix is optimal to Problem \eqref{eq:nuc_norm_prob} (see Lemma \ref{thm:nuc_norm_opt}).
The orthogonal projection of $\mathbf{X}\in\mathbb{R}^{n_1\times n_2}$ onto $\mathcal{T}$ can be stated from
its projections onto $\mathcal{U}$ and $\mathcal{V}$ (see Equation (3.5) in~\cite{candes2009exact}):
\begin{align}
    \mathcal{P}_\mathcal{T}(\mathbf{X}) = \mathbf{P}_\mathcal{U}\mathbf{X} + \mathbf{X}\mathbf{P}_\mathcal{V} - \mathbf{P}_\mathcal{U}\mathbf{X}\mathbf{P}_\mathcal{V} \label{eq:PT_defn}
\end{align}
where $\mathbf{P}_\mathcal{U}$ and $\mathbf{P}_\mathcal{V}$ are the orthogonal projection matrices onto $\mathcal{U}$ and $\mathcal{V}$ respectively.
The orthogonal projection onto $\mathcal{T}^\perp$ can also be stated in the same manner:
\begin{align}
    \mathcal{P}_{\mathcal{T}^\perp}(\mathbf{X}) = (\mathbf{I}_{n_1}-\mathbf{P}_\mathcal{U})\mathbf{X}(\mathbf{I}_{n_2}-\mathbf{P}_\mathcal{V}). \nonumber
\end{align}

\fixes{To measure the amount of useful information held in the linear equalities \eqref{eq:svd_const} that can explain $\mathbf{M}$, we develop quantities similar to the upper bounds of Assumptions \ref{ass:coherence} and \ref{ass:maxEval} for the vector space spanned by $\mathbf{A}^{(1)},\dots,\mathbf{A}^{(\fixes{h})}$, denoted by $\mathcal{Q}$.
First, we measure how much of the vector space $\mathcal{T}$ remains uncovered by $\mathcal{Q}$:
\begin{subequations}
    \begin{align}
        \mu_{\mathcal{Q}^\perp} & := \frac{\sum_{i=1}^{n_1}\sum_{j=1}^{n_2}\left\|\mathcal{P}_\mathcal{T}\mathcal{P}_{\mathcal{Q}^\perp}\left(\mathbf{e}_i\mathbf{e}_j^\intercal\right)\right\|_F^2}{\sum_{i=1}^{n_1}\sum_{j=1}^{n_2}\left\|\mathcal{P}_\mathcal{T}\left(\mathbf{e}_i\mathbf{e}_j^\intercal\right)\right\|_F^2}. \label{eq:muQperp_defn}
    \end{align}
\end{subequations}
This measurement gives an element-wise average of non-coverage by $\mathcal{Q}$ in $\mathcal{T}$ which has a maximum value of 1.
Second, we measure how much of the subgradient of the nuclear norm at $\mathbf{M}$ is not contained in $\mathcal{Q}$:
\begin{subequations}
    \begin{align}
        \nu_{\mathcal{Q}^\perp} & := \frac{1}{r}\left\|\mathcal{P}_{\mathcal{Q}^\perp}\left(\mathbf{E}\right)\right\|_F^2. \label{eq:nuQperp_defn}
    \end{align}
\end{subequations}
where $\mathbf{E}:=\sum_{k=1}^r\mathbf{u}_k\mathbf{v}_k^\intercal$ and has a maximum value of 1.}

\fixes{Notice that if $\mathcal{Q}$ covers the entire space of $\mathcal{T}$ (i.e. $\mathcal{Q}\supseteq\mathcal{T}$), then $\mu_{\mathcal{Q}^\perp}=\nu_{\mathcal{Q}^\perp}=0$ since $\mathbf{E}\in\mathcal{T}$.
This extreme case will be important in explaining the significance of our main result, Theorem \ref{thm:samp_complex}, in regards to how much fewer samples are needed for exact completion.
However, this does not mean that no observations are needed because the useful information described above only refers to the information in the $r$ rank-1 matrices but does not say anything about the singular values themselves that need to be determined.}

Finally, using the above definitions and assumptions, we can state our theorem on sample complexity with a high-probability matrix completion guarantee. 

\begin{theorem}
    \label{thm:samp_complex}
    Let $\mathbf{M}$ be an $n_1 \times n_2$ matrix with $n_1 \geq n_2$ such that the following $\fixes{h}$ linear equality constraints are satisfied: $\langle \mathbf{A}^{(l)}, \mathbf{M} \rangle = b^{(l)}$ for all $l\in\{1,\dots,\fixes{h}\}$.
    Also, let $\mathbf{M}$ be of rank $r$ and have the following singular value decomposition $\sum_{k=1}^r\sigma_k\mathbf{u}_k\mathbf{v}_k^\intercal$ that satisfies Assumptions \ref{ass:coherence} and \fixes{\ref{ass:maxEval}}.
    Suppose that $m$ entries of $\mathbf{M}$ are sampled uniformly at random. Then there exists constants \fixes{$C_R$ and $C_K$ such that if
    \begin{subequations}
        \begin{align}
            m &  >  \left(C_K e^2 \nu_0 \sqrt{\beta r n_1\log n_1}2^{\frac{2}{\beta \log n_1}+\frac{5}{2}}+2\sqrt{qn_1n_2\sqrt{\nu_{\mathcal{Q}^\perp}r}}\right)^2 \nonumber \\
            & \quad - qn_1n_2, \label{eq:thm_samp_complex_a} \\
            m & > \sqrt{10\mu_0 rn_1n_2}\left(C_R \sqrt{\beta  rn_1\log n_1}+n_1n_2\sqrt{\mu_{\mathcal{Q}^\perp}q}\right) \nonumber \\
            & \quad \times (1+\sqrt{q}) - qn_1n_2, \label{eq:thm_samp_complex_b} \\
            m & \geq 2^4C_R^2\beta \mu_0 r  n_1 \log n_1 - qn_1n_2, \label{eq:thm_samp_complex_c} \\
            m & \geq 2^4\mu_{\mathcal{Q}^\perp}\mu_0qn_1^2n_2^2  - qn_1n_2, \label{eq:thm_samp_complex_d} \\
            m & \geq qn_1n_2\sqrt{\nu_{\mathcal{Q}^\perp}r} - qn_1n_2, \label{eq:thm_samp_complex_e} \\
            m & \geq \max\{2,\beta\}n_1\log n_1 \label{eq:thm_samp_complex_f}
    \end{align}
    \end{subequations}}
    for some $\beta\geq 1$ \fixes{and $q>0$}, then the solution to Problem \eqref{eq:nuc_norm_prob} is unique and equal to $\mathbf{M}$ with probability at least $1-6n_1^{-\beta}$.
\end{theorem}

The proof is given in detail in the next subsection.

\begin{corollary}
    \label{thm:neighborhood1}
    \fixes{If $\mu_{\mathcal{Q}^\perp} < \min\left\{\frac{1}{2^4 },\frac{1}{10r}\right\}\frac{1}{\mu_0n_1n_2}$, $\nu_{\mathcal{Q}^\perp}<\frac{1}{2^4 r}$, and
    \begin{align}
        m & \geq \max\{2,\beta\}n_1\log n_1
    \end{align}
    for some $\beta\geq 1$, then the solution to Problem \eqref{eq:nuc_norm_prob} is unique and equal to $\mathbf{M}$ with probability at least $1-6n_1^{-\beta}$.}
\end{corollary}

\fixes{See Appendix \ref{sec:neighborhood1} for the proof.}

\begin{corollary}
    \label{thm:Qperp_equal1}
    \fixes{If $\mu_{\mathcal{Q}^\perp} = \nu_{\mathcal{Q}^\perp} =1$, and
    \begin{align}
        m & > \max\left\{C_1\nu_0^2\beta r,C_2\sqrt{\mu_0n_2}\beta r,C_3 \mu_0\beta r,2,\beta\right\} n_1\log n_1
    \end{align}
    for some $\beta\geq 1$ where $C_1:=32C_K^2e^42^{\frac{4}{\beta \log n_1}}$, $C_2:=\sqrt{10}C_R$, $C_3:=16C_R^2$ then the solution to Problem \eqref{eq:nuc_norm_prob} is unique and equal to $\mathbf{M}$ with probability at least $1-6n_1^{-\beta}$.}
\end{corollary}

\begin{proof}
    \fixes{The resultant comes directly by taking the limit of \eqref{eq:thm_samp_complex_a} - \eqref{eq:thm_samp_complex_e} with $q$ approaching zero from the right in Theorem \ref{thm:samp_complex}.}
\end{proof}

\begin{remark}
    \fixes{The sample complexity described in Corollary \ref{thm:Qperp_equal1} is within $O(\max\{\mu_0^{-\frac{1}{2}}n_2^{\frac{1}{4}},\mu_0^{-\frac{1}{2}}\nu_0\})$ of \cite{candes2009exact} for the unconstrained problem.}
\end{remark}

\fixes{Corollary \ref{thm:neighborhood1} shows us that when $\mu_{\mathcal{Q}^\perp}$ and $\nu_{\mathcal{Q}^\perp}$ are close to 0, i.e. $\mathcal{Q}$ almost completely covers $\mathcal{T}$, then the reduction in sample complexity is on the order of $n_1\log n_1$.
On the other hand, from Corollary \ref{thm:Qperp_equal1} when $\mu_{\mathcal{Q}^\perp}=\nu_{\mathcal{Q}^\perp}=1$, i.e. $\mathcal{Q}$ does not cover any of $\mathcal{T}$, there is no reduction.
Therefore, this is a preliminary indication that $\mu_{\mathcal{Q}\perp}$ and $\nu_{\mathcal{Q}\perp}$ are useful metrics to characterize the sample complexity reduction from constraints.
For the cases between the conditions of Corollary \ref{thm:neighborhood1} and \ref{thm:Qperp_equal1}, the parameter $q>0$ can be tuned to minimize the active inequalities of \eqref{eq:thm_samp_complex_a} - \eqref{eq:thm_samp_complex_e} in Theorem \ref{thm:samp_complex}.}

When connecting this theoretical result back to the state estimation problem for a distribution network,
\fixes{it shows us that states which result in a $\mathcal{T}$ that lies in the constraint vector space $\mathcal{Q}$ require significantly less samples for estimation than those that are not.}

\subsection{Proof of Theorem \ref{thm:samp_complex}}

A challenge with the uniform random sampling model is that the probability that an element will be sampled depends on which ones have already been sampled.
Instead of proving Theorem \ref{thm:samp_complex} directly with uniform random sampling, we prove it with the following Bernoulli sampling model:
\begin{subequations}
    \label{eq:bernoulli}
    \begin{align}
        \mathbb{P}(\delta_{ij}=1) & = p := \frac{m}{n_1n_2} \\
        \Omega' & := \left\{(i,j):\delta_{ij}=1\right\}.
    \end{align}
\end{subequations}
since the probability an element will be sampled is independent of whether any other element has been sampled or not.
Then we can invoke the result from \cite{candes2006robust}, Section II.C, which states that the probability of failure with the uniform sampling model can be bounded by twice that of the Bernoulli sampling model.

First, we have a lemma which states sufficient optimality conditions for Problem \eqref{eq:nuc_norm_prob} which are later shown in the proof to be satisfied by $\mathbf{M}$ under certain probabilistic conditions.
The two lemmas following the first are used in the proof to define the conditions which $\mathbf{M}$ satisfies Lemma \ref{thm:nuc_norm_opt}.
Let $\mathcal{P}_\Omega(\mathbf{X})$ be the projection operator onto $\Omega$ where projection's element equals $X_{ij}$ if $(i,j)\in\Omega$ or 0 otherwise.
Also, let the $\mathcal{R}_\Omega(\mathbf{X})$ be the sampling operator which maps the elements of $\mathbf{X}$ to a vector of size $m$ for only the element locations that are in $\Omega$.

\begin{lemma}
    \label{thm:nuc_norm_opt}
    Suppose that there exists some matrix $\mathbf{X}_0=\sum_{k=1}^{r}\sigma_k\mathbf{u}_k\mathbf{v}_k^\intercal$ of rank $r$ that is feasible to Problem \eqref{eq:nuc_norm_prob}.
    If it satisfies the two following conditions:
    \begin{enumerate}
        \item There exists a dual point $(\boldsymbol{\lambda},\boldsymbol{\gamma})$ where $\mathbf{Y}=\mathcal{R}_\Omega^\intercal\boldsymbol{\lambda}+\sum_{l=1}^{\fixes{h}} \gamma_l\mathbf{A}^{(l)}$ such that
        \begin{subequations}
            \begin{align}
                \mathcal{P}_\mathcal{T}(\mathbf{Y}) & = \sum_{k=1}^{r}\mathbf{u}_k\mathbf{v}_k^\intercal \\
                \left\|\mathcal{P}_{\mathcal{T}^\perp}(\mathbf{Y})\right\| & < 1
            \end{align}
        \end{subequations}
        \item The sampling operator $\mathcal{R}_\Omega$ restricted to the elements of $\mathcal{T}$ is injective.
    \end{enumerate}
    then $\mathbf{X}_0$ is the unique solution to Problem \eqref{eq:nuc_norm_prob}.
\end{lemma}
See Appendix \ref{sec:proof_nuc_norm_opt} for the proof.


\begin{lemma}
    \label{thm:inverse_condition}
    \fixes{Suppose that if $\Omega$ is sampled according to the Bernoulli model \eqref{eq:bernoulli}, $n_1\geq n_2$, Assumptions \ref{ass:coherence} and \ref{ass:maxEval} are satisfied,
    \begin{align}
        C_R\sqrt{\frac{\beta r \mu_0 n_1 \log n_1}{m+qn_1n_2}} \leq \frac{1}{4}, \quad \text{and} \quad \sqrt{\frac{\mu_{\mathcal{Q}^\perp}\mu_0qn_1^2n_2^2}{m+qn_1n_2}} \leq \frac{1}{4}, \nonumber
    \end{align}
    then the following inequalities are true
    \begin{subequations}
        \begin{align}
            \|\mathcal{P}_\mathcal{T}(\mathcal{P}_\Omega+q\mathcal{P}_\mathcal{Q})\mathcal{P}_\mathcal{T}(\mathbf{X})\|_F \geq \frac{1}{2}(p+q)\|\mathcal{P}_\mathcal{T}(\mathbf{X})\|_F \label{eq:inverse_condition_a} \\
            \|\mathcal{P}_\mathcal{T}(\mathcal{P}_\Omega+q\mathcal{P}_\mathcal{Q})\mathcal{P}_\mathcal{T}(\mathbf{X})\|_F \leq \frac{3}{2}(p+q)\|\mathcal{P}_\mathcal{T}(\mathbf{X})\|_F \label{eq:inverse_condition_b} \\
            \|(\mathcal{P}_\Omega+q\mathcal{P}_\mathcal{Q})\mathcal{P}_\mathcal{T}(\mathbf{X})\|_F \leq \frac{1+\sqrt{q}}{2}\sqrt{5(p+q)}\|\mathcal{P}_\mathcal{T}(\mathbf{X})\|_F \label{eq:inverse_condition_c}
        \end{align}
        with probability at least $1-3n_1^{-\beta}$.
    \end{subequations}}
\end{lemma}
\fixes{See Appendix \ref{sec:proof_inverse_condition} for the proof.}



\begin{lemma}
    \label{thm:bound_Hsum}
    Suppose that if $\Omega$ is sampled according to the Bernoulli model \eqref{eq:bernoulli}, $n_1\geq n_2$, Assumption \ref{ass:coherence} is satisfied,
    \begin{align}
        \fixes{C_R\sqrt{\frac{\beta r \mu_0 n_1 \log n_1}{m+qn_1n_2}} \leq \frac{1}{4}, \quad \text{and} \quad \sqrt{\frac{\mu_{\mathcal{Q}^\perp}\mu_0qn_1^2n_2^2}{m+qn_1n_2}} \leq \frac{1}{4},} \nonumber
    \end{align}
    then there are numerical constants $C_R$ and $C_{k_0}$ such that for all $\beta>1$,
    \begin{align}
        \frac{1}{p\fixes{+q}}\left\|\mathcal{P}_{\mathcal{T}^\perp}(\mathcal{P}_\Omega\fixes{+q\mathcal{P}_\mathcal{Q}})\mathcal{P}_\mathcal{T}\sum_{k=k_0}^\infty \mathcal{H}^k(\fixes{\mathbf{E}})\right\| \quad\quad\quad\quad\quad\quad \nonumber
    \end{align}
    \begin{align}
        \leq \frac{\fixes{(1+\sqrt{q})}\sqrt{\fixes{10\mu_0r}n_1n_2}\left(C_R \sqrt{\beta rn_1\log n_1}\fixes{+n_1n_2\sqrt{\mu_{\mathcal{Q}^\perp}q}}\right)^{k_0}}{(m\fixes{+qn_1n_2})^{\frac{k_0+1}{2}}}    \label{eq:spectral_norm_sum}
    \end{align}
    with probability at least $1-3n_1^{-\beta}$.
    The operator $\mathcal{H}$ is defined by \eqref{eq:H_defn}.
\end{lemma}
See Appendix \ref{sec:proof_bound_Hsum} for the proof.


\begin{lemma}
    \label{thm:bound_H0}
    \fixes{Suppose that if $\Omega$ is sampled according to the Bernoulli model \eqref{eq:bernoulli}, $n_1\geq n_2$, Assumption \ref{ass:maxEval} is satisfied, $\max\{2,\beta\}n_1\log n_1 \leq m$, and $\frac{qn_1n_2\sqrt{\nu_{\mathcal{Q}^\perp}r}}{m+qn_1n_2} \leq 1$, then
    \begin{align}
        \frac{1}{p\fixes{+q}}\left\|\mathcal{P}_{\mathcal{T}^\perp}(\mathcal{P}_\Omega\fixes{+q\mathcal{P}_\mathcal{Q}})\mathcal{P}_\mathcal{T}(\mathbf{E})\right\| \quad\quad\quad\quad\quad\quad\quad\quad\quad \nonumber
    \end{align}
    \begin{align}
        \leq \frac{C_K e^2 \nu_0 \sqrt{\beta r n_1\log n_1}2^{\frac{2}{\beta \log n_1}+\frac{3}{2}}+\sqrt{qn_1n_2\sqrt{\nu_{\mathcal{Q}^\perp}r}}}{\sqrt{m+qn_1n_2}} \label{eq:spectral_norm_H0}
    \end{align}
    with probability at least $1-n_1^{-\beta}$.}
\end{lemma}

\fixes{See Appendix \ref{sec:proof_bound_H0} for the proof.}

Finally with the above lemmas in place, we formally prove Theorem \ref{thm:samp_complex}.
The main effort of this proof will be to show that there exists a subgradient of the nuclear norm at $\mathbf{M}$, the underlying matrix, which satisfies Lemma \ref{thm:nuc_norm_opt} under high-probability given a sufficient number of samples.
If Lemma \ref{thm:nuc_norm_opt} is satisfied, then the optimal solution to Problem \eqref{eq:nuc_norm_prob} and $\mathbf{M}$ uniquely coincide.
It is first proved for the Bernoulli sampling model \eqref{eq:bernoulli} and then the results are converted into that of being under the uniform sampling model.
\begin{proof}
    The first step is to find a candidate $\mathbf{Y}$ that satisfies the first condition of Lemma \ref{thm:nuc_norm_opt} by solving
    \begin{subequations}
        \label{eq:candidateY_prob}
	    \begin{align}
		    \min_{\mathbf{X},\fixes{\mathbf{Z}}} \quad & \frac{1}{2}\|\mathbf{X}\|_F^2\fixes{+\frac{1}{2q}\|\mathbf{Z}\|_F^2} \\
    	    \text{s.t.} \quad &  \mathcal{P}_\mathcal{T}\left(\mathcal{P}_\Omega(\mathbf{X}) \fixes{+\mathcal{P}_\mathcal{Q}(\mathbf{Z})}\right)= \fixes{\mathbf{E}}
	    \end{align}
    \end{subequations}
    \fixes{where $q>0$ can later be tuned} and using its solution \fixes{$\mathbf{X}+\mathbf{Z}$} as $\mathbf{Y}$.

    The Karush-Kuhn-Tucker conditions for optimality with dual matrix $\boldsymbol{\nu}$ \fixes{(not to be confused with the constants $\nu_0$, $\nu_\mathcal{Q}$, and $\nu_{\mathcal{Q}^\perp}$)} gives us
    \begin{subequations}
        \begin{align}
            \mathbf{X} - \mathcal{P}_\Omega\mathcal{P}_\mathcal{T}(\boldsymbol{\nu}) & = 0 \label{eq:candidateY_prob_kkt1} \\
            \fixes{\frac{1}{q}\mathbf{Z} - \mathcal{P}_\mathcal{Q}\mathcal{P}_\mathcal{T}(\boldsymbol{\nu})} & = 0 \label{eq:candidateY_prob_kkt2} \\
            \mathcal{P}_\mathcal{T}\left(\mathcal{P}_\Omega(\mathbf{X}) \fixes{+\mathcal{P}_\mathcal{Q}(\mathbf{Z})}\right) - \fixes{\mathbf{E}} & = 0. \label{eq:candidateY_prob_kkt3}
        \end{align}
    \end{subequations}
    To solve for a closed form of $\mathbf{X}\fixes{+\mathbf{Z}}$, start with \eqref{eq:candidateY_prob_kkt1} \fixes{and \eqref{eq:candidateY_prob_kkt2}:}
    \begin{subequations}
        \begin{align}
            \mathbf{X} & = \mathcal{P}_\Omega\mathcal{P}_\mathcal{T}(\boldsymbol{\nu}), \label{eq:candidateY_prob_X}  \\
            \fixes{\mathbf{Z}} & = \fixes{q\mathcal{P}_\mathcal{Q}\mathcal{P}_\mathcal{T}(\boldsymbol{\nu})}. \label{eq:candidateY_prob_Z}
    \end{align}
    \end{subequations}
    Next, apply $\mathcal{P}_\mathcal{T}\mathcal{P}_\Omega$ \fixes{and $\mathcal{P}_\mathcal{T}\mathcal{P}_\mathcal{Q}$, respectively:}
    \begin{align}
        \mathcal{P}_\mathcal{T}\mathcal{P}_\Omega(\mathbf{X}) & = \mathcal{P}_\mathcal{T}\mathcal{P}_\Omega\mathcal{P}_\Omega\mathcal{P}_\mathcal{T}(\boldsymbol{\nu}) = \mathcal{P}_\mathcal{T}\mathcal{P}_\Omega\mathcal{P}_\mathcal{T}(\boldsymbol{\nu}), \nonumber \\
        \fixes{\mathcal{P}_\mathcal{T}\mathcal{P}_\mathcal{Q}(\mathbf{Z})} & = \fixes{q\mathcal{P}_\mathcal{T}\mathcal{P}_\mathcal{Q}\mathcal{P}_\mathcal{Q}\mathcal{P}_\mathcal{T}(\boldsymbol{\nu}) = q\mathcal{P}_\mathcal{T}\mathcal{P}_\mathcal{Q}\mathcal{P}_\mathcal{T}(\boldsymbol{\nu}),} \nonumber
    \end{align}
    and then apply \eqref{eq:candidateY_prob_kkt3} \fixes{after summing the two} to get
    \begin{align}
        \fixes{\mathbf{E}} = \mathcal{P}_\mathcal{T}\left(\mathcal{P}_\Omega\fixes{+q\mathcal{P}_\mathcal{Q}}\right)\mathcal{P}_\mathcal{T}(\boldsymbol{\nu}). \nonumber
    \end{align}
    Taking the inverse, assuming it exists, gives
    \begin{align}
        \boldsymbol{\nu} = \left(\mathcal{P}_\mathcal{T}\left(\mathcal{P}_\Omega\fixes{+q\mathcal{P}_\mathcal{Q}}\right)\mathcal{P}_\mathcal{T}\right)^{-1}(\fixes{\mathbf{E}}), \nonumber
    \end{align}
    and finally applying it to \fixes{\eqref{eq:candidateY_prob_X} and \eqref{eq:candidateY_prob_Z}} gives the candidate $\mathbf{Y}$:
    \begin{align}
        \mathbf{Y} & = \mathbf{X}\fixes{+\mathbf{Z}} \nonumber \\
        & = \left(\mathcal{P}_\Omega\fixes{+q\mathcal{P}_\mathcal{Q}}\right)\mathcal{P}_\mathcal{T}\left(\mathcal{P}_\mathcal{T}\left(\mathcal{P}_\Omega\fixes{+q\mathcal{P}_\mathcal{Q}}\right)\mathcal{P}_\mathcal{T}\right)^{-1}(\fixes{\mathbf{E}}). \label{eq:candidateY}
    \end{align}

    The benefit to the formulation of Problem \eqref{eq:candidateY_prob} is that it minimizes $\|\mathcal{P}_{\mathcal{T}^\perp}(\mathbf{Y})\|_F$.
    By Pythagoras we have
    \begin{align}
        \|\mathbf{Y}\|_F^2 & = \|\mathcal{P}_{\mathcal{T}}(\fixes{\mathbf{X}+\mathbf{Z}})\|_F^2 + \|\mathcal{P}_{\mathcal{T}^\perp}(\fixes{\mathbf{X}+\mathbf{Z}})\|_F^2. \nonumber
    \end{align}
    The first term on the RHS remains constant since $\mathcal{P}_\mathcal{T}(\fixes{\mathbf{X}+\mathbf{Z}})=\mathcal{P}_\mathcal{T}\left(\fixes{\mathcal{P}_\Omega(\mathbf{X}) +\mathcal{P}_\mathcal{Q}(\mathbf{Z})}\right)=\fixes{\mathbf{E}}$ since \fixes{from \eqref{eq:candidateY_prob_X} and \eqref{eq:candidateY_prob_Z} we have $\mathcal{P}_\Omega(\mathbf{X})=\mathcal{P}_\Omega\mathcal{P}_\Omega\mathcal{P}_\mathcal{T}(\boldsymbol{\nu})=\mathcal{P}_\Omega\mathcal{P}_\mathcal{T}(\boldsymbol{\nu})=\mathbf{X}$ and $\mathcal{P}_\mathcal{Q}(\mathbf{Z})=\mathcal{P}_\mathcal{Q}\mathcal{P}_\mathcal{Q}\mathcal{P}_\mathcal{T}(\boldsymbol{\nu})=\mathcal{P}_\mathcal{Q}\mathcal{P}_\mathcal{T}(\boldsymbol{\nu})=\mathbf{Z}$, respectively.}
    Thus, only the second term can vary  which is an upper bound on the spectral norm $\left\|\mathcal{P}_{\mathcal{T}^\perp}(\mathbf{Y})\right\|$.

    The next step is to transform the candidate subgradient $\mathbf{Y}$ given by Equation \eqref{eq:candidateY} into a more manageable form.
    Notice that by taking its projection onto $\mathcal{T}$, the candidate $\mathbf{Y}$ trivially satisfies the first part of the first condition of Lemma \ref{thm:nuc_norm_opt}.
    
    \fixes{To express $\left(\mathcal{P}_\mathcal{T}\left(\mathcal{P}_\Omega\fixes{+q\mathcal{P}_\mathcal{Q}}\right)\mathcal{P}_\mathcal{T}\right)^{-1}(\mathbf{E})$ in an analyzable form via the Neumann series,} let us define the following operator
    \begin{align}
        \mathcal{H} := \mathcal{P}_\mathcal{T}-\frac{1}{p\fixes{+q}}\mathcal{P}_{\mathcal{T}}\left(\mathcal{P}_\Omega\fixes{+q\mathcal{P}_\mathcal{Q}}\right)\mathcal{P}_\mathcal{T} \label{eq:H_defn}
    \end{align}
    \fixes{where $q\in[0,1]$ is an adjustable parameter,} which gives
    \begin{align}
        \left(\mathcal{P}_\mathcal{T}\left(\mathcal{P}_\Omega\fixes{+q\mathcal{P}_\mathcal{Q}}\right)\mathcal{P}_\mathcal{T}\right)^{-1}(\fixes{\mathbf{E}}) & = ((p\fixes{+q})(\mathbf{I}-\mathcal{H}))^{-1}(\fixes{\mathbf{E}}) \nonumber \\
        & = \frac{1}{p\fixes{+q}}\sum_{k=0}^\infty \mathcal{H}^k(\fixes{\mathbf{E}}). \nonumber
    \end{align}
    \fixes{From the first two inequalities of Lemma \ref{thm:inverse_condition}, we have that the operation $\mathcal{P}_\mathcal{T}\left(\mathcal{P}_\Omega+q\mathcal{P}_\mathcal{Q}\right)\mathcal{P}_\mathcal{T}$ is well-conditioned and thus invertible.
    Also, since the inequalities are true when $q=0$, we have that
    \begin{align}
        \frac{1}{2}p\|\mathcal{P}_\mathcal{T}(\mathbf{X})\|_F \leq \|\mathcal{P}_\mathcal{T}\mathcal{P}_\Omega\mathcal{P}_\mathcal{T}(\mathbf{X})\|_F \leq \frac{3}{2}p\|\mathcal{P}_\mathcal{T}(\mathbf{X})\|_F \nonumber
    \end{align}
    which means that the operator $\mathcal{P}_\mathcal{T}\mathcal{P}_\Omega\mathcal{P}_\mathcal{T}$ is invertible for any $p>0$.
    This results in satisfying the invertibility condition of the sampling operator $\mathcal{R}_\Omega$ in Lemma \ref{thm:nuc_norm_opt}.} 
    
    \fixes{From the above Neumann series and since $\mathbf{E}$ is in $\mathcal{T}$, we can express $\|\mathcal{P}_{\mathcal{T}^\perp}(\mathbf{Y})\|$ as
    \begin{align}
        \|\mathcal{P}_{\mathcal{T}^\perp}(\mathbf{Y})\| \quad\quad\quad\quad\quad\quad\quad\quad\quad\quad\quad\quad\quad\quad\quad\quad\quad\quad \nonumber
    \end{align}
    \begin{align}
        & = \frac{1}{p+q}\left\|\left(\mathcal{P}_{\mathcal{T}^\perp}(\mathcal{P}_\Omega+q\mathcal{P}_\mathcal{Q})\mathcal{P}_\mathcal{T}\right)\left(\sum_{k=0}^\infty \mathcal{H}^k(\mathbf{E})\right)\right\| \nonumber \\
        & \leq \frac{1}{p+q}\left\|\left(\mathcal{P}_{\mathcal{T}^\perp}(\mathcal{P}_\Omega+q\mathcal{P}_\mathcal{Q})\mathcal{P}_\mathcal{T}\right)(\mathbf{E})\right\| \nonumber \\
        & \quad + \frac{1}{p+q}\left\|\left(\mathcal{P}_{\mathcal{T}^\perp}(\mathcal{P}_\Omega+q\mathcal{P}_\mathcal{Q})\mathcal{P}_\mathcal{T}\right)\left(\sum_{k=1}^\infty \mathcal{H}^k(\mathbf{E})\right)\right\| \nonumber \\
        & \leq 2 \frac{1}{p+q} \max \Bigg\{\left\|\left(\mathcal{P}_{\mathcal{T}^\perp}(\mathcal{P}_\Omega+q\mathcal{P}_\mathcal{Q})\mathcal{P}_\mathcal{T}\right)(\mathbf{E})\right\|, \nonumber \\
        & \quad\quad\quad \left\|\left(\mathcal{P}_{\mathcal{T}^\perp}(\mathcal{P}_\Omega+q\mathcal{P}_\mathcal{Q})\mathcal{P}_\mathcal{T}\right)\left(\sum_{k=1}^\infty \mathcal{H}^k(\mathbf{E})\right)\right\| \Bigg\}. \nonumber
    \end{align}
    Thus, setting the RHS to be less than 1 and using Lemma \ref{thm:bound_H0} and Lemma \ref{thm:bound_Hsum} with $k_0=1$ in the arguments gets  \eqref{eq:thm_samp_complex_a} and \eqref{eq:thm_samp_complex_b} respectively, after solving for the sample size $m$.
    Finally, the inequalities \eqref{eq:thm_samp_complex_c} and \eqref{eq:thm_samp_complex_d} come from the conditions in Lemmas \ref{thm:inverse_condition} and \ref{thm:bound_Hsum}, and inequalities \eqref{eq:thm_samp_complex_e} and \eqref{eq:thm_samp_complex_f} come from the conditions in Lemma \ref{thm:bound_H0}.}
    
    However, everything previous to this point proves the satisfaction of Lemma \ref{thm:nuc_norm_opt} for the Bernoulli sampling model with probability at least $1-3n_1^{-\beta}$ as stated in Lemmas \ref{thm:bound_opnormH} and \ref{thm:bound_Hsum}.
    From \cite{candes2006robust} Section II.C, the probability of failure with the uniform sampling model can be bounded by twice that of the Bernoulli sampling model.
    Therefore, we multiply the failure probability by two to get $1-6n_1^{-\beta}$ as stated in the theorem.
\end{proof}

\section{Performance Evaluation}
\label{sec:perf_eval}
First, we demonstrate how the addition of linear equality constraints can decrease the number of samples needed to exactly recover the underlying matrix.
Afterwards, we apply the technique to distribution network data to show how estimation error can be decreased.

\subsection{Matrix Completion with Constraints}
\label{sec:perf_eval_toy}

The goal of this simulation is to observe the impact of the size of \fixes{$\mu_{\mathcal{Q}^\perp}$, and $\nu_{\mathcal{Q}^\perp}$} on the sample complexity through randomly generated matrix completion examples.

\subsubsection{Setup}

The underlying matrix $\mathbf{M}$ was built by first generating an $n_1\times n_2$ matrix with each element sampled independently from a \fixes{uniform} distribution.
The singular value decomposition of the generated matrix was taken which gives $n_2$ singular values $(\sigma_1,\dots,\sigma_{n_2})$, assuming $n_1\geq n_2$, and their associated basis vectors $(\mathbf{u}_1,\dots,\mathbf{u}_{n_2})$ and $(\mathbf{v}_1,\dots,\mathbf{v}_{n_2})$.
To make the rank of $\mathbf{M}$ to be $r$, \fixes{the first $r$ singular values and associated basis vectors} are combined together in \eqref{eq:svd} to get the final $\mathbf{M}$ that is used in the simulation.

The equality constraints \eqref{eq:svd_const} are generated in two steps.
First, each matrix $\mathbf{A}^{(l)}$ is made by generating an $n_1\times n_2$ matrix with each element sampled independently from a \fixes{uniform} distribution.
Afterwards each scalar $b^{(l)}$ is determined by evaluating the LHS of \eqref{eq:svd_const} \fixes{with a convex combination of $\mathbf{A}^{(l)}$ projected into $\mathcal{T}$ and $\mathcal{T}^\perp$.
Thus, the convex combination of the projected matrix pair can be used to tune the value of $\mu_{\mathcal{Q}^\perp}$ and $\nu_{\mathcal{Q}^\perp}$ while holding the number of constraints constant.}
There are \fixes{$r(n_1+n_2-r)$} equality constraints generated by this process \fixes{so that when only the $\mathbf{A}^{(l)}$ matrices projected into $\mathcal{T}$ are used, they can form the $\mathcal{Q}$ vector space that covers all of $\mathcal{T}$.}

The uniform sampling was done by taking a random permutation of all the locations for an $n_1\times n_2$ matrix and using the first $m$ locations as the observed samples.
To increase (decrease) number of samples, the next (previous) locations in the permutation were simply added to (subtracted from) the existing observed samples.
Multiple random permutations were tested in parallel so that a sample probably of exact matrix completion could be calculated among them.
The probability of exact matrix completion for a given sample size was calculated as the fraction of random permutation sequences in which the solution to \eqref{eq:nuc_norm_prob} subtracted from $\mathbf{M}$ resulted in a Frobenius norm smaller that a specific tolerance.

The specifics of this simulation were under the following settings: $n_1=40$, $n_2=\fixes{10}$, $r=2$, \fixes{400} random permutations, and the maximum \fixes{relative} tolerance in the Forbenius norm to determine exact completion was set to $10^{-\fixes{3}}$.

\subsubsection{Results}

To demonstrate \fixes{how} the sample complexity \fixes{is affected by the relationship between the vector space $\mathcal{T}$ and the vector space of equality constraint matrices $\mathcal{Q}$, we indirectly varied $\mu_{\mathcal{Q}^\perp}$ and $\nu_{\mathcal{Q}^\perp}$ and} fixed 
the probability for exact matrix completion for a given set of linear equality constraints to \fixes{90}\%.
This was done by increasing the sample size until \fixes{90}\% of the \fixes{400} randomly permuted sequences of matrix locations each gave an exact matrix completion.
\fixes{This was done for each measurement pair $(\mu_{\mathcal{Q}^\perp},\nu_{\mathcal{Q}^\perp})$ that resulted from varying the convex combination between each $\mathbf{A}^{(l)}$ being projected into $\mathcal{T}$ and $\mathcal{T}^\perp$.}
The constrained nuclear norm minimization matrix completion method \eqref{eq:nuc_norm_prob} was tested against its unconstrained version.

\fixes{The number of samples needed versus $\mu_{\mathcal{Q}^\perp}$ and $\nu_{\mathcal{Q}^\perp}$ are shown in Figure \ref{fig:toyexample_40x10_mix32}.
It is easily observed that decreasing $\mu_{\mathcal{Q}^\perp}$ and $\nu_{\mathcal{Q}^\perp}$ simultaneously decreases the number of samples needed.
In fact, while the matrix completion case with $r(n_1+n_2-r)=96$ equality constraints and $\mu_{\mathcal{Q}^\perp}=\nu_{\mathcal{Q}^\perp}=1$ decreases the amount of samples needed by 11\% from the unconstrained case, the case where $\mu_{\mathcal{Q}^\perp}=\nu_{\mathcal{Q}^\perp}=0$ decreases the amount of samples needed by 76\%.}

\fixes{It is also interesting to note that if the equality constraints were naively thought of as just permanent diffused samples, then adding 96 linearly independent equality constraints would have an expected decrease of 96 samples which amounts to only a 36\% decrease from the unconstrained case.
This decrease in sample size is smaller than the case with $\mu_{\mathcal{Q}^\perp}\leq 0.97$ and $\nu_{\mathcal{Q}^\perp}\leq 0.54$ (see Figure \ref{fig:toyexample_40x10_mix32}).
Therefore, this shows that the advantages of added equality constraints comes more from how well the vector space spanned by $\mathbf{A}^{(1)},\dots,\mathbf{A}^{(h)}$ covers $\mathcal{T}$ than the number of equality constraints itself.
In other words, the added information from the equality constraints can reveal a portion of $\mathcal{T}$ a priori that replaces some of need for samples to learn $\mathcal{T}$.
Thus, the more of $\mathcal{T}$ that is revealed via the equality constraints as measured by $\mu_{\mathcal{Q}^\perp}$ and $\nu_{\mathcal{Q}^\perp}$, the less samples are needed.
From the perspective of a given set of equality constraints, this means that a reduction in needed samples can be significant or insignificant depending on the singular basis vectors $\mathbf{u}_1,\dots,\mathbf{u}_r$ and $\mathbf{v}_1,\dots,\mathbf{v}_r$ used to build $\mathcal{T}$ for underlying matrix $\mathbf{M}$.}

\begin{figure}
    \centering
    \subfigure[]{\includegraphics[width=0.95\columnwidth]{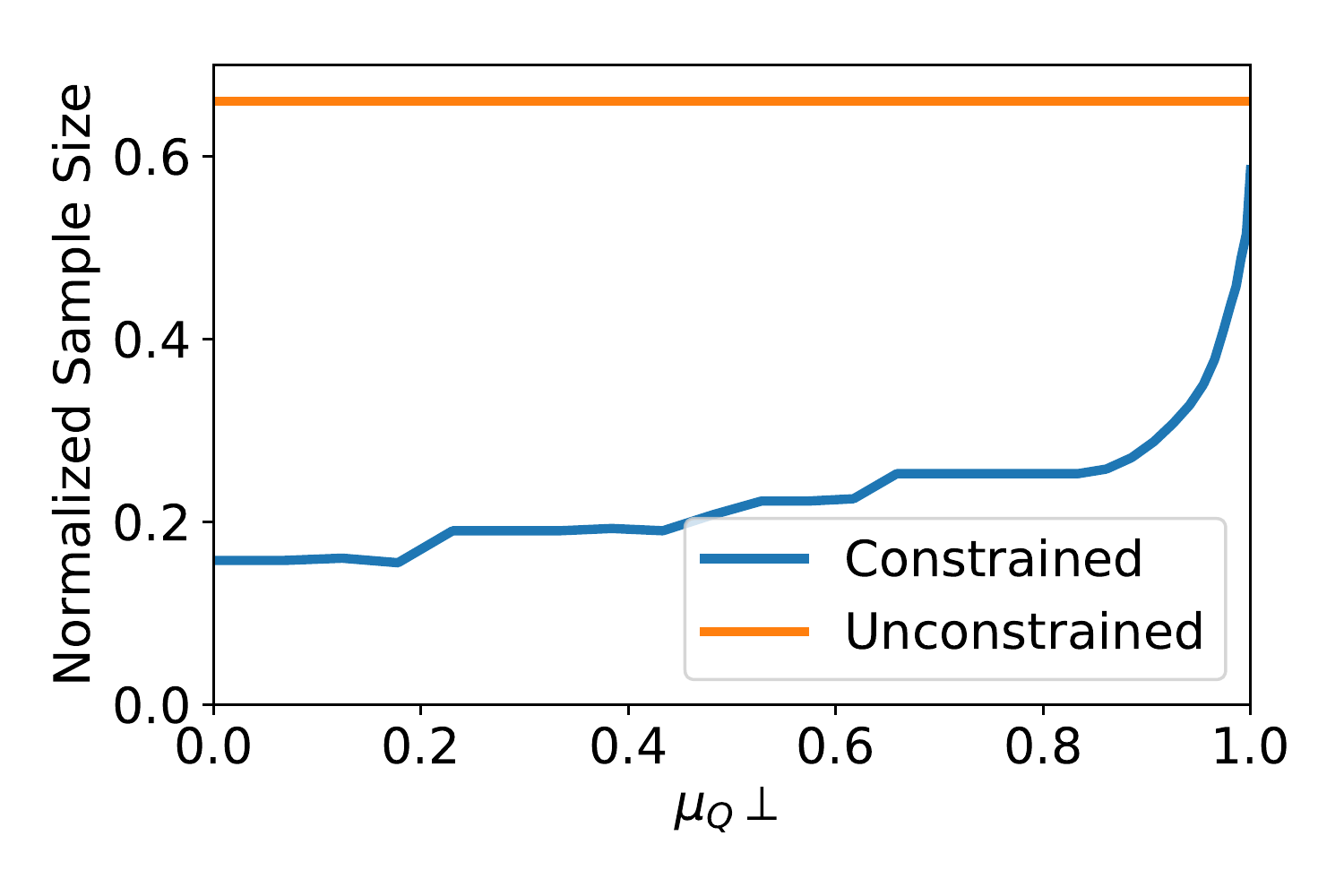}}
    \subfigure[]{\includegraphics[width=0.95\columnwidth]{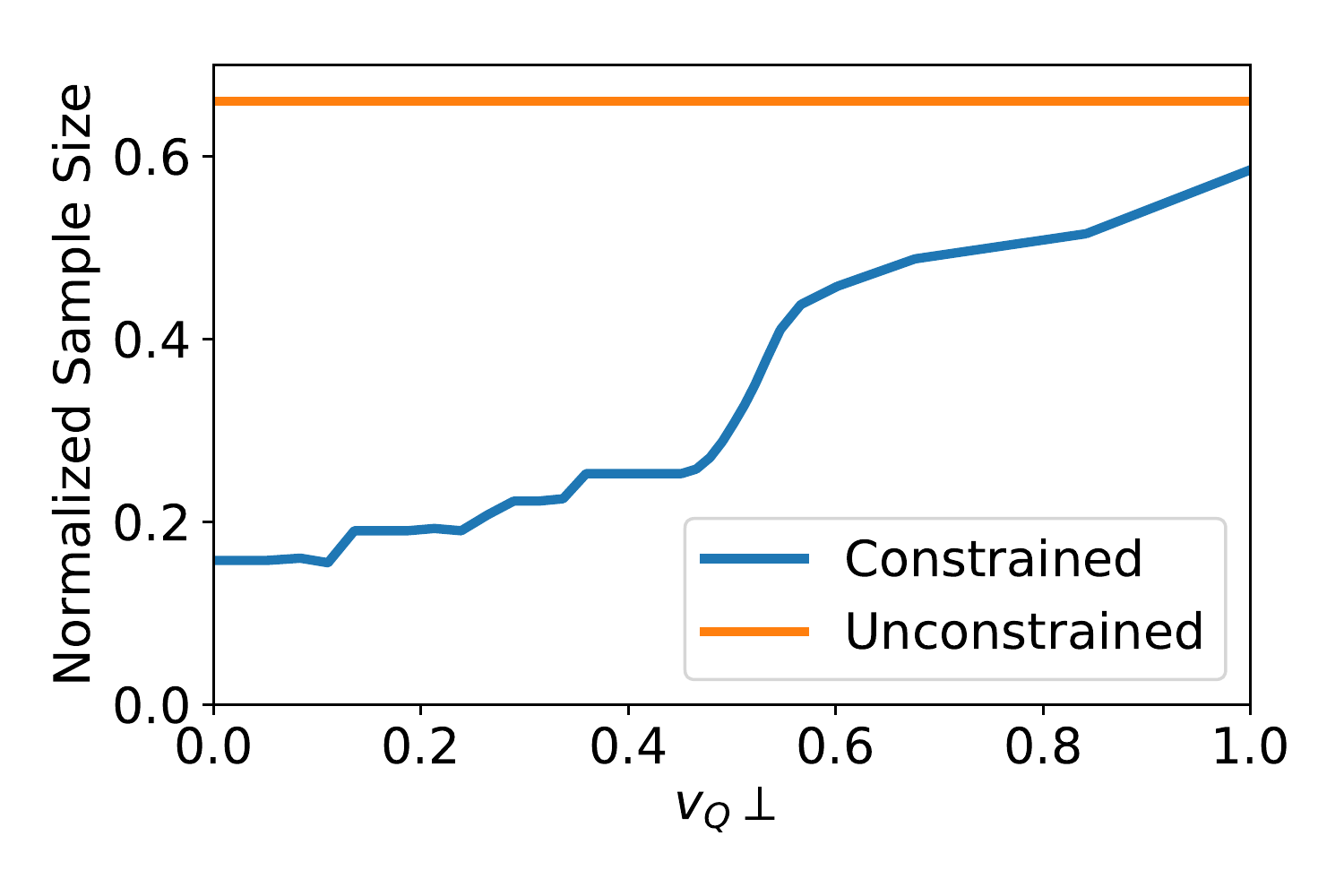}}
    \caption{Number of samples needed for exact completion with a probability of 0.9 normalized by the total number of elements $n_1n_2$ versus the constraint measurements (a) $\mu_{\mathcal{Q}^\perp}$ and (b) $\nu_{\mathcal{Q}^\perp}$ for both the constrained and unconstrained matrix completion problems.}
    \label{fig:toyexample_40x10_mix32}
\end{figure}

\subsection{Distribution Network}
\label{sec:perf_eval_distnet}

Using a power system emulator, our goal is to show how incorporating equality constraints based on the physics of the system can improve the accuracy for state estimation.

\subsubsection{Setup}


The distribution network data was created using MATPOWER~\cite{zimmerman2011matpower} on a 141 bus radial distribution network test case~\cite{khodr2008maximum}.
A diagram of the network is shown in Appendix \ref{sec:app_dist}.
The underlying matrix $\mathbf{M}$ that represents the state of the power system was formed according to the structure described in Section \ref{sec:state_est_prob} with all quantities being transferred into pu.
Therefore, the state matrix $\mathbf{M}$ has 281 rows and 17 columns.
The set of $4(n_\text{b}+n_\text{l})=1124$ linear equality constraints \eqref{eq:nuc_norm_prob_const} were formed according to the following linear power system equations: \eqref{eq:syseqn_line_pwrcons}-\eqref{eq:syseqn_ohms_law}.
An additional set of $n_\text{l}=140$ linear equality constraints were formed according to the linear approximation equations \eqref{eq:lin_approx_const}.


To sample the values of the state matrix $\mathbf{M}$, we set that Bus 1 and Bus 80 each have a PMU which can measure all 8 bus state values.
The remaining 139 buses and 140 lines were chosen uniformly at random to be measured.
When a bus was chosen, only the following 4 values were revealed: the real and reactive power injections, the magnitude of the voltage, and the magnitude of the current injection.
Therefore for both the voltage and current injections, the complex parts are never observed except for Buses 1 and 80 at the start.
When a line was chosen, only the following 5 values were revealed: real and reactive power injections into the line for both the From and To sides of the line, and the magnitude of the current flowing through the line.
Therefore, the real and reactive power losses and the complex current flow are never observed.
The number of samples in the figures refer to the number of lines and buses samples, not the number of data points taken.
For each sample size, 50 different random permutations of the buses and lines were used to do the uniform sampling, similar to Section \ref{sec:perf_eval_toy}.


Two different sets of constraints were tested in Problem \eqref{eq:nuc_norm_prob} for state estimation.
The first set (``w/ const") only includes the linear equality constraints derived from the linear power system model equations \eqref{eq:syseqn_line_pwrcons}-\eqref{eq:syseqn_ohms_law}.
The second set (``w/ const+appx"), are the linear equality constraints derived from the linear approximation equations \eqref{eq:lin_approx_const} and include the first set.
We also solved the Least Squares (LS) problem as a benchmark by replacing the Nuclear Norm with the Frobenius Norm in Problem \eqref{eq:nuc_norm_prob}.
To measure the estimation error, the Root Mean Squared Error (RMSE) was taken for voltage magnitude and voltage angle for the unmeasured values.
Because all other state quantities can be derived from the complex bus voltages and the physical properties of the power system equipment, our focus in these simulations is on the accuracy of the estimated complex voltages at the buses.
The estimated voltage angle is calculated by translating $\text{Re}(V_s)$ and $\text{Im}(V_s)$ from the estimated state matrix into polar form.
The estimated voltage magnitude $|V_s|$ is taken directly from the estimated state matrix.
The error is calculated by subtracting the estimation from the true value and are only of the unobserved matrix elements.

\begin{figure}[t]
    \subfigure[Voltage Magnitude]{\includegraphics[width=0.9\columnwidth]{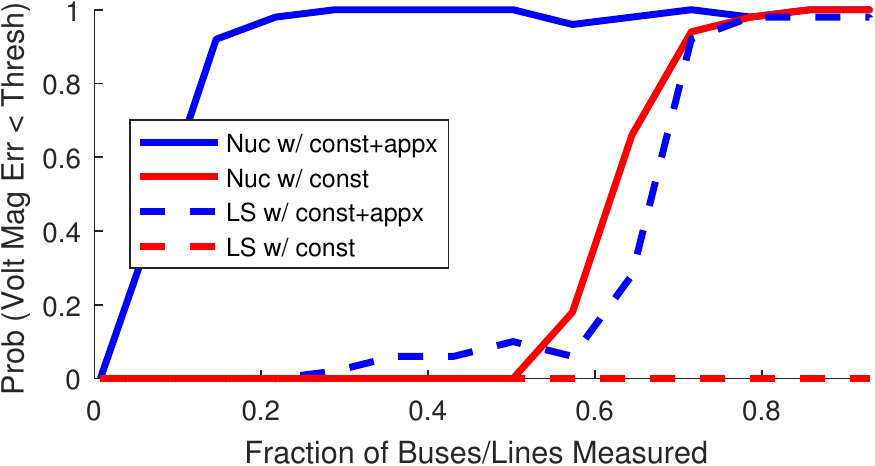}\label{fig:case141_sim03_bus_volterr_mag_prob_v_samp}}
    \subfigure[Voltage Angle]{\includegraphics[width=0.9\columnwidth]{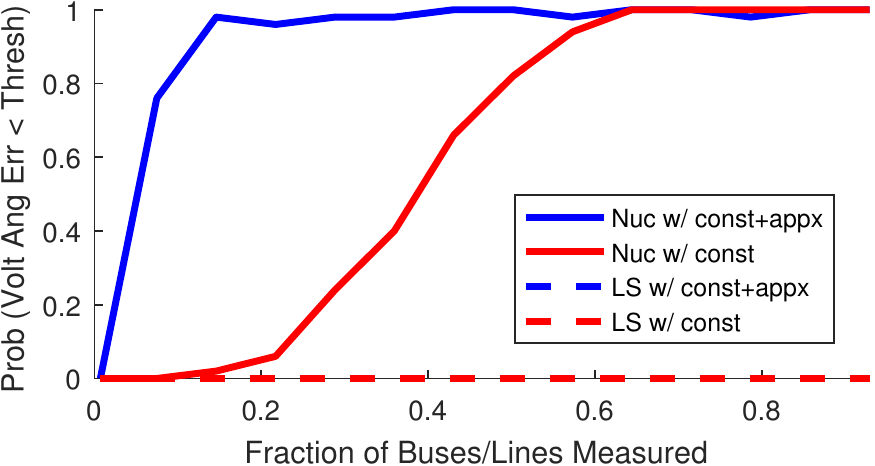}\label{fig:case141_sim03_bus_volterr_ang_prob_v_samp}}
    \caption{Probability of the estimated Voltage (a) Magnitude RMSE being below a threshold of $1\times 10^{-4}$ pu and (b) Angle RMSE being below a threshold of $5\times 10^{-5}$ degrees vs. fraction of buses and lines with observed data.}
    \label{fig:case141_sim03_bus_volterr_prob_v_samp}
\end{figure}

\begin{figure}[t]
    \subfigure[Voltage Magnitude]{\includegraphics[width=0.9\columnwidth]{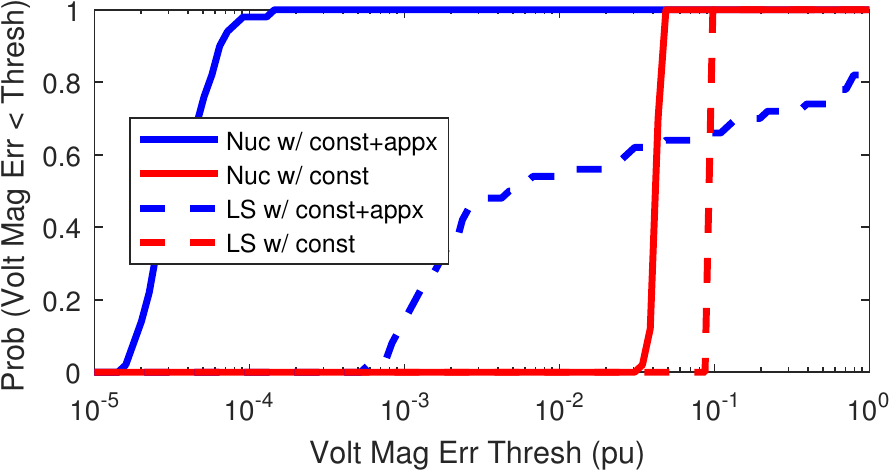}\label{fig:case141_sim03_bus_volterr_mag_prob_v_thresh}}
    \subfigure[Voltage Angle]{\includegraphics[width=0.9\columnwidth]{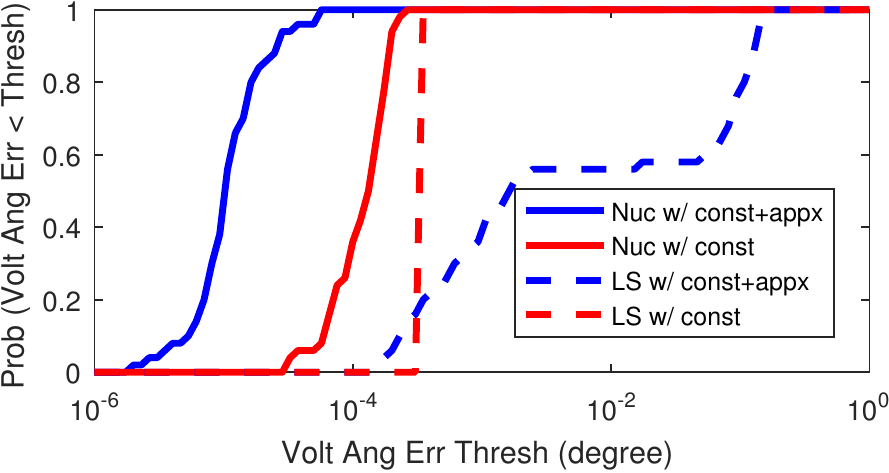}\label{fig:case141_sim03_bus_volterr_ang_prob_v_thresh}}
    \caption{Cumulative Distribution Functions of unmeasured voltage (a) magnitude and (b) angle estimation RMSE of the trials tested with 22\% of the buses and lines being measured.}
    \label{fig:case141_sim03_bus_volterr_prob_v_thresh}
\end{figure}

\begin{figure}[t]
    \subfigure[Voltage Magnitude]{\includegraphics[width=0.9\columnwidth]{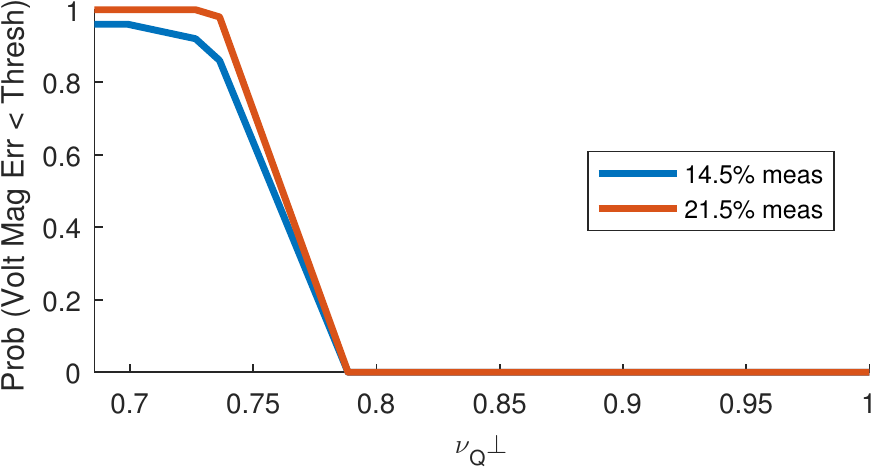}\label{fig:case141_sim03_bus_volterr_mag_prob_v_muA}}
    \subfigure[Voltage Angle]{\includegraphics[width=0.9\columnwidth]{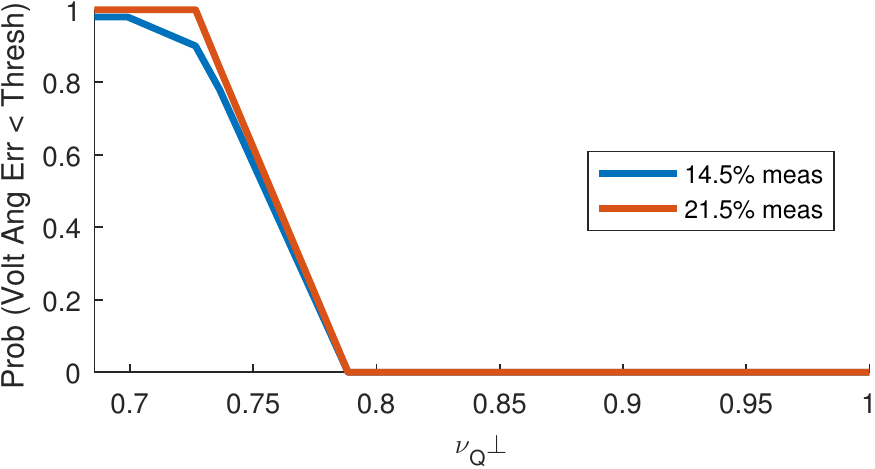}\label{fig:case141_sim03_bus_volterr_ang_prob_v_muA}}
    \caption{Probability of the estimated Voltage (a) Magnitude RMSE being below a threshold of $1\times 10^{-4}$ pu and (b) Angle RMSE being below a threshold of $5\times 10^{-5}$ degrees vs. $\nu_{\mathcal{Q}^\perp}$.}
    \label{fig:case141_sim03_bus_volterr_prob_v_muA}
\end{figure}

\subsubsection{Results}


To see the how the sample size affects the accuracy of the estimated voltages, we set RMSE thresholds and then counted the fraction of trials tested for each sample size that had RMSEs lower than the threshold.
Figure \ref{fig:case141_sim03_bus_volterr_prob_v_samp} shows the results for error thresholds of $1\times 10^{-4}$ pu and $5\times 10^{-5}$ degrees for voltage magnitude and voltage angle, respectively.
From these plots, we can make two strong obesrvations.
The first is that the Nuclear Norm method almost always has a higher probability of being more accurate than the Least Squares method for all sample sizes.
The second is that the linear approximation equations \eqref{eq:lin_approx_const} greatly improve the accuracy of the Nuclear Norm Minimization method to the point that even with only 20\% of the buses and lines measured, the unmeasured voltages have over a 90\% probability of having their average error be below $1\times 10^{-4}$ pu and $5\times 10^{-5}$ degrees.


To investigate deeper into the estimation error for a low-observability scenario, the estimation error Cumulative Distribution Functions (CDFs) for voltage magnitude and voltage angle are plotted in Figure \ref{fig:case141_sim03_bus_volterr_prob_v_thresh} when only 22\% of the buses and lines are measured.
They again show that the Least Squares method has magnitudes of greater error than that of the Nuclear Norm method.
However, it is interesting to note that with only the linear power system model equations \eqref{eq:syseqn_line_pwrcons}-\eqref{eq:syseqn_ohms_law}, the Nuclear Norm method and Least Squares method are within the same magnitude of error with Least Squares still being much more inaccurate.


To see how the value of \fixes{$\nu_{\mathcal{Q}^\perp}$} affects the state-estimation, we randomly deleted constraints from the ``const+appx" set and solved Problem \eqref{eq:nuc_norm_prob} while measuring \fixes{$\nu_{\mathcal{Q}^\perp}$.}
Figure \ref{fig:case141_sim03_bus_volterr_prob_v_muA} shows the results for the same error thresholds as before with two different sample sizes of 14.5\% and \fixes{21.5\%} of the buses and lines.
As constraints are added, the value of \fixes{$\nu_{\mathcal{Q}^\perp}$} decreases.
We can observe a threshold value of \fixes{$\nu_{\mathcal{Q}^\perp}$} at 0.79 before the added constraints help to increase the probability to have small error.
This gives evidence to the idea that the constraint set must achieve a small enough \fixes{$\nu_{\mathcal{Q}^\perp}$} before it can be fully utilized with a small sample size.

\section{Related Work}
In traditional state estimation, the focus is mainly on transmission networks that have an abundance of measurement equipment so that the focus is on how to remove bad data using weighted least-squares techniques~\cite{monticelli2000electric}.
Alternating direction method of multipliers was used to distribute the state estimation problem over control areas for a transmission system~\cite{kekatos2013distributed}.
A statistical method of adjusting an interpolation matrix that is used for estimating a dynamically changing state in PMU-unobservable areas is proposed by \cite{zhao2016power}.
For distribution networks that are measurement poor, these techniques cannot be used since they require full observability.
Much of the interest has instead been on estimating the topology of the network, especially during a contingency, and changes in power loss and voltage for capacitor decisions~\cite{baran2001,dehghanpour2018survey,primadianto2017review}.
With a similar motivation to our problem \cite{barukvcic2017evolutionary} uses an evolutionary optimization approach to estimate all of the voltages in a radial distribution network with as few measurements as possible.

Matrix Completion has only been recently considered for power system state estimation with the use of PMUs~\cite{wang2015low}.
The structure of our problem mainly follows that of \cite{donti2019matrix} to take advantage of power system physics to add information.
While we focus our analysis on a more theoretical perspective of sample complexity, \cite{donti2019matrix} uses a detailed simulations to measure estimation errors under different low-observablility scenarios.
The work of \cite{wang2015low,gao2016missing,genes2016recovering,genes2018robust} focuses on the time correlation of a single state variable type by using one of the matrix dimensions to represent time, as compared to the state variable type in our problem and \cite{donti2019matrix} to the focus on the correlation between state variable types at a single moment.

Much of our work in this paper borrowed the low-rank matrix completion theoretical framework from \cite{candes2009exact} which used the Bernoulli sampling model to bound the failure probability for uniform sampling.
Also based on this framework, \cite{recht2011simpler} proved a slightly different sample complexity bound with a more simple proof using uniform sampling with replacement instead of Bernoulli sample to bound the failure probability.
Linear equality constraints were used to convey information for matrix completion in \cite{recht2010guaranteed} instead of sampling.
However, compared to our problem, theirs modeled the constraints themselves as random instead of as a permanent feature of the matrix being completed.
In an environment with measurement errors, \cite{candes2011robust} uses a nuclear norm and L1-norm minimization problem to robustly decompose a measurement matrix with missing measurements into a low-rank matrix with the wanted information and a sparse matrix with the errors.
For the objective to deal separate the measurement matrix into low-rank and sparse matrices, \cite{mardani2013decentralized} replaces the nuclear norm with the Frobenius norm so that the optimization problem can be solved in a distributed manner.

\section{Conclusion}

In this paper, we develop a method for distribution network state estimation which has the characteristic of being \emph{underdetermined} as opposed to the traditional overdetermined state estimation problem found in transmission networks.
The method was adapted from low-rank matrix completion techniques but has the added benefit of taking advantage of the physical properties that govern a power system to reduce the number of samples needed to estimate the state.
Six case studies of distribution networks show that almost all of the state information can be encoded in a low-rank matrix.
The sample complexity for high-probability exact matrix completion was proved for the constrained matrix completion problem using nuclear norm minimization.
This shows how the additional information obtained from linear equality constraints can reduce the number of samples needed to exactly recover the underlying matrix.
The method was tested on a 141 bus distribution network test case and shows that the estimation error for voltage magnitude and angle at each bus can be significantly reduced with linear power system equations and linearized power system approximation equations.

There are three significant future research directions that can be taken from this paper.
First, this paper assumes that there is no error in the measurements made at the buses and lines and so incorporating measurement error management into the method is an important direction that could be done.
Second, state estimation for distribution networks would be continually repeated over time; thus, incorporating recent past measurements and estimations have potential to improve their accuracy.
Third, developing distributed algorithms for the constraint matrix completion is of interest.


\section{Acknowledgements}
\fixes{This work was authored in part by the National Renewable Energy Laboratory, managed and operated by Alliance for Sustainable Energy, LLC, for the U.S. Department of Energy (DOE) under Contract No. DE-AC36-
08GO28308. This work was supported in part by the Laboratory Directed Research and Development (LDRD) Program at NREL. The views expressed in the article do not necessarily represent the views of the DOE or the U.S. Government. The U.S. Government retains and the publisher, by accepting the article for publication, acknowledges that the U.S. Government retains a nonexclusive, paid-up, irrevocable, worldwide license to publish or reproduce the published form of this work, or allow others to do so, for
U.S. Government purposes. This work was also partially funded by NSF grants CNS-1464388, CNS-1617698, CNS-1730128, CNS-1717558, DGE-1633299.}

\bibliographystyle{IEEEtran}
\bibliography{references}

\appendix
\subsection{Proofs}
\label{sec:proofs}
The follow lemmas will be useful in proving the main lemmas that make up Theorem \ref{thm:samp_complex}.


\begin{lemma}
    \label{thm:bound_supremum}
    \fixes{Let $Y_1,\dots,Y_n$ be a sequence of independent random variables from a Banach space, and let $\mathcal{F}$ be a countable family of real valued functions such that if $f\in\mathcal{F}$ then $-f\in\mathcal{F}$.
    Define a specific value resulting from the sequence as
    \begin{align}
        Y_\text{sup}:= \sup_{f\in\mathcal{F}} \sum_{i=1}^n f(Y_i). \nonumber
    \end{align}
    If $|f|\leq B$ for positive constant $B$ and $\mathbb{E}(f(Y_i))=0$ for all $f\in\mathcal{F}$ and $i\in\{1,\dots,n\}$, then for any $t\geq 0$
    \begin{align}
        \mathbb{P}(|Y_\text{sup}-\mathbb{E}(Y_\text{sup})| > t) \quad\quad\quad\quad\quad\quad\quad\quad\quad\quad\quad\quad\quad\quad \nonumber \\ \leq 3 \exp\left(-\frac{t}{KB}\log\left(1+\frac{Bt}{\sigma^2+B\mathbb{E}(Y_\text{sup})}\right)\right) \nonumber
    \end{align}
    where $K$ is a numerical constant and $\sigma^2:=\sup_{f\in\mathcal{F}} \sum_{i=1}^n \mathbb{E}((f(Y_i))^2)$.}
\end{lemma}

\fixes{The proof can be found in \cite{talagrand1996new}.}


\begin{lemma}
    \label{thm:bound_expectation}
    \fixes{Suppose that if $\Omega$ is sampled according to the Bernoulli model \eqref{eq:bernoulli}, $n_1\geq n_2$, and  Assumption \ref{ass:coherence} is satisfied, then there exists a constant $C'_R$ such that
    \begin{align}
        \mathbb{E}\left(\frac{1}{p}\left\|\mathcal{P}_\mathcal{T}\left(\mathcal{P}_\Omega-p\mathbf{I}\right)\mathcal{P}_\mathcal{T}\right\|\right) \leq C'_R \sqrt{\frac{\mu_0 r n_1 \log n_1}{m}} \nonumber
    \end{align} as long as the RHS is smaller than 1.}
\end{lemma}

\fixes{The proof can be found in \cite{candes2007sparsity}.}


\begin{lemma}
    \label{thm:bound_opnormH}
    Suppose that if $\Omega$ is sampled according to the Bernoulli model \eqref{eq:bernoulli}, $n_1\geq n_2$, and Assumption \ref{ass:coherence} is satisfied, then there is a numerical constant $C_R$ such that for all $\beta>1$,
    \begin{align}
        \frac{1}{p\fixes{+q}}\left\|\mathcal{P}_\mathcal{T}\left(\mathcal{P}_\Omega\fixes{+q\mathcal{P}_\mathcal{Q}}\right)\mathcal{P}_\mathcal{T}-(p\fixes{+q})\mathcal{P}_\mathcal{T}\right\| \quad\quad\quad \quad\quad\quad \nonumber \\
        \leq \frac{C_R \sqrt{\mu_0\beta rn_1\log n_1}\fixes{+n_1n_2\sqrt{\mu_{\mathcal{Q}^\perp}\mu_0q}}}{\sqrt{m\fixes{+qn_1n_2}}} \nonumber
    \end{align}
    with probability at least $1-3n_1^{-\beta}$ as long as
    \begin{align}
        C_R \sqrt{\frac{\mu_0\beta rn_1\log n_1}{m\fixes{+qn_1n_2}}}<1 \quad \fixes{\text{and} \quad \frac{\mu_{\mathcal{Q}^\perp}\mu_0qn_1^2n_2^2}{m+qn_1n_2}\leq 1.} \nonumber
    \end{align}
\end{lemma}

\begin{proof}
    \fixes{First, by the triangle inequality we have
    \begin{align}
        \frac{1}{p+q}\left\|\mathcal{P}_\mathcal{T}\left(\mathcal{P}_\Omega+q\mathcal{P}_\mathcal{Q}\right)\mathcal{P}_\mathcal{T}-(p+q)\mathcal{P}_\mathcal{T}\right\| \quad\quad\quad \quad\quad\quad \nonumber
    \end{align}
    \begin{align}
        & \leq \frac{1}{p+q}\left(\left\|\mathcal{P}_\mathcal{T}\left(\mathcal{P}_\Omega-p\mathbf{I}\right)\mathcal{P}_\mathcal{T}\right\| + q\left\|\mathcal{P}_\mathcal{T}\left(\mathcal{P}_\mathcal{Q}-\mathbf{I}\right)\mathcal{P}_\mathcal{T}\right\|\right) \nonumber \\
        & = \frac{1}{p+q}\left(\left\|\mathcal{P}_\mathcal{T}\left(\mathcal{P}_\Omega-p\mathbf{I}\right)\mathcal{P}_\mathcal{T}\right\| + q\left\|\mathcal{P}_\mathcal{T}\mathcal{P}_{\mathcal{Q}^\perp}\mathcal{P}_\mathcal{T}\right\|\right) \label{eq:Hnorm_2terms}
    \end{align}
    and will analyze each term independently.}

    \fixes{Any matrix $\mathbf{X}\in\mathbb{R}^{n_1\times n_2}$ can be written as
    \begin{align}
        \mathbf{X}=\sum_{i=1}^{n_1}\sum_{j=1}^{n_2}\langle \mathbf{X},\mathbf{e}_i\mathbf{e}_j^\intercal \rangle \mathbf{e}_i\mathbf{e}_j^\intercal. \nonumber
    \end{align}
    Then step-by-step, we form the inside of the first term
    \begin{align}
        \mathcal{P}_\mathcal{T}\left(\mathcal{P}_\Omega-p\mathbf{I}\right)\mathcal{P}_\mathcal{T}(\mathbf{X}) \quad\quad\quad\quad\quad\quad\quad\quad\quad\quad\quad\quad\quad\quad\quad \nonumber
    \end{align}
    \begin{align}
        & =\mathcal{P}_\mathcal{T}\left(\mathcal{P}_\Omega-p\mathbf{I}\right)\mathcal{P}_\mathcal{T}\sum_{i=1}^{n_1}\sum_{j=1}^{n_2}\langle \mathbf{X},\mathbf{e}_i\mathbf{e}_j^\intercal \rangle \mathbf{e}_i\mathbf{e}_j^\intercal \nonumber \\
        & =\mathcal{P}_\mathcal{T}\left(\mathcal{P}_\Omega-p\mathbf{I}\right)\sum_{i=1}^{n_1}\sum_{j=1}^{n_2}\langle \mathcal{P}_\mathcal{T}(\mathbf{X}),\mathbf{e}_i\mathbf{e}_j^\intercal \rangle \mathbf{e}_i\mathbf{e}_j^\intercal \nonumber \\
        & =\mathcal{P}_\mathcal{T}\sum_{i=1}^{n_1}\sum_{j=1}^{n_2}(\delta_{ij}-p)\langle \mathcal{P}_\mathcal{T}(\mathbf{X}),\mathbf{e}_i\mathbf{e}_j^\intercal \rangle \mathbf{e}_i\mathbf{e}_j^\intercal \nonumber \\
        & =\sum_{i=1}^{n_1}\sum_{j=1}^{n_2}(\delta_{ij}-p)\langle \mathcal{P}_\mathcal{T}(\mathbf{X}),\mathbf{e}_i\mathbf{e}_j^\intercal \rangle \mathcal{P}_\mathcal{T}(\mathbf{e}_i\mathbf{e}_j^\intercal) \nonumber \\
        & =\sum_{i=1}^{n_1}\sum_{j=1}^{n_2}(\delta_{ij}-p)\langle \mathbf{X},\mathcal{P}_\mathcal{T}(\mathbf{e}_i\mathbf{e}_j^\intercal) \rangle \mathcal{P}_\mathcal{T}(\mathbf{e}_i\mathbf{e}_j^\intercal) \label{eq:Hnorm_2terms1}
    \end{align}
    and likewise, form the inside of the second term
    \begin{align}
        \mathcal{P}_\mathcal{T}\mathcal{P}_{\mathcal{Q}^\perp}\mathcal{P}_\mathcal{T}(\mathbf{X}) \quad\quad\quad\quad\quad\quad\quad\quad\quad\quad\quad\quad\quad\quad \nonumber
    \end{align}
    \begin{align}
        & = \mathcal{P}_\mathcal{T}\mathcal{P}_{\mathcal{Q}^\perp}\mathcal{P}_\mathcal{T}\sum_{i=1}^{n_1}\sum_{j=1}^{n_2}\langle \mathbf{X},\mathbf{e}_i\mathbf{e}_j^\intercal \rangle \mathbf{e}_i\mathbf{e}_j^\intercal \nonumber \\
        & = \mathcal{P}_\mathcal{T}\mathcal{P}_{\mathcal{Q}^\perp}\mathcal{P}_{\mathcal{Q}^\perp}\mathcal{P}_\mathcal{T}\sum_{i=1}^{n_1}\sum_{j=1}^{n_2}\langle \mathbf{X},\mathbf{e}_i\mathbf{e}_j^\intercal \rangle \mathbf{e}_i\mathbf{e}_j^\intercal \nonumber \\
        & = \mathcal{P}_\mathcal{T}\mathcal{P}_{\mathcal{Q}^\perp}\sum_{i=1}^{n_1}\sum_{j=1}^{n_2}\langle \mathcal{P}_{\mathcal{Q}^\perp}\mathcal{P}_\mathcal{T}(\mathbf{X}),\mathbf{e}_i\mathbf{e}_j^\intercal \rangle \mathbf{e}_i\mathbf{e}_j^\intercal \nonumber \\
        & = \sum_{i=1}^{n_1}\sum_{j=1}^{n_2}\langle \mathcal{P}_{\mathcal{Q}^\perp}\mathcal{P}_\mathcal{T}(\mathbf{X}),\mathbf{e}_i\mathbf{e}_j^\intercal \rangle \mathcal{P}_\mathcal{T}\mathcal{P}_{\mathcal{Q}^\perp}(\mathbf{e}_i\mathbf{e}_j^\intercal) \nonumber \\
        & = \sum_{i=1}^{n_1}\sum_{j=1}^{n_2}\langle \mathbf{X},\mathcal{P}_\mathcal{T}\mathcal{P}_{\mathcal{Q}^\perp}(\mathbf{e}_i\mathbf{e}_j^\intercal) \rangle \mathcal{P}_\mathcal{T}\mathcal{P}_{\mathcal{Q}^\perp}(\mathbf{e}_i\mathbf{e}_j^\intercal) \label{eq:Hnorm_2terms2}
    \end{align}
    Both terms in \eqref{eq:Hnorm_2terms} will be analyzed separately using \eqref{eq:Hnorm_2terms1} and \eqref{eq:Hnorm_2terms2}, and then applied back together for the lemma's statement.}
    
    \fixes{Before analyzing them, from \eqref{eq:PT_defn} we have that
    \begin{align}
        \mathcal{P}_\mathcal{T}(\mathbf{e}_i\mathbf{e}_j^\intercal) & = \mathbf{P}_\mathcal{U}(\mathbf{e}_i\mathbf{e}_j^\intercal) + (\mathbf{e}_i\mathbf{e}_j^\intercal)\mathbf{P}_\mathcal{V} - \mathbf{P}_\mathcal{U}(\mathbf{e}_i\mathbf{e}_j^\intercal)\mathbf{P}_\mathcal{V} \nonumber \\
        & = (\mathbf{P}_\mathcal{U}\mathbf{e}_i)\mathbf{e}_j^\intercal + \mathbf{e}_i(\mathbf{P}_\mathcal{V}\mathbf{e}_j)^\intercal - (\mathbf{P}_\mathcal{U}\mathbf{e}_i)(\mathbf{P}_\mathcal{V}\mathbf{e}_j)^\intercal \nonumber
    \end{align}
    which allows for the following upper bound on its Frobenius norm
    \begin{align}
        \|\mathcal{P}_\mathcal{T}(\mathbf{e}_i\mathbf{e}_j^\intercal)\|_F^2 & = \langle \mathcal{P}_\mathcal{T}(\mathbf{e}_i\mathbf{e}_j^\intercal),\mathcal{P}_\mathcal{T}(\mathbf{e}_i\mathbf{e}_j^\intercal) \rangle \nonumber \\
        & = \langle \mathcal{P}_\mathcal{T}(\mathbf{e}_i\mathbf{e}_j^\intercal),\mathbf{e}_i\mathbf{e}_j^\intercal \rangle \nonumber \\
        & = \langle (\mathbf{P}_\mathcal{U}\mathbf{e}_i)\mathbf{e}_j^\intercal,\mathbf{e}_i\mathbf{e}_j^\intercal \rangle + \langle \mathbf{e}_i(\mathbf{P}_\mathcal{V}\mathbf{e}_j)^\intercal,\mathbf{e}_i\mathbf{e}_j^\intercal \rangle \nonumber \\
        & \quad\quad - \langle (\mathbf{P}_\mathcal{U}\mathbf{e}_i)(\mathbf{P}_\mathcal{V}\mathbf{e}_j)^\intercal,\mathbf{e}_i\mathbf{e}_j^\intercal \rangle \nonumber \\
        & = \langle (\mathbf{P}_\mathcal{U}\mathbf{e}_i),\mathbf{e}_i\mathbf{e}_j^\intercal\mathbf{e}_j \rangle + \langle  {e}_j^\intercal,(\mathbf{P}_\mathcal{V}\mathbf{e}_j)\mathbf{e}_i^\intercal\mathbf{e}_i \rangle \nonumber \\
        & \quad\quad - \langle (\mathbf{P}_\mathcal{U}\mathbf{e}_i)(\mathbf{P}_\mathcal{V}\mathbf{e}_j)^\intercal,\mathbf{P}_\mathcal{U}\mathbf{e}_i\mathbf{e}_j^\intercal \rangle \nonumber \\
        & = \langle (\mathbf{P}_\mathcal{U}\mathbf{e}_i),\mathbf{e}_i \rangle + \langle  {e}_j^\intercal,(\mathbf{P}_\mathcal{V}\mathbf{e}_j) \rangle \nonumber \\
        & \quad\quad - \langle \mathbf{e}_j,(\mathbf{P}_\mathcal{V}\mathbf{e}_j)(\mathbf{P}_\mathcal{U}\mathbf{e}_i)^\intercal(\mathbf{P}_\mathcal{U}\mathbf{e}_i) \rangle \nonumber \\
        & = \langle (\mathbf{P}_\mathcal{U}\mathbf{e}_i),(\mathbf{P}_\mathcal{U}\mathbf{e}_i) \rangle + \langle  (\mathbf{P}_\mathcal{V}\mathbf{e}_j),(\mathbf{P}_\mathcal{V}\mathbf{e}_j) \rangle \nonumber \\
        & \quad\quad - \|\mathbf{P}_\mathcal{U}\mathbf{e}_i\|_F^2\langle \mathbf{e}_j,(\mathbf{P}_\mathcal{V}\mathbf{e}_j) \rangle \nonumber \\
        & = \|\mathbf{P}_\mathcal{U}\mathbf{e}_i\|_F^2 + \|\mathbf{P}_\mathcal{V}\mathbf{e}_j\|_F^2 \nonumber \\
        & \quad\quad - \|\mathbf{P}_\mathcal{U}\mathbf{e}_i\|_F^2\langle (\mathbf{P}_\mathcal{V}\mathbf{e}_j),(\mathbf{P}_\mathcal{V}\mathbf{e}_j) \rangle \nonumber \\
        & = \|\mathbf{P}_\mathcal{U}\mathbf{e}_i\|_F^2 + \|\mathbf{P}_\mathcal{V}\mathbf{e}_j\|_F^2 - \|\mathbf{P}_\mathcal{U}\mathbf{e}_i\|_F^2 \|\mathbf{P}_\mathcal{V}\mathbf{e}_j\|_F^2 \nonumber \\
        & \leq \|\mathbf{P}_\mathcal{U}\mathbf{e}_i\|_F^2 + \|\mathbf{P}_\mathcal{V}\mathbf{e}_j\|_F^2 \nonumber \\
        & \leq \frac{\mu(\mathcal{U})r}{n_1}  + \frac{\mu(\mathcal{V})r}{n_2} \leq 2\frac{\mu_0r}{n_2} \label{eq:PTeiej_mu0}
    \end{align}
    where the fourth and fifth equalities come from applying the cyclic property of the trace operator.}
    
    \fixes{Bounding the first term of \eqref{eq:Hnorm_2terms} at high probability will come from the application of Lemma \ref{thm:bound_supremum} where $Y_{ij}:=\delta_{ij}-p$ and $f(Y_{ij}):=\frac{1}{p+q}Y_{ij}\langle \mathbf{X}_1,\mathcal{P}_\mathcal{T}(\mathbf{e}_i\mathbf{e}_j^\intercal) \rangle \langle \mathcal{P}_\mathcal{T}(\mathbf{e}_i\mathbf{e}_j^\intercal), \mathbf{X}_2 \rangle$ for some $\mathbf{X}_1:\|\mathbf{X}_1\|_F\leq 1$ and $\mathbf{X}_2:\|\mathbf{X}_2\|_F\leq 1$.
    At this point, the condition $\mathbb{E}(f(Y_{ij}))=0$ is already satisfied.
    We can rewrite the first term of \eqref{eq:Hnorm_2terms} with \eqref{eq:Hnorm_2terms1} as
    \begin{align}
        \frac{1}{p+q}\left\|\mathcal{P}_\mathcal{T}\left(\mathcal{P}_\Omega-p\mathbf{I}\right)\mathcal{P}_\mathcal{T}\right\| \quad\quad\quad\quad\quad\quad\quad\quad\quad\quad\quad\quad \nonumber
    \end{align}
    \begin{align}
         & =\sup \frac{\left\langle \mathbf{X}_1, \sum_{i=1}^{n_1}\sum_{j=1}^{n_2}(\delta_{ij}-p)\langle \mathbf{X}_2,\mathcal{P}_\mathcal{T}(\mathbf{e}_i\mathbf{e}_j^\intercal)\rangle \mathcal{P}_\mathcal{T}(\mathbf{e}_i\mathbf{e}_j^\intercal) \right\rangle}{p+q}  \nonumber \\
         & =\sup \sum_{i=1}^{n_1}\sum_{j=1}^{n_2} \frac{(\delta_{ij}-p)}{p+q}  \langle \mathbf{X}_1, \mathcal{P}_\mathcal{T}(\mathbf{e}_i\mathbf{e}_j^\intercal)\rangle \langle \mathbf{X}_2,\mathcal{P}_\mathcal{T}(\mathbf{e}_i\mathbf{e}_j^\intercal)\rangle \nonumber \\
         & =: Y_\text{sup} \nonumber
    \end{align}
    where the supremum is on $\mathbf{X}_1:\|\mathbf{X}_1\|_F\leq 1$ and $\mathbf{X}_2:\|\mathbf{X}_2\|_F\leq 1$.
    To determine $B$, we have
    \begin{align}
        |f(\delta_{ij}-p)| & = \frac{|\delta_{ij}-p|}{p+q} |\langle \mathbf{X}_1, \mathcal{P}_\mathcal{T}(\mathbf{e}_i\mathbf{e}_j^\intercal)\rangle| |\langle \mathbf{X}_2,\mathcal{P}_\mathcal{T}(\mathbf{e}_i\mathbf{e}_j^\intercal)\rangle| \nonumber \\
        & \leq \frac{|\delta_{ij}-p|}{p+q} \|\mathcal{P}_\mathcal{T}(\mathbf{e}_i\mathbf{e}_j^\intercal)\|_F^2 \|\mathbf{X}_1\|_F \|\mathbf{X}_1\|_F \nonumber \\
        & \leq \frac{|\delta_{ij}-p|}{p+q} \|\mathcal{P}_\mathcal{T}(\mathbf{e}_i\mathbf{e}_j^\intercal)\|_F^2 \nonumber \\
        & \leq \frac{1}{p+q} \|\mathcal{P}_\mathcal{T}(\mathbf{e}_i\mathbf{e}_j^\intercal)\|_F^2 \nonumber \\
        & \leq \frac{2r\mu_0}{(p+q)n_2} \nonumber \\
        & = \frac{2r\mu_0n_1}{m+qn_1n_2} =: B \nonumber
    \end{align}
    where the fourth inequality comes from applying \eqref{eq:PTeiej_mu0} and the last equality is from $p=\frac{m}{n_1n_2}$.
    Lastly, $\sigma^2$ needs to be determined.
    Note that second moment of $(\delta_{ij}-p)$ is $\mathbb{E}((\delta_{ij}-p)^2) = \mathbb{E}(\delta_{ij}^2-2\delta_{ij}+p^2) = p(1-p)$.
    The summand inside the definition of $\sigma^2$ is
    \begin{align}
        \mathbb{E}((f(\delta_{ij}-p))^2) \quad\quad\quad\quad\quad\quad\quad\quad\quad\quad\quad\quad\quad\quad\quad\quad\quad \nonumber
    \end{align}
    \begin{align}
        & = \frac{\mathbb{E}((\delta_{ij}-p)^2)}{(p+q)^2} \langle \mathbf{X}_1, \mathcal{P}_\mathcal{T}(\mathbf{e}_i\mathbf{e}_j^\intercal)\rangle^2 \langle \mathbf{X}_2,\mathcal{P}_\mathcal{T}(\mathbf{e}_i\mathbf{e}_j^\intercal)\rangle^2 \nonumber \\
        & = \frac{p(1-p)}{(p+q)^2} \langle \mathbf{X}_1, \mathcal{P}_\mathcal{T}(\mathbf{e}_i\mathbf{e}_j^\intercal)\rangle^2 \langle \mathbf{X}_2,\mathcal{P}_\mathcal{T}(\mathbf{e}_i\mathbf{e}_j^\intercal)\rangle^2 \nonumber \\
        & = \frac{p(1-p)}{(p+q)^2} \langle \mathbf{X}_1, \mathcal{P}_\mathcal{T}(\mathbf{e}_i\mathbf{e}_j^\intercal)\rangle^2 \langle \mathcal{P}_\mathcal{T}(\mathbf{X}_2),\mathbf{e}_i\mathbf{e}_j^\intercal \rangle^2 \nonumber \\
        & \leq \frac{p(1-p)}{(p+q)^2} \|\mathcal{P}_\mathcal{T}(\mathbf{e}_i\mathbf{e}_j^\intercal)\|_F^2 \|\mathbf{X}_1\|_F^2 \langle \mathcal{P}_\mathcal{T}(\mathbf{X}_2),\mathbf{e}_i\mathbf{e}_j^\intercal \rangle^2 \nonumber \\
        & \leq \frac{p(1-p)}{(p+q)^2} \|\mathcal{P}_\mathcal{T}(\mathbf{e}_i\mathbf{e}_j^\intercal)\|_F^2 \langle \mathcal{P}_\mathcal{T}(\mathbf{X}_2),\mathbf{e}_i\mathbf{e}_j^\intercal \rangle^2 \nonumber \\
        & \leq \frac{p}{(p+q)^2} \|\mathcal{P}_\mathcal{T}(\mathbf{e}_i\mathbf{e}_j^\intercal)\|_F^2 \langle \mathcal{P}_\mathcal{T}(\mathbf{X}_2),\mathbf{e}_i\mathbf{e}_j^\intercal \rangle^2 \nonumber \\
        & \leq \frac{2r\mu_0}{n_2}\frac{p}{(p+q)^2} \langle \mathcal{P}_\mathcal{T}(\mathbf{X}_2),\mathbf{e}_i\mathbf{e}_j^\intercal \rangle^2 \nonumber \\
        & \leq \frac{2r\mu_0}{n_2}\frac{1}{p+q} \langle \mathcal{P}_\mathcal{T}(\mathbf{X}_2),\mathbf{e}_i\mathbf{e}_j^\intercal \rangle^2 \nonumber \\
        & = \frac{2r\mu_0n_1}{m+qn_1n_2} \langle \mathcal{P}_\mathcal{T}(\mathbf{X}_2),\mathbf{e}_i\mathbf{e}_j^\intercal \rangle^2 \nonumber
    \end{align}
    where the fourth inequality comes from applying \eqref{eq:PTeiej_mu0}.
    To determine an upper bound on $\sigma^2$ we have
    \begin{align}
        \sigma^2 & = \sup \sum_{i=1}^{n_1}\sum_{j=1}^{n_2} \mathbb{E}((f(\delta_{ij}-p))^2) \nonumber \\
        & \leq \sup \frac{2r\mu_0n_1}{m+qn_1n_2}\sum_{i=1}^{n_1}\sum_{j=1}^{n_2}\langle \mathcal{P}_\mathcal{T}(\mathbf{X}_2),\mathbf{e}_i\mathbf{e}_j^\intercal \rangle^2 \nonumber \\
        & = \sup \frac{2r\mu_0n_1}{m+qn_1n_2}\|\mathcal{P}_\mathcal{T}(\mathbf{X}_2)\|_F^2 \nonumber \\
        & \leq \frac{2r\mu_0n_1}{m+qn_1n_2}. \nonumber
    \end{align}
    From Lemma \ref{thm:bound_expectation} and multiplying it by $\frac{p}{p+q}$, we bound the expectation of the first term of \eqref{eq:Hnorm_2terms} to be less than 1 which also satisfies the condition of Lemma \ref{thm:bound_expectation}
    \begin{align}
        \mathbb{E}\left(\frac{1}{p+q}\left\|\mathcal{P}_\mathcal{T}\left(\mathcal{P}_\Omega-p\mathbf{I}\right)\mathcal{P}_\mathcal{T}\right\|\right) \quad\quad\quad\quad\quad\quad\quad\quad\quad \nonumber
    \end{align}
    \begin{align}
        & \leq  C'_R \frac{p}{p+q} \sqrt{\frac{\mu_0 r n_1 \log n_1}{m}} \nonumber \\
        & =  C'_R \frac{m}{m+qn_1n_2} \sqrt{\frac{\mu_0 r n_1 \log n_1}{m}} \nonumber \\
        & =  C'_R \sqrt{\frac{m}{m+qn_1n_2}} \sqrt{\frac{\mu_0 r n_1 \log n_1}{m+qn_1n_2}} \nonumber \\
        & \leq  C'_R \sqrt{\frac{\mu_0 r n_1 \log n_1}{m+qn_1n_2}} \label{eq:Hnorm_2terms1_expecation} \\
        & = \left(C_R-\sqrt{\frac{4K}{\log 2}}\right) \sqrt{\frac{\mu_0 r n_1 \log n_1}{m+qn_1n_2}} \nonumber \\
        & \leq C_R \sqrt{\frac{\mu_0 r n_1 \log n_1}{m+qn_1n_2}} < 1 \nonumber 
    \end{align}
    where the third equality comes from setting $C'_R:=C_R - \sqrt{\frac{4K}{\log 2}}$ which will be used again later and the last line comes from the assumption in the lemma's statement for any $\beta>1$.}
    
    \fixes{Using the above results, we apply Lemma \ref{thm:bound_supremum} to bound the distance the first term of \eqref{eq:Hnorm_2terms} is from its expected value.
    \begin{align}
        \mathbb{P}(|Y_\text{sup}-\mathbb{E}(Y_\text{sup})| > t) \quad\quad\quad\quad\quad\quad\quad\quad\quad\quad\quad\quad\quad\quad \nonumber
    \end{align}
    \begin{align}
        & \leq 3 \exp\left(-\frac{t}{KB}\log\left(1+\frac{Bt}{\sigma^2+B\mathbb{E}(Y_\text{sup})}\right)\right) \nonumber \\
        & \leq 3 \exp\left(-\frac{t \log 2}{KB}\min\left\{1,\frac{Bt}{\sigma^2+B\mathbb{E}(Y_\text{sup})}\right\}\right) \nonumber \\
        & \leq 3 \exp\left(-\frac{t(m+qn_1n_2) \log 2}{2Kr\mu_0n_1}\min\left\{1,\frac{t}{2}\right\}\right) \nonumber
    \end{align}
    where the second inequality comes from the fact for any $u\geq 0$, we have $\log (1+u)\geq \log 2 \min\{1,u\}$.
    Set $t:=\sqrt{\left(\frac{4K}{\log 2}\right)\frac{\beta r \mu_0 n_1 \log n_1}{m+qn_1n_2}}$ to get the following
    \begin{align}
        \mathbb{P}\left(|Y_\text{sup}-\mathbb{E}(Y_\text{sup})| > \sqrt{\left(\frac{4K}{\log 2}\right)\frac{\beta r \mu_0 n_1 \log n_1}{m+qn_1n_2}}\right) \quad\quad \nonumber
    \end{align}
    \begin{align}
        & \leq 3 \exp\left(-\min\left\{\sqrt{\frac{\log 2}{K}\frac{\beta (m+qn_1n_2) \log n_1}{r \mu_0 n_1}},\beta \log n_1\right\}\right) \nonumber \\
        & = 3 n_1^{-\beta} \label{eq:Hnorm_2terms1_concentration}
    \end{align}
    where the equality is sufficiently true when we have $C_R = C'_R + \sqrt{\frac{4K}{\log 2}}$ and assume $C_R \sqrt{\frac{\mu_0\beta rn_1\log n_1}{m\fixes{+qn_1n_2}}}<1$ from the lemma's statement.}
    
    \fixes{Bringing together the bound on the expectation \eqref{eq:Hnorm_2terms1_expecation} and concentration \eqref{eq:Hnorm_2terms1_concentration} of the first term of \eqref{eq:Hnorm_2terms}, we have
    \begin{align}
        \frac{1}{p+q}\left\|\mathcal{P}_\mathcal{T}\left(\mathcal{P}_\Omega-p\mathbf{I}\right)\mathcal{P}_\mathcal{T}\right\| \quad\quad\quad\quad\quad\quad\quad\quad\quad\quad\quad\quad \nonumber
    \end{align}
    \begin{align}
        & \leq  C'_R \sqrt{\frac{\mu_0 r n_1 \log n_1}{m+qn_1n_2}} + \sqrt{\left(\frac{4K}{\log 2}\right)\frac{\beta r \mu_0 n_1 \log n_1}{m+qn_1n_2}} \nonumber \\ 
        & \leq C_R\sqrt{\frac{\beta r \mu_0 n_1 \log n_1}{m+qn_1n_2}} \label{eq:Hnorm_2terms1_final}
    \end{align}
    with probability $1-3n_1^{-\beta}$ for $\beta>1$ and setting $C_R := C'_R + \sqrt{\frac{4K}{\log 2}}$.}
    
    \fixes{Now we move on to analyze the second term of \eqref{eq:Hnorm_2terms} with \eqref{eq:Hnorm_2terms2} and the supremum on $\mathbf{X}_1:\|\mathbf{X}_1\|_F\leq 1$ and $\mathbf{X}_2:\|\mathbf{X}_2\|_F\leq 1$.
    \begin{align}
        \frac{q}{p+q}\left\|\mathcal{P}_\mathcal{T}\mathcal{P}_{\mathcal{Q}^\perp}\mathcal{P}_\mathcal{T}\right\| \quad\quad\quad\quad\quad\quad\quad\quad\quad\quad\quad\quad \nonumber
    \end{align}
    \begin{align}
         & =\sup \left\langle \mathbf{X}_1, \sum_{i=1}^{n_1}\sum_{j=1}^{n_2}\langle \mathbf{X}_2,\mathcal{P}_\mathcal{T}\mathcal{P}_{\mathcal{Q}^\perp}(\mathbf{e}_i\mathbf{e}_j^\intercal)\rangle \mathcal{P}_\mathcal{T}\mathcal{P}_{\mathcal{Q}^\perp}(\mathbf{e}_i\mathbf{e}_j^\intercal) \right\rangle  \nonumber \\
         & \quad \cdot \frac{q}{p+q} \nonumber \\
         & =\sup \frac{q}{p+q} \sum_{i=1}^{n_1}\sum_{j=1}^{n_2}   \langle \mathbf{X}_1, \mathcal{P}_\mathcal{T}\mathcal{P}_{\mathcal{Q}^\perp}(\mathbf{e}_i\mathbf{e}_j^\intercal)\rangle \langle \mathbf{X}_2,\mathcal{P}_\mathcal{T}\mathcal{P}_{\mathcal{Q}^\perp}(\mathbf{e}_i\mathbf{e}_j^\intercal)\rangle \nonumber \\
         & \leq \sup \frac{q}{p+q} \sum_{i=1}^{n_1}\sum_{j=1}^{n_2}   \|\mathcal{P}_\mathcal{T}\mathcal{P}_{\mathcal{Q}^\perp}(\mathbf{e}_i\mathbf{e}_j^\intercal)\|_F^2 \|\mathbf{X}_1\|_F \|\mathbf{X}_2\|_F \nonumber \\
         & \leq \frac{q}{p+q} \sum_{i=1}^{n_1}\sum_{j=1}^{n_2}   \|\mathcal{P}_\mathcal{T}\mathcal{P}_{\mathcal{Q}^\perp}(\mathbf{e}_i\mathbf{e}_j^\intercal)\|_F^2 \nonumber \\
         & = \frac{q}{p+q} \mu_{\mathcal{Q}^\perp} \sum_{i=1}^{n_1}\sum_{j=1}^{n_2}   \|\mathcal{P}_\mathcal{T}(\mathbf{e}_i\mathbf{e}_j^\intercal)\|_F^2 \nonumber \\
         & \leq \frac{q}{p+q} \mu_{\mathcal{Q}^\perp}\mu_0n_1n_2 \nonumber \\
         & = \frac{qn_1n_2}{m+qn_1n_2} \mu_{\mathcal{Q}^\perp}\mu_0n_1n_2 \nonumber \\
         & \leq \sqrt{\frac{\mu_{\mathcal{Q}^\perp}\mu_0qn_1^2n_2^2}{m+qn_1n_2}} \label{eq:Hnorm_2terms2_final}
    \end{align}
    where the third equality comes from the definition of $\mu_{\mathcal{Q}^\perp}$ in \eqref{eq:muQperp_defn}, and the last inequality is true because of the condition in the lemma statement that $\frac{\mu_{\mathcal{Q}^\perp}\mu_0qn_1^2n_2^2}{m+qn_1n_2}\leq 1$.}
    
    \fixes{Finally, plugging in \eqref{eq:Hnorm_2terms1_final} and \eqref{eq:Hnorm_2terms2_final} into \eqref{eq:Hnorm_2terms} gets the resultant.}
\end{proof}


\begin{lemma}
    \label{thm:dualnorm_nuc_spec}
    For each pair $\mathbf{W}$ and $\mathbf{H}$, we have $\langle \mathbf{W},\mathbf{H} \rangle \leq \|\mathbf{W}\| \|\mathbf{H}\|_*$.
    In particular, for each $\mathbf{H}$, there is a $\mathbf{W}$ such that $\|\mathbf{W}\|=1$ where it achieves the equality.
\end{lemma}
This comes directly from the fact that the spectral norm and nuclear norm are dual to each other which is proved in \cite{recht2010guaranteed}.


\begin{lemma}
    \label{thm:khintchine}
    \fixes{(Noncommutative Khintchine inequality) Suppose $\{\mathbf{X}_1,\dots,\mathbf{X}_r\}\in\mathbb{R}^{n_1\times n_2}$ is a finite sequence of matrices, $\{\epsilon_1,\dots,\epsilon_r\}$ is a Rademacher sequence, and $b\geq 2$, then
    \begin{align}
        \left(\mathbb{E}_\epsilon\left(\left\|\sum_{k=1}^r\epsilon_k\mathbf{X}_k\right\|_{S_b}^b\right)\right)^{\frac{1}{b}} \quad\quad\quad\quad\quad\quad\quad\quad\quad\quad\quad\quad \nonumber
    \end{align}
    \begin{align}
        \leq C_K \sqrt{b} \max \left\{ \left\|\left(\sum_{k=1}^r\mathbf{X}_r^\intercal\mathbf{X}_r\right)^{\frac{1}{2}}\right\|_{S_b}, \left\|\left(\sum_{k=1}^r\mathbf{X}_r\mathbf{X}_r^\intercal\right)^{\frac{1}{2}}\right\|_{S_b} \right\} \nonumber
    \end{align}
    where $C_K=2^{-\frac{1}{4}}\sqrt{\frac{\pi}{e}}$.}
\end{lemma}

\fixes{The proof can be found in \cite{buchholz2001operator}.}


\begin{lemma}
    \label{thm:deltas}
    \fixes{Suppose that $b$ is an integer such that $1\leq b \leq n_2 p$ and $n_2p\geq 2 \log n_1$, then
    \begin{align}
        \mathbb{E}_\delta \max \left\{\left\| \sum_{i=1}^{n_1}\sum_{j=1}^{n_2}\delta_{ij}E_{ij}^2\mathbf{e}_i\mathbf{e}_i^\intercal \right\|^b , \left\| \sum_{i=1}^{n_1}\sum_{j=1}^{n_2}\delta_{ij}E_{ij}^2\mathbf{e}_j\mathbf{e}_j^\intercal \right\|^b  \right\} \nonumber \\
        \leq 4 (2 n_1 p \|\mathbf{E}\|_\infty^2)^b. \nonumber
    \end{align}}
\end{lemma}

\fixes{The proof can be found in \cite{candes2009exact}.}

\subsubsection{Lemma \ref{thm:nuc_norm_opt}}
\label{sec:proof_nuc_norm_opt}

\begin{proof}
    We start by perturbing $\mathbf{X}_0$ by any $\mathbf{H}$ so that $\mathbf{X}_0+\mathbf{H}$ remains feasible to Problem \eqref{eq:nuc_norm_prob}.
    Since $\mathbf{X}_0$ is already feasible, then it must be that $\mathcal{R}_\Omega(\mathbf{H})=\mathbf{0}$ and $\langle \mathbf{A}^{(l)}, \mathbf{H} \rangle = 0$ for all $l\in\{1,\dots,\fixes{h}\}$.
    
    Take any subgradient $\mathbf{Y}^0$ of the nuclear norm at $\mathbf{X}_0$ which also satisfies
    \begin{subequations}
        \begin{align}
            \mathcal{P}_\mathcal{T}(\mathbf{Y}^0) & = \sum_{k=1}^{r}\mathbf{u}_k\mathbf{v}_k^\intercal \\
            \left\|\mathcal{P}_{\mathcal{T}^\perp}(\mathbf{Y}^0)\right\| & \leq 1.
        \end{align}
    \end{subequations}
    Let us also define $\mathbf{W}^0:=\mathcal{P}_{\mathcal{T}^\perp}(\mathbf{Y}^0)$ and $\mathbf{W}:=\mathcal{P}_{\mathcal{T}^\perp}(\mathbf{Y})$ for $\mathbf{Y}$ in the lemma's statement.
    This allows us to write $\mathbf{Y}^0=\mathcal{P}_\mathcal{T}(\mathbf{Y}^0)+\mathcal{P}_{\mathcal{T}^\perp}(\mathbf{Y}^0)=\mathcal{P}_\mathcal{T}(\mathbf{Y}^0)+\mathbf{W}^0$ and $\mathbf{Y}=\mathcal{P}_\mathcal{T}(\mathbf{Y})+\mathcal{P}_{\mathcal{T}^\perp}(\mathbf{Y})=\mathcal{P}_\mathcal{T}(\mathbf{Y})+\mathbf{W}$.
    Since $\mathcal{P}_\mathcal{T}(\mathbf{Y}^0)=\mathcal{P}_\mathcal{T}(\mathbf{Y})$, then we have
    \begin{align}
        \mathbf{Y}^0 = \mathbf{W}^0 - \mathbf{W} + \mathbf{Y}. \nonumber
    \end{align}
    
    Starting from the definition of the subgradient we have
    \begin{align}
        \|\mathbf{X}_0+\mathbf{H}\|_* & \geq \|\mathbf{X}_0\|_* + \langle\mathbf{Y}^0,\mathbf{H}\rangle \nonumber \\
        & = \|\mathbf{X}_0\|_* + \langle\mathbf{W}^0 - \mathbf{W} + \mathbf{Y},\mathbf{H}\rangle \nonumber \\
        & = \|\mathbf{X}_0\|_* + \langle\mathbf{W}^0 - \mathbf{W},\mathbf{H}\rangle \nonumber \\
        & \quad + \left\langle \mathcal{R}_\Omega^\intercal\boldsymbol{\lambda}+\sum_{l=1}^{\fixes{h}} \gamma_l\mathbf{A}^{(l)}, \mathbf{H} \right\rangle \nonumber \\
        & = \|\mathbf{X}_0\|_* + \langle\mathbf{W}^0 - \mathbf{W},\mathbf{H}\rangle \nonumber \\
        & \quad + \left\langle \mathcal{R}_\Omega^\intercal\boldsymbol{\lambda}, \mathbf{H} \right\rangle+\sum_{l=1}^{\fixes{h}} \gamma_l\left\langle\mathbf{A}^{(l)}, \mathbf{H} \right\rangle \nonumber \\
        & = \|\mathbf{X}_0\|_* + \langle\mathbf{W}^0 - \mathbf{W},\mathbf{H}\rangle \label{eq:nuc_norm_opt1}
    \end{align}
    The first equality comes from applying the previous equation and the second comes from applying the definition of $\mathbf{Y}$ given in the lemma's statement.
    The last equality comes from the fact that $\mathcal{R}_\Omega(\mathbf{H})=\mathbf{0}$ and $\langle \mathbf{A}^{(l)}, \mathbf{H} \rangle = 0$ for all $l\in\{1,\dots,\fixes{h}\}$.
    
    Let $\mathbf{Z}$ be any matrix that satisfies the following conditions
    \begin{subequations}
        \begin{align}
            \mathcal{P}_{\mathcal{T}^\perp}(\mathbf{Z}) & = \mathbf{W}^0 \label{eq:nuc_norm_opt2a} \\
            \|\mathbf{Z}\| & \leq 1 \label{eq:nuc_norm_opt2b} \\
            \langle \mathbf{Z}, \mathcal{P}_{\mathcal{T}^\perp}(\mathbf{H})\rangle & = \|\mathcal{P}_{\mathcal{T}^\perp}(\mathbf{H})\|_*. \label{eq:nuc_norm_opt2c}
        \end{align}
    \end{subequations}
    Since $\mathbf{W}^0$ and $\mathbf{W}$ are in $\mathcal{T}^\perp$ and the projection operator is self-adjoint, we have
    \begin{align}
        \langle \mathbf{W}^0 - \mathbf{W}, \mathbf{H}\rangle & = \langle \mathcal{P}_{\mathcal{T}^\perp}(\mathbf{W}^0 - \mathbf{W}), \mathbf{H}\rangle \nonumber \\
        & = \langle \mathbf{W}^0 - \mathbf{W}, \mathcal{P}_{\mathcal{T}^\perp}(\mathbf{H})\rangle \nonumber \\
        & = \langle \mathbf{W}^0, \mathcal{P}_{\mathcal{T}^\perp}(\mathbf{H})\rangle - \langle \mathbf{W}, \mathcal{P}_{\mathcal{T}^\perp}(\mathbf{H})\rangle \nonumber \\
        & \geq \langle \mathbf{W}^0, \mathcal{P}_{\mathcal{T}^\perp}(\mathbf{H})\rangle - \|\mathbf{W}\| \|\mathcal{P}_{\mathcal{T}^\perp}(\mathbf{H})\|_* \nonumber \\
        & = \langle \mathcal{P}_{\mathcal{T}^\perp}(\mathbf{Z}), \mathcal{P}_{\mathcal{T}^\perp}(\mathbf{H})\rangle - \|\mathbf{W}\| \|\mathcal{P}_{\mathcal{T}^\perp}(\mathbf{H})\|_* \nonumber \\
        & = \langle \mathbf{Z}, \mathcal{P}_{\mathcal{T}^\perp}(\mathbf{H})\rangle - \|\mathbf{W}\| \|\mathcal{P}_{\mathcal{T}^\perp}(\mathbf{H})\|_* \nonumber \\
        & = \|\mathcal{P}_{\mathcal{T}^\perp}(\mathbf{H})\|_* - \|\mathbf{W}\| \|\mathcal{P}_{\mathcal{T}^\perp}(\mathbf{H})\|_* \label{eq:nuc_norm_opt3}
    \end{align}
    The first inequality applies Lemma \ref{thm:dualnorm_nuc_spec} to the last term.
    The fourth equality applies \eqref{eq:nuc_norm_opt2a} and the last applies \eqref{eq:nuc_norm_opt2c}.
    
    Finally, plugging \eqref{eq:nuc_norm_opt3} into \eqref{eq:nuc_norm_opt1} gives the following
    \begin{align}
        \|\mathbf{X}_0+\mathbf{H}\|_* \geq \|\mathbf{X}_0\|_* + \left(1 - \|\mathbf{W}\|\right)\|\mathcal{P}_{\mathcal{T}^\perp}(\mathbf{H})\|_*.
    \end{align}
    Since $\|\mathbf{W}\|<1$ from the lemma's statement, then $\|\mathbf{X}_0+\mathbf{H}\|_* > \|\mathbf{X}_0\|_*$ unless $\|\mathcal{P}_{\mathcal{T}^\perp}(\mathbf{H})\|_*=0$.
    This results in $\mathcal{P}_{\mathcal{T}^\perp}(\mathbf{H})=\mathbf{0}$ which means that $\mathbf{H}$ is in $\mathcal{T}$.
    Since $\mathcal{R}_\Omega(\mathbf{H})=\mathbf{0}$, then $\mathbf{H}=\mathbf{0}$ by the given injectivity condition \fixes{along with the fact that $\mathbf{0}\in\mathcal{T}$.}
    Therefore, $\|\mathbf{X}_0+\mathbf{H}\|_* > \|\mathbf{X}_0\|_*$ unless $\mathbf{H}=\mathbf{0}$ which makes $\mathbf{X}_0$ the unique optimal feasible solution to Problem \eqref{eq:nuc_norm_prob}.
\end{proof}

\subsubsection{Lemma \ref{thm:inverse_condition}}
\label{sec:proof_inverse_condition}

\begin{proof}
    \fixes{From summing the conditions in the lemma's statement and Lemma \ref{thm:bound_opnormH}, we have that:
    \begin{align}
        \frac{1}{2} & \geq \frac{1}{p\fixes{+q}}\left\|\mathcal{P}_\mathcal{T}\left(\mathcal{P}_\Omega\fixes{+q\mathcal{P}_\mathcal{Q}}\right)\mathcal{P}_\mathcal{T}-(p\fixes{+q})\mathcal{P}_\mathcal{T}\right\| \nonumber \\
        & \geq \frac{1}{p\fixes{+q}}\frac{\left\|\mathcal{P}_\mathcal{T}\left(\mathcal{P}_\Omega\fixes{+q\mathcal{P}_\mathcal{Q}}\right)\mathcal{P}_\mathcal{T}(\mathbf{X})-(p\fixes{+q})\mathcal{P}_\mathcal{T}(\mathbf{X})\right\|_F}{\|\mathcal{P}_\mathcal{T}(\mathbf{X})\|_F} \nonumber
    \end{align}
    which rearranged is
    \begin{align}
        \left\|\mathcal{P}_\mathcal{T}\left(\mathcal{P}_\Omega\fixes{+q\mathcal{P}_\mathcal{Q}}\right)\mathcal{P}_\mathcal{T}(\mathbf{X})-(p\fixes{+q})\mathcal{P}_\mathcal{T}(\mathbf{X})\right\|_F \nonumber \\
        \leq \frac{1}{2}(p+q)\|\mathcal{P}_\mathcal{T}(\mathbf{X})\|_F. \label{eq:inverse_condition1}
    \end{align}
    By the reverse triangle inequality it becomes:
    \begin{align}
        \left\|\mathcal{P}_\mathcal{T}\left(\mathcal{P}_\Omega+q\mathcal{P}_\mathcal{Q}\right)\mathcal{P}_\mathcal{T}(\mathbf{X})\right\|_F - (p+q)\left\|\mathcal{P}_\mathcal{T}(\mathbf{X})\right\|_F \nonumber \\
        \leq \frac{1}{2}(p+q)\|\mathcal{P}_\mathcal{T}(\mathbf{X})\|_F \nonumber
    \end{align} which reduces to \eqref{eq:inverse_condition_b}.
    By taking the reverse triangle inequality of \eqref{eq:inverse_condition1} in the other direction, we have:
    \begin{align}
        (p+q)\left\|\mathcal{P}_\mathcal{T}(\mathbf{X})\right\|_F - \left\|\mathcal{P}_\mathcal{T}\left(\mathcal{P}_\Omega+q\mathcal{P}_\mathcal{Q}\right)\mathcal{P}_\mathcal{T}(\mathbf{X})\right\|_F \nonumber \\
        \leq \frac{1}{2}(p+q)\|\mathcal{P}_\mathcal{T}(\mathbf{X})\|_F \nonumber
    \end{align} which reduces to \eqref{eq:inverse_condition_a}.}
    
    \fixes{From the first condition in the lemma's statement and \eqref{eq:Hnorm_2terms1_final}, we have that:
    \begin{align}
        \frac{1}{4} & \geq \frac{1}{p+q}\left\|\mathcal{P}_\mathcal{T}\left(\mathcal{P}_\Omega-p\mathbf{I}\right)\mathcal{P}_\mathcal{T}\right\| \nonumber \\
        & \geq \frac{1}{p+q} \frac{\left\|\mathcal{P}_\mathcal{T}\left(\mathcal{P}_\Omega-p\mathbf{I}\right)\mathcal{P}_\mathcal{T}(\mathbf{X})\right\|_F}{\|\mathcal{P}_\mathcal{T}(\mathbf{X})\|_F} \nonumber
    \end{align}
    which rearranged is
    \begin{align}
         \frac{1}{4}(p+q)\|\mathcal{P}_\mathcal{T}(\mathbf{X})\|_F & \geq \left\|\mathcal{P}_\mathcal{T}\left(\mathcal{P}_\Omega-p\mathbf{I}\right)\mathcal{P}_\mathcal{T}(\mathbf{X})\right\|_F \nonumber \\
         & \geq \left\|\mathcal{P}_\mathcal{T}\mathcal{P}_\Omega\mathcal{P}_\mathcal{T}(\mathbf{X})\right\|_F - p\|\mathcal{P}_\mathcal{T}(\mathbf{X})\|_F \nonumber
    \end{align}
    where the second inequality is from the reverse triangle inequality.
    From there, it reduces to:
    \begin{align}
        \left\|\mathcal{P}_\mathcal{T}\mathcal{P}_\Omega\mathcal{P}_\mathcal{T}(\mathbf{X})\right\|_F \leq \frac{1}{4}(5p+q)\|\mathcal{P}_\mathcal{T}(\mathbf{X})\|_F. \label{eq:inverse_condition2}
    \end{align}
    Also, we have:
    \begin{align}
        \|\mathcal{P}_\Omega\mathcal{P}_\mathcal{T}(\mathbf{X})\|_F^2 & = \langle \mathcal{P}_\Omega\mathcal{P}_\mathcal{T}(\mathbf{X}),\mathcal{P}_\Omega\mathcal{P}_\mathcal{T}(\mathbf{X}) \rangle \nonumber \\
        & = \langle \mathcal{P}_\mathcal{T}(\mathbf{X}),\mathcal{P}_\Omega\mathcal{P}_\mathcal{T}(\mathbf{X}) \rangle \nonumber \\
        & = \langle \mathcal{P}_\mathcal{T}(\mathbf{X}),\mathcal{P}_\mathcal{T}\mathcal{P}_\Omega\mathcal{P}_\mathcal{T}(\mathbf{X}) \rangle \nonumber \\
        & \leq \|\mathcal{P}_\mathcal{T}(\mathbf{X})\|_F \|\mathcal{P}_\mathcal{T}\mathcal{P}_\Omega\mathcal{P}_\mathcal{T}(\mathbf{X})\|_F \nonumber \\
        & \leq \frac{5p+q}{4} \|\mathcal{P}_\mathcal{T}(\mathbf{X})\|_F^2 \label{eq:inverse_condition3}
    \end{align}
    where the last inequality came from applying \eqref{eq:inverse_condition2}.}
    
    \fixes{From the second condition in the lemma's statement and \eqref{eq:Hnorm_2terms2_final}, we have that:
    \begin{align}
        \frac{1}{4} & \geq \frac{q}{p+q}\left\|\mathcal{P}_\mathcal{T}(\mathcal{P}_\mathcal{Q}-\mathbf{I})\mathcal{P}_\mathcal{T}\right\| \nonumber \\
        & \geq \frac{q}{p+q} \frac{\left\|\mathcal{P}_\mathcal{T}(\mathcal{P}_\mathcal{Q}-\mathbf{I})\mathcal{P}_\mathcal{T}(\mathbf{X})\right\|_F}{\|\mathcal{P}_\mathcal{T}(\mathbf{X})\|_F} \nonumber
    \end{align}
    which rearranged is
    \begin{align}
         \frac{1}{4}(p+q)\|\mathcal{P}_\mathcal{T}(\mathbf{X})\|_F & \geq q\left\|\mathcal{P}_\mathcal{T}(\mathcal{P}_\mathcal{Q}-\mathbf{I})\mathcal{P}_\mathcal{T}(\mathbf{X})\right\|_F. \nonumber \\
         & \geq q\left\|\mathcal{P}_\mathcal{T}\mathcal{P}_\mathcal{Q}\mathcal{P}_\mathcal{T}(\mathbf{X})\right\|_F - q\|\mathcal{P}_\mathcal{T}(\mathbf{X})\|_F \nonumber
    \end{align}
    where the second inequality is from the reverse triangle inequality.
    From there, it reduces to:
    \begin{align}
        \left\|\mathcal{P}_\mathcal{T}\mathcal{P}_\mathcal{Q}\mathcal{P}_\mathcal{T}(\mathbf{X})\right\|_F \leq \frac{1}{4q}(p+5q)\|\mathcal{P}_\mathcal{T}(\mathbf{X})\|_F. \label{eq:inverse_condition4}
    \end{align}
    Also, we have:
    \begin{align}
        \|q\mathcal{P}_\mathcal{Q}\mathcal{P}_\mathcal{T}(\mathbf{X})\|_F^2 & = q^2\langle \mathcal{P}_\mathcal{Q}\mathcal{P}_\mathcal{T}(\mathbf{X}),\mathcal{P}_\mathcal{Q}\mathcal{P}_\mathcal{T}(\mathbf{X}) \rangle \nonumber \\
        & = q^2\langle \mathcal{P}_\mathcal{T}(\mathbf{X}),\mathcal{P}_\mathcal{Q}\mathcal{P}_\mathcal{T}(\mathbf{X}) \rangle \nonumber \\
        & = q^2\langle \mathcal{P}_\mathcal{T}(\mathbf{X}),\mathcal{P}_\mathcal{T}\mathcal{P}_\mathcal{Q}\mathcal{P}_\mathcal{T}(\mathbf{X}) \rangle \nonumber \\
        & \leq q^2 \|\mathcal{P}_\mathcal{T}(\mathbf{X})\|_F \|\mathcal{P}_\mathcal{T}\mathcal{P}_\mathcal{Q}\mathcal{P}_\mathcal{T}(\mathbf{X})\|_F \nonumber \\
        & \leq \frac{q(p+5q)}{4} \|\mathcal{P}_\mathcal{T}(\mathbf{X})\|_F^2 \label{eq:inverse_condition5}
    \end{align}
    where the last inequality came from applying \eqref{eq:inverse_condition4}.}
    
    \fixes{Finally, taking the square root of \eqref{eq:inverse_condition3} and \eqref{eq:inverse_condition5} allows the following:
    \begin{align}
        \|(\mathcal{P}_\Omega+q\mathcal{P}_\mathcal{Q})\mathcal{P}_\mathcal{T}(\mathbf{X})\|_F \quad\quad\quad\quad\quad\quad\quad\quad\quad\quad\quad\quad\quad\quad \nonumber
    \end{align}
    \begin{align}
        & \leq \|\mathcal{P}_\Omega\mathcal{P}_\mathcal{T}(\mathbf{X})\|_F +  \|q\mathcal{P}_\mathcal{Q}\mathcal{P}_\mathcal{T}(\mathbf{X})\|_F \nonumber \\
        & \leq \sqrt{\frac{5p+q}{4}} \|\mathcal{P}_\mathcal{T}(\mathbf{X})\|_F + \sqrt{\frac{q(p+5q)}{4}} \|\mathcal{P}_\mathcal{T}(\mathbf{X})\|_F \nonumber \\
        & \leq \frac{1+\sqrt{q}}{2}\sqrt{5(p+q)}\|\mathcal{P}_\mathcal{T}(\mathbf{X})\|_F \nonumber
    \end{align}
    which is exactly \eqref{eq:inverse_condition_c}.}
\end{proof}

\subsubsection{Lemma \ref{thm:bound_Hsum}}
\label{sec:proof_bound_Hsum}

\begin{proof}
    Starting with the LHS of \eqref{eq:spectral_norm_sum}, we have the following
    \begin{align}
        \frac{1}{p\fixes{+q}}\left\|\mathcal{P}_{\mathcal{T}^\perp}(\mathcal{P}_\Omega\fixes{+q\mathcal{P}_\mathcal{Q}})\mathcal{P}_\mathcal{T}\sum_{k=k_0}^\infty \mathcal{H}^k(\fixes{\mathbf{E}})\right\| \quad\quad\quad\quad\quad\quad \nonumber
    \end{align}
    \begin{align}
        & \leq \frac{1}{p\fixes{+q}}\left\|(\mathcal{P}_\Omega\fixes{+q\mathcal{P}_\mathcal{Q}})\mathcal{P}_\mathcal{T}\sum_{k=k_0}^\infty \mathcal{H}^k(\fixes{\mathbf{E}})\right\| \nonumber \\
        & \leq \frac{1}{p\fixes{+q}}\left\|(\mathcal{P}_\Omega\fixes{+q\mathcal{P}_\mathcal{Q}})\mathcal{P}_\mathcal{T}\sum_{k=k_0}^\infty \mathcal{H}^k(\fixes{\mathbf{E}})\right\|_F \nonumber \\
        & \leq \fixes{(1+\sqrt{q})} \sqrt{\frac{\fixes{5}}{2(p\fixes{+q})}}\left\|\mathcal{P}_\mathcal{T}\sum_{k=k_0}^\infty \mathcal{H}^k(\fixes{\mathbf{E}})\right\|_F \nonumber \\
        & = \fixes{(1+\sqrt{q})} \sqrt{\frac{\fixes{5}n_1n_2}{2(m\fixes{+qn_1n_2})}}\left\|\sum_{k=k_0}^\infty \mathcal{H}^k(\fixes{\mathbf{E}})\right\|_F. \label{eq:bound_Hsum_1}
    \end{align}
    The first inequality comes from the fact that the projection operator is a non-expansive mapping and the second comes from the Frobenius norm being at least as large as the spectral norm.
    The third comes from applying \fixes{the third inequality} of Lemma \ref{thm:inverse_condition}.
    The equality comes from the fact that $\mathcal{H}$ in \eqref{eq:H_defn} is already projected into $\mathcal{T}$ and $p=\frac{m}{n_1n_2}$.
    
    From here, we focus on bounding the Frobenius norm on the RHS of \eqref{eq:bound_Hsum_1}
    \begin{align}
        \left\|\sum_{k=k_0}^\infty \mathcal{H}^k(\fixes{\mathbf{E}})\right\|_F \quad\quad\quad\quad\quad\quad\quad\quad\quad\quad\quad\quad\quad\quad\quad\quad\quad\quad \nonumber 
    \end{align}
    \begin{align}
        & \leq \sum_{k=k_0}^\infty \left\|\mathcal{H}^k(\fixes{\mathbf{E}})\right\|_F \nonumber \\
        & \leq \left\|\fixes{\mathbf{E}}\right\|_F\sum_{k=k_0}^\infty \|\mathcal{H}\|^k \nonumber \\
        & = \|\fixes{\mathbf{E}}\|_F\frac{\|\mathcal{H}\|^{k_0}}{1-\|\mathcal{H}\|} \nonumber \\
        & \leq 2\|\fixes{\mathbf{E}}\|_F\|\mathcal{H}\|^{k_0} \nonumber \\
        & = \fixes{2\sqrt{r}\|\mathcal{H}\|^{k_0}} \nonumber \\
        & \leq 2\fixes{\sqrt{r}} \left(\frac{C_R \sqrt{\mu_0\beta rn_1\log n_1}\fixes{+n_1n_2\sqrt{\mu_{\mathcal{Q}^\perp}\mu_0q}}}{\sqrt{m\fixes{+qn_1n_2}}}\right)^{k_0} \label{eq:bound_Hsum_2}
    \end{align}
    The first inequality comes from the triangle inequality.
    The second inequality comes from one of the equivalent definitions of the operator norm under the Frobenius norm.
    The first equality comes from the geometric series as long as $\|\mathcal{H}\|<1$.
    The third inequality is true as long as $\|\mathcal{H}\|\leq\frac{1}{2}$ \fixes{which is true from summing the conditions in the lemma's statement and the bound resulting from} Lemma \ref{thm:bound_opnormH}.
    \fixes{The second equality comes from the fact that $\|\mathbf{E}\|_F=\sqrt{r}$.}
    The fourth inequality comes from applying Lemma \ref{thm:bound_opnormH} to $\|\mathcal{H}\|$.
    
    Finally, plugging \eqref{eq:bound_Hsum_2} into \eqref{eq:bound_Hsum_1} gets the resultant with the same probability as Lemmas \ref{thm:inverse_condition} and \ref{thm:bound_opnormH}.
\end{proof}

\subsubsection{Lemma \ref{thm:bound_H0}}
\label{sec:proof_bound_H0}

\begin{proof}
    \fixes{We start by splitting the LHS of \eqref{eq:spectral_norm_H0} into two terms and analyzing each one separately
    \begin{align}
        \frac{1}{p\fixes{+q}}\left\|\mathcal{P}_{\mathcal{T}^\perp}(\mathcal{P}_\Omega\fixes{+q\mathcal{P}_\mathcal{Q}})\mathcal{P}_\mathcal{T}(\mathbf{E})\right\| \quad\quad\quad\quad\quad\quad\quad\quad\quad \nonumber
    \end{align}
    \begin{align}
        & = \frac{1}{p+q}\left\|\mathcal{P}_{\mathcal{T}^\perp}(\mathcal{P}_\Omega\fixes{+q\mathcal{P}_\mathcal{Q}}-(p+q)\mathbf{I})\mathcal{P}_\mathcal{T}(\mathbf{E})\right\| \nonumber \\
        & = \frac{1}{p+q}\left\|\mathcal{P}_{\mathcal{T}^\perp}(\mathcal{P}_\Omega\fixes{+q\mathcal{P}_\mathcal{Q}}-(p+q)\mathbf{I})(\mathbf{E})\right\| \nonumber \\
        & \leq \frac{1}{p+q}\left\|(\mathcal{P}_\Omega\fixes{+q\mathcal{P}_\mathcal{Q}}-(p+q)\mathbf{I})(\mathbf{E})\right\| \nonumber \\
        & \leq \frac{1}{p+q}\left(\left\|(\mathcal{P}_\Omega-p\mathbf{I})(\mathbf{E})\right\|+q\left\|\mathcal{P}_{\mathcal{Q}^\perp}(\mathbf{E})\right\|\right) \label{eq:H0_bound1}
    \end{align}
    where the first equality comes from the fact that $\mathcal{P}_{\mathcal{T}^\perp}\mathcal{P}_\mathcal{T}=\mathbf{0}$ and the second comes from the fact that $\mathbf{E}\in\mathcal{T}$.
    The first inequality comes from the fact that the projection operation is non-expansive.
    The second inequality comes from applying the triangle inequality.}
    
    \fixes{First, we show the proofs of two simple identities that will be used later.
    The first shows that the Schatten b-norm is at least as large as than the spectral norm:
    \begin{align}
        \|\mathbf{X}\| & = \sup_k \{\sigma_k(\mathbf{X})\} \nonumber \\
        & = \left(\sup_k \{(\sigma_k(\mathbf{X}))^b\}\right)^{\frac{1}{b}} \nonumber \\
        & \leq \left(\sum_{k=1}^{n_2}(\sigma_k(\mathbf{X}))^b\right)^{\frac{1}{b}} = \|\mathbf{X}\|_{S_b}. \label{eq:H0_bound2}
    \end{align}
    The second shows that the Schatten b-norm is at most as large as a multiple of the spectral norm if $b\geq \log n_2$:
    \begin{align}
        \|\mathbf{X}\|_{S_b} & = \left(\sum_{k=1}^{n_2}(\sigma_k(\mathbf{X}))^b\right)^{\frac{1}{b}} \nonumber \\
        & \leq \left(n_2\sup_k \{(\sigma_k(\mathbf{X}))^b\}\right)^{\frac{1}{b}} \nonumber \\
        & = n_2^{\frac{1}{b}} \sup_k \{\sigma_k(\mathbf{X})\} \nonumber \\
        & \leq e  \sup_k \{\sigma_k(\mathbf{X})\} = e \|\mathbf{X}\|. \label{eq:H0_bound3}
    \end{align}}
    
    \fixes{Define the inside of the first term of \eqref{eq:H0_bound1} as
    \begin{align}
        \mathbf{S} & :=\frac{1}{p+q}(\mathcal{P}_\Omega-p\mathbf{I})(\mathbf{E}) \nonumber \\
        & = \frac{1}{p+q}\sum_{i=1}^{n_1}\sum_{j=1}^{n_2}(\delta_{ij}-p)E_{ij}\mathbf{e}_i\mathbf{e}_j^\intercal \nonumber
    \end{align}
    where the equality puts it into the Bernoulli sampling model form.
    Let $\mathbf{S}'$ with $\delta'_{ij}$ be an independent copy of $\mathbf{S}$, so that we have:
    \begin{align}
        \mathbf{S}-\mathbf{S}' = \frac{1}{p+q}\sum_{i=1}^{n_1}\sum_{j=1}^{n_2}(\delta_{ij}-\delta'_{ij})E_{ij}\mathbf{e}_i\mathbf{e}_j^\intercal. \nonumber
    \end{align}
    Notice that by symmetry
    \begin{align}
        \mathbf{S}_\epsilon-\mathbf{S}'_\epsilon = \frac{1}{p+q}\sum_{i=1}^{n_1}\sum_{j=1}^{n_2}\epsilon_{ij}(\delta_{ij}-\delta'_{ij})E_{ij}\mathbf{e}_i\mathbf{e}_j^\intercal \nonumber
    \end{align}
    has the same distribution as $\mathbf{S}-\mathbf{S}'$ where $\epsilon_{ij}$ is an independent Rademacher sequence,
    \begin{align}
        \mathbf{S}_\epsilon := \frac{1}{p+q}\sum_{i=1}^{n_1}\sum_{j=1}^{n_2}\epsilon_{ij}\delta_{ij}E_{ij}\mathbf{e}_i\mathbf{e}_j^\intercal, \nonumber
    \end{align}
    and $\mathbf{S}'_\epsilon$ is a copy of $\mathbf{S}_\epsilon$ with independent $\delta'_{ij}$.}
    
    \fixes{By Jensen's inequality and the fact that $\mathbb{E}(\mathbf{S}')=\mathbf{0}$, we have
    \begin{align}
        \mathbb{E}\left(\|\mathbf{S}\|^b\right) & = \mathbb{E}\left(\|\mathbf{S}-\mathbb{E}(\mathbf{S}')\|^b\right) \nonumber \\
        & \leq \mathbb{E}\left(\|\mathbf{S}-\mathbf{S}'\|^b\right) \label{eq:H0_bound4}
    \end{align}
    and by the triangle inequality we have
    \begin{align}
        \left(\mathbb{E}\left(\|\mathbf{S}_\epsilon-\mathbf{S}'_\epsilon\|^b\right)\right)^{\frac{1}{b}} & \leq \left(\mathbb{E}\left(\|\mathbf{S}_\epsilon\|^b\right)\right)^{\frac{1}{b}} + \left(\mathbb{E}\left(\|\mathbf{S}'_\epsilon\|^b\right)\right)^{\frac{1}{b}} \nonumber \\
        & = 2 \left(\mathbb{E}\left(\|\mathbf{S}_\epsilon\|^b\right)\right)^{\frac{1}{b}}. \label{eq:H0_bound5}
    \end{align}
    From the fact that the distributions of $\mathbf{S}-\mathbf{S}'$ and $\mathbf{S}_\epsilon-\mathbf{S}'_\epsilon$ are equivalent, we have
    \begin{align}
        \left(\mathbb{E}\left(\|\mathbf{S}\|^b\right)\right)^{\frac{1}{b}} & \leq \left(\mathbb{E}\left(\|\mathbf{S}-\mathbf{S}'\|^b\right)\right)^{\frac{1}{b}} \nonumber \\
        & = \left(\mathbb{E}\left(\|\mathbf{S}_\epsilon-\mathbf{S}'_\epsilon\|^b\right)\right)^{\frac{1}{b}} \nonumber \\
        & \leq 2 \left(\mathbb{E}\left(\|\mathbf{S}_\epsilon\|^b\right)\right)^{\frac{1}{b}} \nonumber \\
        & \leq 2 \left(\mathbb{E}\left(\|\mathbf{S}_\epsilon\|_{S_b}^b\right)\right)^{\frac{1}{b}} \nonumber \\
        & = 2 \left(\mathbb{E}\left(\|\mathbf{S}_\epsilon\|_{S_b}^b\right)\right)^{\frac{1}{b}\frac{b'}{b}\frac{b}{b'}} \nonumber \\
        & \leq 2 \left(\mathbb{E}\left(\|\mathbf{S}_\epsilon\|_{S_{b'}}^{b'}\right)\right)^{\frac{1}{b'}} \nonumber \\
        & = 2 \left(\mathbb{E}_\delta\mathbb{E}_\epsilon\left(\|\mathbf{S}_\epsilon\|_{S_{b'}}^{b'}\right)\right)^{\frac{1}{b'}}. \label{eq:H0_bound6}
    \end{align}
    The first inequality comes from \eqref{eq:H0_bound4}, the second comes from \eqref{eq:H0_bound5}, and the third comes from \eqref{eq:H0_bound2}.
    The fourth comes from applying Jensen's inequality with the convex function $f(x)=x^{\frac{b'}{b}}$ for $b'\geq b$.}
    
    \fixes{With the use of Lemma \ref{thm:khintchine} and Jensen's inequality with the concave function $f(x)=x^{\frac{1}{b'}}$ for $b'\geq 1$, we have
    \begin{align}
        \left(\mathbb{E}_\delta\mathbb{E}_\epsilon\left(\|\mathbf{S}_\epsilon\|_{S_{b'}}^{b'}\right)\right)^{\frac{1}{b'}} \quad\quad\quad\quad\quad\quad\quad\quad\quad\quad\quad\quad\quad\quad\quad \nonumber
    \end{align}
    \begin{align}
        & \leq C_K  \frac{\sqrt{b'}}{p+q} \Bigg(\mathbb{E}_\delta\max \Bigg\{ \Bigg\|\left(\sum_{i=1}^{n_1}\sum_{j=1}^{n_2}\delta_{ij}E_{ij}^2\mathbf{e}_i\mathbf{e}_i^\intercal\right)^{\frac{1}{2}}\Bigg\|_{S_{b'}}^{b'}, \nonumber \\
        & \quad\quad\quad\quad\quad\quad\quad\quad\quad \Bigg\|\left(\sum_{i=1}^{n_1}\sum_{j=1}^{n_2}\delta_{ij}E_{ij}^2\mathbf{e}_j\mathbf{e}_j^\intercal\right)^{\frac{1}{2}}\Bigg\|_{S_{b'}}^{b'} \Bigg\}\Bigg)^{\frac{1}{b'}} \nonumber \\
        & \leq C_K  \frac{e\sqrt{b'}}{p+q} \Bigg(\mathbb{E}_\delta\max \Bigg\{ \Bigg\|\left(\sum_{i=1}^{n_1}\sum_{j=1}^{n_2}\delta_{ij}E_{ij}^2\mathbf{e}_i\mathbf{e}_i^\intercal\right)^{\frac{1}{2}}\Bigg\|^{b'}, \nonumber \\
        & \quad\quad\quad\quad\quad\quad\quad\quad\quad \Bigg\|\left(\sum_{i=1}^{n_1}\sum_{j=1}^{n_2}\delta_{ij}E_{ij}^2\mathbf{e}_j\mathbf{e}_j^\intercal\right)^{\frac{1}{2}}\Bigg\|^{b'} \Bigg\}\Bigg)^{\frac{1}{b'}} \nonumber \\
        & = C_K  \frac{e\sqrt{b'}}{p+q} \Bigg(\mathbb{E}_\delta\max \Bigg\{ \Bigg\|\sum_{i=1}^{n_1}\sum_{j=1}^{n_2}\delta_{ij}E_{ij}^2\mathbf{e}_i\mathbf{e}_i^\intercal\Bigg\|^{\frac{b'}{2}}, \nonumber \\
        & \quad\quad\quad\quad\quad\quad\quad\quad\quad \Bigg\|\sum_{i=1}^{n_1}\sum_{j=1}^{n_2}\delta_{ij}E_{ij}^2\mathbf{e}_j\mathbf{e}_j^\intercal\Bigg\|^{\frac{b'}{2}} \Bigg\}\Bigg)^{\frac{1}{b'}} \nonumber \\
        & \leq  C_K  \frac{e\sqrt{b'}}{p+q}\left(4 (2 n_1 p \|\mathbf{E}\|_\infty^2)^{\frac{b'}{2}}\right)^{\frac{1}{b'}} \nonumber \\
        & = \sqrt{2} C_K  \frac{e\|\mathbf{E}\|_\infty\sqrt{b'pn_1}}{p+q}4^{\frac{1}{b'}}. \label{eq:H0_bound7}
    \end{align}
    The second inequality comes from \eqref{eq:H0_bound3} assuming that $b'\geq \log n_2$.
    The first equality comes from the fact that matrices resulting from the summations are diagonal matrices.
    The third inequality comes from applying Lemma \ref{thm:deltas} since the assumption of $\max\{2,\beta\}n_1\log n_1 \leq m$ satisfies the condition of that Lemma.}
    
    \fixes{Putting together \eqref{eq:H0_bound6} and \eqref{eq:H0_bound7} and setting $b=b'=\beta \log n_1$ for $\beta\geq 1$ gives
    \begin{align}
        \left(\mathbb{E}\left(\|\mathbf{S}\|^b\right)\right)^{\frac{1}{b}} & \leq 2\sqrt{2} C_K  \frac{e\|\mathbf{E}\|_\infty\sqrt{b'pn_1}}{p+q}4^{\frac{1}{b'}} \nonumber \\
        & = C_K e\|\mathbf{E}\|_\infty \frac{\sqrt{p \beta n_1\log n_1}}{p+q}2^{\frac{2}{\beta \log n_1}+\frac{3}{2}}  \nonumber \\
        & \leq C_K e\|\mathbf{E}\|_\infty \sqrt{\frac{\beta n_1\log n_1}{p+q}}2^{\frac{2}{\beta \log n_1}+\frac{3}{2}} \nonumber \\
        & \leq C_K e \nu_0 \sqrt{\frac{r}{n_1n_2}} \sqrt{\frac{\beta n_1\log n_1}{p+q}}2^{\frac{2}{\beta \log n_1}+\frac{3}{2}} \nonumber \\
        & = C_K e \nu_0 \sqrt{\frac{\beta r n_1\log n_1}{m+qn_1n_2}}2^{\frac{2}{\beta \log n_1}+\frac{3}{2}} \label{eq:H0_bound8}
    \end{align}
    where the second inequality comes from multiplying the RHS by $\sqrt{\frac{p+q}{q}}\geq 1$.
    The third inequality comes from applying Assumption \ref{ass:maxEval}.}
    
    \fixes{From the Markov inequality with $t>0$, we have
    \begin{align}
        \mathbb{P}\left(\|S\|^b\geq t^b \mathbb{E}\left(\|\mathbf{S}\|^b\right) \right) & \leq t^{-b} \nonumber \\
        \mathbb{P}\left(\|S\|\geq t \left(\mathbb{E}\left(\|\mathbf{S}\|^b\right)\right)^{\frac{1}{b}} \right) & \leq t^{-b} \nonumber
    \end{align}
    and then plugging in \eqref{eq:H0_bound8}, $b=\beta \log n_1$, and $t=e$, it becomes
    \begin{align}
        \mathbb{P}\left(\|S\|\geq C_K e^2 \nu_0 \sqrt{\frac{\beta r n_1\log n_1}{m+qn_1n_2}}2^{\frac{2}{\beta \log n_1}+\frac{3}{2}} \right) \leq n_1^{-\beta}. \label{eq:H0_bound9}
    \end{align}}
    
    \fixes{Now we focus on bounding the second term of \eqref{eq:H0_bound1}
    \begin{align}
        \frac{q}{p+q}\left\|\mathcal{P}_{\mathcal{Q}^\perp}(\mathbf{E})\right\| & \leq \frac{q}{p+q}\left\|\mathcal{P}_{\mathcal{Q}^\perp}(\mathbf{E})\right\|_F \nonumber \\
        & = \frac{q\sqrt{\nu_{\mathcal{Q}^\perp}r}}{p+q} \nonumber \\
        & = \frac{qn_1n_2\sqrt{\nu_{\mathcal{Q}^\perp}r}}{m+qn_1n_2} \nonumber \\
        & \leq \sqrt{\frac{qn_1n_2\sqrt{\nu_{\mathcal{Q}^\perp}r}}{m+qn_1n_2}} \label{eq:H0_bound10}
    \end{align}
    where the first equality comes from \eqref{eq:nuQperp_defn} and the second inequality is a result of assuming that $\frac{qn_1n_2\sqrt{\nu_{\mathcal{Q}^\perp}r}}{m+qn_1n_2} \leq 1$}.
    
    \fixes{Finally, putting \eqref{eq:H0_bound9} and \eqref{eq:H0_bound10} into \eqref{eq:H0_bound1}, get the resultant.}
\end{proof}

\subsubsection{Lemma \ref{thm:neighborhood1}}
\label{sec:neighborhood1}

\begin{proof}
    \fixes{This corollary is proved by showing that the right-hand sides of \eqref{eq:thm_samp_complex_a} - \eqref{eq:thm_samp_complex_e} are less than equal to zero under the stated conditions for some sufficiently large value of $q$.}
    
    \fixes{First, we examine the RHS of \eqref{eq:thm_samp_complex_a} by setting $\epsilon_1:=2(\nu_{\mathcal{Q}^\perp}r)^\frac{1}{4}$ which means that $\epsilon_1<1$ from the corollary's condition:
    \begin{align}
        & \left(C_K e^2 \nu_0 \sqrt{\beta r n_1\log n_1}2^{\frac{2}{\beta \log n_1}+\frac{5}{2}}+\epsilon_1\sqrt{qn_1n_2}\right)^2 - qn_1n_2 \nonumber \\
        & = \left(C_K e^2 \nu_0 \sqrt{\beta r n_1\log n_1}2^{\frac{2}{\beta \log n_1}+\frac{5}{2}}+(\epsilon_1-1)\sqrt{qn_1n_2}\right) \nonumber \\
        & \quad \times \left(C_K e^2 \nu_0 \sqrt{\beta r n_1\log n_1}2^{\frac{2}{\beta \log n_1}+\frac{5}{2}}+(\epsilon_1+1)\sqrt{qn_1n_2}\right). \nonumber
    \end{align}
    The above expression becomes non-positive from the first factor if $q$ is:
    \begin{align}
        q \geq \frac{C_K^2 e^4 \nu_0^2 \beta r n_1\log n_12^{\frac{4}{\beta \log n_1}+5}}{n_1n_2(1-\epsilon_1)^2}. \label{eq:neighborhood1_1} 
    \end{align}
    Thus, when the above inequality is true, then the RHS of \eqref{eq:thm_samp_complex_a} is less than or equal to 0.}
    
    \fixes{From the condition in the corollary's statement and setting $\epsilon_2:=\sqrt{10r\mu_{\mathcal{Q}^\perp}\mu_0n_1n_2}$ means that $\epsilon_2<1$ and the RHS of \eqref{eq:thm_samp_complex_b} is
    \begin{align}
        & \left(C_Rr\sqrt{10\beta \mu_0 n_1 \log n_1}+\epsilon_2\sqrt{qn_1n_2}\right) \nonumber \\
        & \quad \times(\sqrt{n_1n_2}+\sqrt{qn_1n_2}) - qn_1n_2 \nonumber \\
        & \quad\quad = C_Rrn_1\sqrt{10\beta \mu_0 n_2 \log n_1} - (1-\epsilon_2)qn_1n_2 \nonumber \\
        & \quad\quad\quad + (C_Rr\sqrt{10\beta \mu_0 n_1 \log n_1}+\epsilon_2\sqrt{n_1n_2})\sqrt{qn_1n_2} \nonumber \\
        & \quad\quad = C_Rrn_1\sqrt{10\beta \mu_0 n_2 \log n_1} \nonumber \\
        & \quad\quad\quad + \frac{(C_Rr\sqrt{10\beta \mu_0 n_1 \log n_1}+\epsilon_2\sqrt{n_1n_2})^2}{4(1-\epsilon_2)^2} - (1-\epsilon_2) \nonumber \\
        & \quad\quad\quad \times\left(\sqrt{qn_1n_2}-\frac{C_Rr\sqrt{10\beta \mu_0 n_1 \log n_1}+\epsilon_2\sqrt{n_1n_2}}{2(1-\epsilon_2)}\right)^2 \nonumber
    \end{align}
    The above expression becomes non-positive if $q$ is:
    \begin{align}
        q & \geq \frac{1}{n_1n_2}\Bigg(\frac{C_Rr\sqrt{10\beta \mu_0 n_1 \log n_1}+\epsilon_2\sqrt{n_1n_2}}{2(1-\epsilon_2)} \nonumber \\
        & \quad + \Bigg(\frac{(C_Rr\sqrt{10\beta \mu_0 n_1 \log n_1}+\epsilon_2\sqrt{n_1n_2})^2}{4(1-\epsilon_2)^3} \nonumber \\
        & \quad\quad + \frac{C_Rrn_1\sqrt{10\beta \mu_0 n_2 \log n_1}}{(1-\epsilon_2)} \Bigg)^{\frac{1}{2}} \Bigg)^2 \label{eq:neighborhood1_2} 
    \end{align}
    Thus, when the above inequality is true, then the RHS of \eqref{eq:thm_samp_complex_b} is less than or equal to 0.}
    
    \fixes{For \eqref{eq:thm_samp_complex_c}, the RHS is less than or equal to zero when
    \begin{align}
         q \geq \frac{2^4C_R^2\beta \mu_0 r \log n_1}{n_2}. \label{eq:neighborhood1_3} 
    \end{align}}
    
    \fixes{With the conditions stated in the corollary, the RHS for each of \eqref{eq:thm_samp_complex_d} and \eqref{eq:thm_samp_complex_e}, it can be directly observed that the right-hand sides are less than zero for any $q>0$.}
    
    \fixes{Thus, by setting $q$ to be the maximum of \eqref{eq:neighborhood1_1}, \eqref{eq:neighborhood1_2}, and \eqref{eq:neighborhood1_3}, all of the inequalities in Theorem \ref{thm:samp_complex} are satisfied for any $m>0$ except \eqref{eq:thm_samp_complex_f} which remains in Corollary \ref{thm:neighborhood1}}.
\end{proof}

\begin{figure*}
    \centering
    \includegraphics[width=0.80\textwidth]{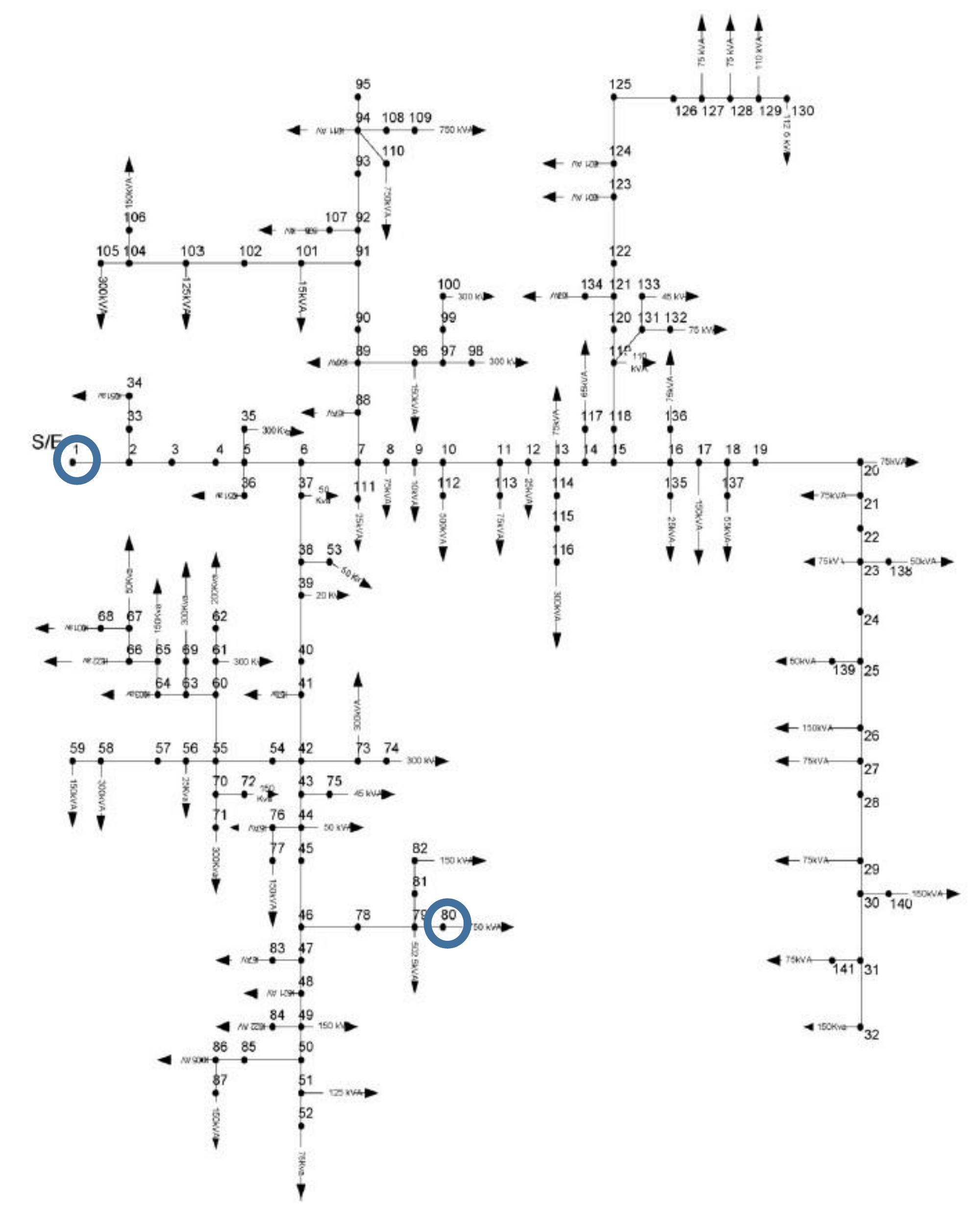}
    \caption{Diagram of the 141 bus distribution network~\cite{khodr2008maximum}.
    The two buses with PMUs are labeled by blue circles.}
    \label{fig:diagram_141bus_PMUs}
\end{figure*}

\subsection{Distribution Network Diagram}
\label{sec:app_dist}

Figure \ref{fig:diagram_141bus_PMUs} gives the diagram of the 141 bus distribution network~\cite{khodr2008maximum} used in the simulations of Section \ref{sec:perf_eval_distnet} with PMU placements denoted by blue circles.

\end{document}